\documentclass[english,notitlepage]{amsart}
\usepackage[latin9]{inputenc}
\usepackage{geometry}
\usepackage{esint}
\usepackage{hyperref}
\usepackage{dsfont}

\usepackage{amssymb,amsmath,amsthm,amscd}
\usepackage{mathrsfs}
\usepackage{tikz}
\usetikzlibrary{positioning}
\usepackage{graphics}
\usetikzlibrary{matrix,arrows}
\usepackage{amssymb} 
\usepackage{babel}
\usepackage{amstext}
\usepackage{amscd}   
\usepackage{epsfig}  
\usepackage{rotating}
\usepackage{psfrag}  
\usepackage{multicol}
\usepackage{multirow}
\usepackage{array}
\usepackage{verbatim}
\usepackage{colortbl}
\usepackage{hhline}
\usepackage{xcolor}
\usepackage[abs]{overpic}
\makeatletter
\numberwithin{equation}{section}
\numberwithin{figure}{section}

\theoremstyle{plain}
\newtheorem{thm}{Theorem}
  \theoremstyle{plain}
  \numberwithin{thm}{section}
  \newtheorem{cor}[thm]{Corollary}
  \theoremstyle{plain}
  \newtheorem{lem}[thm]{Lemma}
  \theoremstyle{remark}
  \newtheorem{rem}[thm]{Remark}
    \theoremstyle{remark}
  \newtheorem{example}[thm]{Example}
    
   \theoremstyle{plain}
  \newtheorem{prop}[thm]{Proposition}
  \def\Ddots{\mathinner{\mkern1mu\raise\p@
\vbox{\kern7\p@\hbox{.}}\mkern2mu
\raise4\p@\hbox{.}\mkern2mu\raise7\p@\hbox{.}\mkern1mu}}
\makeatother

\usepackage{babel}

\newtheorem{definition}{Definition}
\numberwithin{definition}{section}
\newtheorem{assumption}{Assumption}
\newtheorem{result}{Result}

\newcommand{\norm}[1]{\left\| #1 \right\|}
\newcommand{\mklm}[1]{\left\{ #1 \right\}}
\newcommand{\eklm}[1]{\left\langle #1 \right\rangle}

\renewcommand{\d}{\,d}
\newcommand{\N}{{\mathbb N}}
\newcommand{\Z}{{\mathbb Z}}
\newcommand{\C}{{\mathbb C}}

\newcommand{\R}{{\mathbb R}}

\newcommand{\B}{{\mathcal B}}

\newcommand{\E}{{\mathcal E}}

\newcommand{\Jbb}{{\mathbb J }}

\renewcommand{\P}{{\mathcal P}}

\renewcommand{\O}{{\mathcal O}}
\newcommand{\W}{{\mathcal W}}
\newcommand{\Wh}{{{\mathcal W}_h}}
\newcommand{\X}{{\mathfrak X}}
\newcommand{\Xbf}{{\mathbf X}}
\newcommand{\Xbb}{{\mathbb X}}

\newcommand{\1}{{\bf 1}}
\renewcommand{\epsilon}{\varepsilon}
\newcommand{\tphi}{\widetilde{\varphi}}

\renewcommand{\rho}{\varrho}

\newcommand{\bdm}{\begin{displaymath}}
\newcommand{\edm}{\end{displaymath}}
\newcommand{\bq}{\begin{equation}}
\newcommand{\eq}{\end{equation}}
\newcommand{\bqn}{\begin{equation*}}
\newcommand{\eqn}{\end{equation*}}

\newcommand{\Cinft}{{\rm C^\infty }}
\newcommand{\CT}{{\rm C^\infty_c}}

\newcommand{\Sob}{{\rm H}}
\renewcommand{\L}{{\rm L}}
\newcommand{\Lcal}{{\mathcal L}}

\newcommand{\G}{{\mathcal G}}

\newcommand{\SO}{\mathrm{SO}}

\newcommand{\g}{{\bf \mathfrak g}}
\renewcommand{\k}{{\bf \mathfrak k}}

\newcommand{\p}{{\bf \mathfrak p}}

\newcommand{\Ad}{\mathrm{Ad}\,}

\newcommand{\vol}{\text{vol}\,}

\newcommand{\Op}{\mathrm{Op}}

\DeclareMathOperator{\supp}{supp\,}

\DeclareMathOperator{\tr}{tr}

\DeclareMathOperator{\grad}{grad}
\DeclareMathOperator{\sgrad}{s-grad}

\newcommand{\e}[1]{\,{\mathrm e}^{#1}\,}

\begin{document}
\author{Benjamin K\"uster}
\address{Philipps-Universit\"at Marburg, Fachbereich Mathematik und Informatik, Hans-Meerwein-Str., 35032 Marburg, Germany}
\email[Benjamin K\"uster]{\href{mailto:bkuester@mathematik.uni-marburg.de}{bkuester@mathematik.uni-marburg.de}}
\author{Pablo Ramacher}
\email[Pablo Ramacher]{\href{mailto:ramacher@mathematik.uni-marburg.de}{ramacher@mathematik.uni-marburg.de}}
\title[Semiclassical analysis and symmetry reduction II]{Semiclassical analysis and symmetry reduction II. Equivariant quantum ergodicity for invariant Schr\"odinger operators on compact  manifolds} 
\date{\today}

\begin{abstract}
We study  the ergodic properties of Schr\"odinger operators on a compact connected Riemannian manifold $M$ without boundary in case that the underlying Hamiltonian system possesses certain symmetries. More precisely, let  $M$ carry an isometric and effective action of a compact connected Lie group $G$. Relying on an equivariant semiclassical Weyl law proved in Part I of this work, we deduce an equivariant quantum ergodicity theorem under the assumption that the symmetry-reduced Hamiltonian flow on the principal stratum of the singular symplectic reduction of $M$ is ergodic. In particular, we obtain an equivariant version of the Shnirelman-Zelditch-Colin-de-Verdi\`{e}re theorem, as well as a representation theoretic equidistribution theorem. If $M/G$ is an orbifold, similar results were recently obtained by Kordyukov. When $G$ is trivial, one recovers the classical results.
\end{abstract}

\maketitle

\setcounter{tocdepth}{1}
\tableofcontents{}

\section{Introduction}

\subsection{Motivation}

Let $M$ be a compact boundary-less connected Riemannian $\Cinft$-manifold of dimension $n$ with Riemannian volume density $dM$, and denote by $\Delta$ the Laplace-Beltrami operator on $M$. 
One of the central problems in spectral geometry consists in studying the properties of eigenvalues and eigenfunctions of $-\Delta$ in the limit of large eigenvalues. Concretely, let $\mklm{u_j}$ be an orthonormal basis of $\L^2(M)$ of eigenfunctions of $-\Delta$ with respective eigenvalues $\mklm{E_j}$, repeated according to their multiplicity. As $E_j \to \infty$,  one is interested among other things in the asymptotic distribution of eigenvalues,  the pointwise convergence of the $u_j$, bounds of the $\L^p$-norms of the $u_j$ for $1 \leq p \leq \infty$, and the weak convergence of the measures $|u_j|^2 dM$. These questions have been studied extensively over the years, and this paper is the second in a sequel which  addresses them for Schr\"odinger operators in case that the underlying classical system possesses certain symmetries.

In this second part, we shall concentrate  on the ergodic properties of eigenfunctions, while Part I  \cite{kuester-ramacher15a} of this work  dealt with the distribution of eigenvalues. The guiding idea behind  is the correspondence principle of semiclassical physics. To explain this in more detail, consider the unit co-sphere bundle $S^\ast M$, which corresponds to the phase space of a classical free particle moving with constant energy. Each point in $S^\ast M$ represents a state of the classical system, its motion being given by the geodesic flow in $S^\ast M$, and classical observables correspond to functions $a \in \Cinft(S^\ast M)$. On the other hand, by the Kopenhagen interpretation of quantum mechanics, quantum observables correspond to self-adjoint operators $A$ in the Hilbert space $\L^2(M)$. The elements $\psi \in \L^2(M)$ are interpreted as states of the quantum mechanical system, and the expectation value for measuring the property $A$ while the system is in the state $\psi$ is given by $\left<A\psi,\psi\right>_{\L^2(M)}$. The transition between the classical and the quantum-mechanical picture is given by a \emph{quantization map}
\[
 S^k(M) \ni a \quad \longmapsto \quad \Op_\hbar(a), \qquad k \in \R,
\]
where  $\Op_\hbar(a)$  is a pseudodifferential operator in $\L^2(M)$ depending on Planck's constant $\hbar$ and the particular choice of the  map $\Op_\hbar$,  and $S^k(M)\subset \Cinft(T^*M)$ denotes a suitable space of symbol functions. The \emph{correspondence principle} then says that, in the limit of high energies, the quantum mechanical system should behave more and more  like the corresponding classical system. 

The study of the asymptotic distribution of eigenvalues has a  history of more than a hundred years that goes back to work of Weyl \cite{weyl}, Levitan \cite{levitan52}, Avacumovi\v{c} \cite{avacumovic}, and H\"ormander \cite{hoermander68},  the central result being \emph{Weyl's law}, while the behavior of eigenfunctions has been examined more intensively during the last decades. One of the major results  in this direction  is  the \emph{quantum ergodicity theorem} for  chaotic systems, due to  Shnirelman \cite{shnirelman}, Zelditch \cite{zelditch1987}, and Colin de Verdi\`{e}re \cite{colindv}. To explain it, consider   the distributions\footnote{Here one regards $s \in \Cinft(S^\ast M)$ as an element in $S^0(M)\subset \Cinft(T^\ast M)$ by extending it $0$-homogeneously to $T^\ast M$ with the zero-section removed, and then cutting off that extension smoothly near the zero section.}
\[
\mu_j:\Cinft(S^\ast M) \longrightarrow \C, \quad a \longmapsto \eklm{\Op_\hbar(a) u_{j}, u_{j}}_{\L^2(M)}.
\]
If it exists, the distribution limit
$
\mu=\lim_{j \to \infty} \mu_{j}
$ 
constitutes a so-called \emph{quantum limit} for the eigenfunction sequence $\mklm{u_j}$. Furthermore,  the  probability measure on $S^\ast M$ defined by a quantum limit is invariant under the geodesic flow and independent of the choice of $\Op_\hbar$. Since the measure  $\mu$ projects to a weak limit $\bar \mu$ of the measures $\bar \mu_j=|u_j|^2 dM$, it is called a \emph{microlocal lift} of $\bar \mu$, and one can reduce the study of the measures $\bar \mu$ to the classification of quantum limits. The {quantum ergodicity theorem} then  says that if the geodesic flow on $S^\ast M$ is ergodic with respect to the Liouville measure $d(S^\ast M)$, then there exists a subsequence $\mklm{u_{j_k}}_{k\in \N}$ of density $1$  such that the $\mu_{{j_k}}$ converge to $d(S^\ast M)$ as distributions, and consequently the measures $\bar \mu_{j_k}$ converge weakly to $ dM$. Intuitively, the geodesic flow being ergodic means that the geodesics are distributed on $S^*M$ in a sufficiently chaotic way, and this equidistribution of trajectories in the classical system implies asymptotic equidistribution for a density $1$  subsequence of  states of the corresponding quantum system. 

A large class of manifolds whose geodesic flow is ergodic are compact boundary-less manifolds with strictly negative sectional curvature \cite{hopf, brin}, and one of the main conjectures in the field is the {Rudnick-Sarnak conjecture  on quantum unique ergodicity (QUE)} \cite{rudnick} which  says that if $M$ has strictly negative sectional curvature, the whole sequence  $|u_j|^2 dM$ converges weakly to the normalized Riemannian measure $(\vol M)^{-1} dM$ as $j \to \infty$. 
 It has been verified in certain arithmetic situations by Lindenstrauss \cite{lindenstrauss06}, but in general,  the conjecture is still very open. Sequences of eigenfunctions with a quantum limit different from the Liouville measure are called \emph{exceptional subsequences}, and it has been shown by Jacobson and Zelditch \cite{zelditchjacobson}  that any flow-invariant measure on the unit co-sphere bundle of a standard $n$-sphere occurs as a quantum limit for the Laplacian, showing that the family of exceptional subsequences for the Laplacian can be quite large if the geodesic flow fails to be ergodic. However, it was shown by Faure, Nonnenmacher, and de Bi\`{e}vre \cite{nonnenmacher} that ergodicity of the geodesic flow alone is not sufficient to rule out the existence of exceptional subsequences for  particular elliptic operators. Examples of ergodic billiard systems that admit exceptional subsequences of eigenfunctions were  recently found by Hassel \cite{hassel10}. 

\subsection{Problem and setup}
In this article, we will address the problem of  determining quantum limits for sequences of eigenfunctions of  Schr\"odinger operators in case that the underlying classical system possesses certain symmetries. Due to the presence of conserved quantitites, the corresponding Hamiltonian flow will in parts be integrable, and not totally chaotic,
in contrast to the  hitherto examined chaotic systems.

\medskip

\noindent
 The question is then how  the  partially chaotic behavior of the Hamiltonian flow is reflected in the ergodic properties of the eigenfunctions.

\medskip

\noindent
To explain things more precisely, let us recall the setting from Part I. Thus, let $G$ be a compact connected Lie group that acts effectively and isometrically on $M$. Note that  there might be   orbits  of different dimensions, and that the orbit space $\widetilde M:=M/G$  won't be a manifold in general, but a topological quotient space.  If $G$ acts on $M$ with finite isotropy groups, $\widetilde M$ is a compact orbifold, and its  singularities are not too severe. Consider now a \emph{Schr\"odinger operator} on $M$ given by
\[
\breve{P}(h)=-h^2\breve\Delta + V,\qquad \breve{P}(h):\Cinft(M)\to \Cinft(M),\qquad h\in(0,1],
\]
where $\breve \Delta$ denotes the Laplace operator as  differential operator on $M$ with domain $\Cinft(M)$ and $V\in \Cinft(M,\R)$  a $G$-invariant potential. $\breve{P}(h)$ has a unique self-adjoint extension 
\bq
\label{eq:15.08.2015}
P(h):\mathrm{H}^2(M)\to \L^2(M)
\eq
as an unbounded operator in $\L^2(M)$, where $\mathrm{H}^2(M)\subset \L^2(M)$ denotes the second Sobolev space, and one calls $P(h)$ a Schr\"odinger operator, too. For each $h\in(0,1]$, the spectrum of $P(h)$ is discrete, consisting of  eigenvalues $\{E_j(h)\}_{j\in \N}$ which we repeat according to their multiplicity and which form a non-decreasing sequence unbounded towards $+\infty$. Thus, the spectrum of $P(h)$ is bounded from below and its eigenspaces are finite-dimensional. The associated sequence of eigenfunctions $\{u_{j}(h)\}_{j\in\mathbb{N}}$ constitutes a Hilbert basis in $\L^2(M)$, and each eigenfunction $u_{j}(h)$ is smooth.
Now, since  $P(h)$ commutes with the isometric $G$-action, one can use representation theory to study  the eigenfunctions of $P(h)$ in a more detailed way. Indeed, by the Peter-Weyl theorem, the unitary left-regular representation of $G$ 
\begin{equation*}
G\times \L^2(M)  \to  \L^2(M),\qquad (g,f)  \mapsto  \left(L_gf:\, x\mapsto f(g^{-1}\cdot x)\right),
\end{equation*}
has an orthogonal decomposition into isotypic components of the form
\bq
\label{eq:PW} 
\L^2(M)=\bigoplus_{\chi \in \widehat G} \L^2_\chi(M),\qquad \L^2_\chi(M)= T_\chi \, \L^2(M),
\eq
where we wrote   $\widehat{G}$ for the set of equivalence classes of irreducible unitary $G$-representations, and  $T_\chi:\L^2(M)\to \L^2_\chi(M)$ for the associated orthogonal projections.  The \emph{character} belonging to an element $\chi \in \widehat G$ is given by $\chi(g):=\tr \pi_\chi(g)$, where $\pi_\chi$  denotes a representation of class $\chi$. It is also denoted by $\chi$, and the projectors $T_\chi$ are given by the explicit formula
\bq
T_{\chi}: f  \mapsto  \Big(x\mapsto d_{\chi}\int_{G}\overline{\chi(g)}f(g^{-1}\cdot x)\, dg\Big),\label{eq:groupproj}
\eq
where  $dg$ is the normalized Haar measure on $G$ and $d_\chi$ the dimension of $\pi_\chi$. Since each eigenspace of $P(h)$ decomposes into a sum of unitary irreducible $G$-representations, we can study the eigenfunctions of $P(h)$ by considering its bi-restrictions $P(h)|_\chi:\L^2_\chi(M)\cap \Sob^2(M)\to \L^2_\chi(M)$ to the different isotypic components. More generally, for an operator $A:D\to \L^2(M)$ defined on a $T_\chi$-invariant subset $D\subset \L^2(M)$ one can consider the associated \textit{reduced operator}
 \[
A_{\chi}:=T_{\chi}\circ A\circ T_{\chi}|_D.\]
 Since $P(h)$ commutes with $T_\chi$, the reduced operator $P(h)_\chi$ coincides with  $P(h)|_\chi$. Instead of considering only one isotypic component, one can also consider the bi-restriction of $P(h)$ to $h$-dependent sums of isotypic components of the form
\[
\L^2_\Wh(M)=\bigoplus_{\chi \in \Wh} \L^2_\chi(M),
\]
choosing for each $h\in (0,1]$ an appropriate finite subset $\Wh\subset \widehat G$ whose cardinality is allowed to grow in a controlled way as $h\to 0$. The study of a single isotypic component corresponds to choosing $\Wh=\{\chi\}$ for all $h$ and a fixed $\chi\in \widehat G$. 
%Conversely, each eigenspace of the Schr\"odinger operator $P(h)$ constitutes a unitary $G$-module, and its decomposition into a direct sum of irreducible $G$-representations represents the so-called \emph{fine structure} of the spectrum of $P(h)$. 
Note that, so far, it is a priori irrelevant whether the group action has various different orbit types or not.\\

On the other hand, the principal symbol of the Schr\"odinger operator is given by the $G$-invariant symbol function 
\bq
p: T^*M\to \R,\qquad(x,\xi)\mapsto \norm{\xi}^2_x + V(x),\label{eq:hamiltonianfunction}
\eq
and represents a Hamiltonian on the co-tangent  bundle $T^*M$ with  canonical symplectic form $\omega$. It defines a  Hamiltonian flow $\varphi_t:T^*M\to T^*M$, which in  the special case $V\equiv0$ corresponds to the geodesic flow on $T^\ast M$. Consider now for a regular value $c$ of  $p$  the hypersurface  $\Sigma_c:=p^{-1}(\{c\})\subset T^*M$. It is invariant under the Hamiltonian flow $\varphi_t$, and carries a canonical hypersurface measure $d\Sigma_c$  induced by $\omega$. In the special case $\Sigma_c=S^\ast M$, $d\Sigma_c=d(S^*M)$ is commonly called the \emph{Liouville measure}. Now, if $G$ is non-trivial, $\varphi_t$ cannot be ergodic on $(\Sigma_c,d\Sigma_c)$ due to the presence of additional conserved quantities besides the total energy $c$. To describe the dynamics of the system, it is therefore convenient to divide out the symmetries, which can be done by  performing a procedure called \emph{symplectic reduction}. The latter is based on the fundamental fact that the  presence of conserved quantities or first integrals of motion leads to   elimination of variables, and reduces the given configuration space with its symmetries to a lower-dimensional one, in which the degeneracies and the conserved quantitites have been eliminated. Namely, let $\Jbb:T^*M\rightarrow \g^\ast$ denote the momentum map of the Hamiltonian $G$-action on $T^*M$, which represents the conserved quantitites of the system, and consider  the topological quotient space 
\bqn
\widetilde \Omega:=\Omega/G, \qquad   \Omega:=\Jbb^{-1}(\mklm{0}).
\eqn
% In contrast to the situation encountered in the Peter-Weyl theorem,  the orbit structure of the underlying $G$-action on $M$ is not at all irrelevant for the symplectic reduction. Indeed, 
If the $G$-action is not free the space $\Omega$ need not be a manifold. Nevertheless, $\Omega$ and $\widetilde \Omega$ are stratified spaces, where each stratum is a smooth manifold that consists of  orbits of one particular type. In particular, $\Omega$ and $\widetilde \Omega$ each have a \emph{principal stratum} $\Omega_\text{reg}$ and $\widetilde \Omega_\text{reg}$, respectively, which is the smooth manifold consisting of (the union of) all orbits whose isotropy type is the minimal of $M$. Moreover, $\widetilde \Omega_\text{reg}$ carries a canonical symplectic structure, and the Hamiltonian flow on $T^\ast M$ induces a flow $\widetilde{\varphi}_t:\widetilde{\Omega}_\text{reg}\to\widetilde{\Omega}_\text{reg}$, which is the Hamiltonian flow associated to the reduced Hamiltonian $\widetilde{p}:\widetilde{\Omega}_\text{reg}\to\R$ induced by $p$. One calls $\widetilde{\varphi}_t$ the \emph{reduced Hamiltonian flow}. Since the orbit projection $\Omega_\text{reg}\to\widetilde{\Omega}_\text{reg}$ is a submersion, $c$ is also a regular value of the reduced symbol function $\widetilde{p}$, and we define $\widetilde{\Sigma}_c:=\widetilde{p}^{-1}(\{c\})\subset \widetilde{\Omega}_\text{reg}$. Similarly to $(\Sigma_c,d\Sigma_c)$, the  smooth hypersurface $\widetilde{\Sigma}_c=(\Omega_\text{reg} \cap \Sigma_c)/G\subset \widetilde{\Omega}_\text{reg}$ carries a measure $d\widetilde{\Sigma}_c $ induced by the symplectic form on $\widetilde{\Omega}_\text{reg}$, and one can interpret the measure space $(\widetilde{\Sigma}_c,d\widetilde{\Sigma}_c )$  as the symplectic reduction of  $(\Sigma_c,d\Sigma_c)$. Note that  $(\widetilde{\Sigma}_c,d\widetilde{\Sigma}_c )$ corresponds to the measure space  $\big(\Omega_{\text{reg}}\cap \Sigma_c,\frac{d\mu_c}{\vol_\O}\big)$, where $d\mu_c$ denotes the induced volume density on the smooth hypersurface $ \Omega_{\text{reg}}\cap \Sigma_c\subset \Omega_{\text{reg}}$, and  the function $\vol_\O:\Sigma_c\cap \Omega_{\text{reg}}\to (0,\infty)$, $x\mapsto \vol(G\cdot x)$ assigns to an orbit its Riemannian volume, see Section 2.4 of Part I.
%The smooth manifold corresponding to $\widetilde{\Sigma}_c$ before passing to the orbit space is the transversal intersection $\Sigma_c\cap \Omega_{\text{reg}}$, which constitutes a hypersurface of $\Omega_{\text{reg}}$ and as such carries an induced volume density that we denote by $d\mu_c$. The measure space that corresponds to $(\widetilde{\Sigma}_c,d\widetilde{\Sigma}_c )$  is then $\big(\Sigma_c\cap \Omega_{\text{reg}},\frac{d\mu_c}{\vol_\O}\big)$, where the function $\vol_\O:\Sigma_c\cap \Omega_{\text{reg}}\to (0,\infty)$, $x\mapsto \vol(G\cdot x)$ assigns to an orbit its Riemannian volume, see Section 2.4 of Part I. Since we will consider distributions with domain $\CT(\Sigma_c)$, we prefer to deal with subsets of $\Sigma_c$ instead of the associated orbit spaces. Therefore, the measure space $\big(\Sigma_c\cap \Omega_{\text{reg}},\frac{d\mu}{\vol_\O}\big)$ will be of fundamental importance in this paper.

Now, coming back to our initial question, let us assume  that the reduced Hamiltonian flow $\widetilde{\varphi}_t$ is ergodic on $(\widetilde{\Sigma}_c,d\widetilde{\Sigma}_c )$, and choose for each $h\in (0,1]$ an appropriate finite set $\Wh\subset \widehat G$ whose cardinality does not grow too fast as $h\to 0$, see Definition \ref{def:family} below. We then ask whether there is a non-trivial family of index sets $\{\Lambda(h)\}_{h\in (0,1]}$, $\Lambda(h)\subset \N$, such that for $j\in \Lambda(h)$ we have $u_j(h)\in \L^2_\chi(M)$ for some $\chi \in \W_h$, the associated eigenvalue $E_j(h)$ is close to $c$, and the distributions
\[
\mu_j(h):\CT(\Sigma_c) \longrightarrow \C, \quad a \longmapsto \eklm{\Op_h(a) u_{j}(h), u_{j}(h)}_{\L^2(M)}
\]
converge for $j \in \Lambda(h)$ and $h\to0$ to a distribution limit with density $1$, which would answer the corresponding question for the measures $|u_{j}(h)|^2dM$. In particular, in the special case $V\equiv0$, $c=1$, the problem is equivalent to finding quantum limits for sequences of eigenfunctions of the Laplace-Beltrami operator. In case that $\widetilde M$ is an orbifold and $\Wh=\{\chi_0\}$ for all $h$, where $\chi_0$ corresponds to the trivial representation,  this problem has been dealt with recently by  Kordyukov \cite{kordyukov12} using classical techniques.

\medskip

The general idea behind our approach  can be summarized as follows. The existence of symmetries of a  classical Hamiltonian system implies the existence of conserved quantitites and partial integrability of the Hamiltonian flow, forcing the system to behave less chaotically. Symplectic reduction divides out the symmetries, and hence,  order, and allows to study the symmetry-reduced spectral and ergodic properties of the corresponding quantum system. In particular, eigenfunctions should reflect the partially chaotic behavior of the classical system. 
In our approach, we shall combine well-known methods from semiclassical analysis and symplectic reduction with results on singular equivariant asymptotics recently developed in \cite{ramacher10}. In case of the Laplacian, it would also be possible to study the problem via the original classical approach of Shnirelman, Zelditch and Colin de Verdi\`{e}re.

\subsection{Results} To formulate our results, we need to introduce some additional notation. As explained in Section 2.3 of Part I, the $G$-action on $M$ possesses  a principal isotropy type $(H)$ which is represented by a principal isotropy subgroup $H\subset G$, as well as a principal orbit type. We denote by $\kappa$ the dimension of the principal orbits, which agrees with the maximal dimension of a $G$-orbit in $M$, and we assume throughout the whole paper that $\kappa<n=\dim M$. For an element $\chi\in \widehat G$ write $[\pi_{\chi}|_{H}:\mathds{1}]$ for the multiplicity of the trivial representation in the restriction of the irreducible $G$-representation $\pi_\chi$ to $H$.  Let $\widehat G'\subset \widehat G$ be the subset consisting of those classes of representations that appear in the decomposition \eqref{eq:PW} of $\L^2(M)$. In order to consider a growing number of isotypic components of $\L^2(M)$ in the semiclassical limit we make the following
\begin{definition}\label{def:family}A family $\{\W_h\}_{h\in (0,1]}$ of finite sets $\W_h\subset \widehat G'$ is called \emph{semiclassical character family} if there exists a $\vartheta\geq 0$ such that for each $N\in \{0,1,2,\ldots\}$ and each differential operator $D$ on $G$ of order $N$ there is a constant $C>0$ independent of $h$ with
\[
\frac 1 {\# \W_h}\sum_{\chi\in \W_h}\frac{\norm{D\overline{\chi}}_\infty}{\left[\pi_{\chi}|_{H}:\mathds{1}\right]}\leq C\,h^{-\vartheta N}\qquad \forall \;h\in(0,1].
\]
We call the smallest possible $\vartheta$ the \emph{growth rate} of the semiclassical character family.
\end{definition}
%\begin{example}\label{ex:first1}
As a simple example, consider the case  $G=\SO(2)\cong S^1\subset \C$. Then $\widehat G=\{\chi_k:k\in \Z\}$, where the $k$-th character $\chi_k:G\to \C$ is given by $\chi_k\big(e^{i\varphi}\big):=e^{ik\varphi}$, and one obtains a semiclassical character family with growth rate less or equal to $\vartheta$ by setting $\Wh:=\{\chi_k:|k|\leq h^{-\vartheta}\}$.  Analogous families can be constructed for any compact connected Lie group, see Example 1.2 of  Part I.
%\end{example}
Next,  denote by $\Psi_{h}^{m}(M)$, $m \in \R\cup\{-\infty\}$, the set of semiclassical pseudodifferential operators on $M$ of order $m$. The principal symbols of these operators are represented by symbol functions in the classes $S^m(M)$, see Section 2.1 of Part I. Finally, for any measurable function $f$ with domain $D$ a $G$-invariant subset of $M$ or $T^*M$ we write
\bq
{\eklm{f}}_G(x):= \int_Gf(g\cdot x)\d g \label{eq:orbitalintegral},
\eq
and  denote by $\widetilde{\eklm{f}}_G$ the function induced on the orbit space $D/G$ by the $G$-invariant function ${\eklm{f}}_G$. 
As before, let $P(h)$ be a Schr\"odinger operator defined by  \eqref{eq:15.08.2015} with eigenfunctions $\{u_{j}(h)\}_{j\in\mathbb{N}}$ and eigenvalues $\{E_{j}(h)\}_{j\in\mathbb{N}}$. We can now state the main result of this paper.

\begin{result}[{\bf Equivariant quantum ergodicity for Schr\"odinger operators}, Theorem \ref{thm:ergod2}]
\label{res:2}Suppose that the reduced Hamiltonian flow $\widetilde{\varphi}_{t}$ is ergodic
on $\widetilde{\Sigma}_c$. For a number $\beta\in\big(0,\frac{1}{2\kappa+4}\big)$ and a semiclassical character family $\{\Wh\}_{h\in(0,1]}$ with growth rate $\vartheta<\frac{1-(2\kappa+4)\beta}{2\kappa+3}$ set
\[
J(h):=\mklm{j\in \N:E_j(h)\in[c,c+h^\beta],\; \chi_j(h) \in \Wh},
\]
were $\chi_j(h)$ is defined by $u_j(h)\in \L^2_{\chi_j(h)}(M)$. Then, there is a $h_0\in (0,1]$ such that for each $h\in(0,h_0]$ we have a subset $\Lambda(h)\subset J(h)$ satisfying 
\begin{equation*}
\lim_{h\to0}\frac{\#\Lambda(h)}{\#J(h)}=1
\end{equation*}
such that for each semiclassical pseudodifferential operator $A\in\Psi_h^{0}(M)$ with principal symbol $\sigma(A)=[a]$, where $a$ is $h$-independent, the following holds. For all $\epsilon>0$ there is a $h_\epsilon \in (0,h_0]$ such that
\begin{equation*}
\frac{1}{\sqrt{d_{\chi_j(h)}[\pi_{\chi_j(h)}|_{H}:\mathds{1}]}}\;\Big|\left\langle Au_j(h),u_j(h)\right\rangle_{\L^2(M)}-\fintop_{\Sigma_c\cap \Omega_\text{reg}}a\, \frac{d\mu_c}{\vol_\O}\Big|\;<\;\epsilon \qquad\forall\,j\in \Lambda(h), \; \forall\, h\in(0,h_\epsilon].
\end{equation*}
Moreover,  the integral in the previous line equals $\fint_{\widetilde{\Sigma}_c}\widetilde{\eklm{a}}_G \d \widetilde{\Sigma}_c$. 
\end{result}

If $\W_h$ consists of just  a single character, the statement of Result \ref{res:2} is slightly simpler, see Theorem \ref{thm:ergod22}. Result \ref{res:2} will be deduced from  the equivariant semiclassical Weyl law  proved in Part I. The proof of the latter is based on a functional calculus for semiclassical pseudodifferential operators and $h$-dependent test functions developed in \cite{kuester15}, and  reduces to the asymptotic description of certain oscillatory integrals  that have recently been studied in 
 \cite{ramacher10} using resolution of singularities. The involved phase functions are given  in terms of the underlying $G$-action on $M$, and if singular orbits occur,  the corresponding  critical sets are no longer smooth, so that a partial  desingularization process has to be implemented  in order to obtain asymptotics with remainder estimates via  the stationary phase principle. Let us  emphasize  that  the remainder estimate  for the equivariant semiclassical Weyl law proved in Part I, and consequently the desingularization process implemented in \cite{ramacher10},   are   crucial for studying the shrinking spectral windows $[c,c+h^\beta]$  and the growing families $\Wh$ of representations in Result \ref{res:2}.  In the special case of the Laplacian, Result \ref{res:2} becomes an equivariant version of the classical quantum ergodicity theorem of Shnirelman \cite{shnirelman}, Zelditch \cite{zelditch1987}, and  Colin de Verdi\`{e}re \cite{colindv}. To state it, let $\{u_j\}_{j\in\N}$ be an orthonormal basis in $\L^2(M)$ of eigenfunctions of $- \Delta$ with associated eigenvalues $\{E_j\}_{j\in \N}$.
 
 \begin{result}[{\bf Equivariant quantum limits for the Laplacian}, Theorem \ref{thm:quantlim}] 
 \label{res:3}
 Assume that the reduced geodesic flow is ergodic. Choose a semiclassical character family $\mklm{\W_h}_{h \in (0,1]}$ of growth rate $\vartheta<\frac{1}{2\kappa +3}$  and a partition $\mathcal{P}$  of the set $\{E_j\}_{j\in \N}$ of order $\beta\in \big(0,\frac{1-(2\kappa+3)\vartheta}{2\kappa +4}\big)$ in the sense of Definition \ref{def:partition}. Define the set of eigenfunctions
\[
\big\{u^{\W,\mathcal{P}}_i\big\}_{i\in\N}:=\big\{u_j: \chi_j\in \W_{E^{-1/2}_{\P(j)}}\big\},
\]
where $\chi_j$ is defined by $u_j\in \L^2_{\chi_{j}}(M)$. Then, there is a subsequence $\big\{u^{\W,\mathcal{P}}_{i_k}\big\}_{k\in\N}$ of density 1 in $\big\{u^{\W,\mathcal{P}}_i\big\}_{i\in\N}$ such that for all $s\in \Cinft(S^*M)$ one has
\begin{equation*}
\frac{1}{\sqrt{d_{\chi_{i_k}}[{\pi_{\chi_{i_k}}|_{H}}:\mathds{1}]}}\;\Big|\eklm{\Op(s) u^{\W,\mathcal{P}}_{i_k}, u^{\W,\mathcal{P}}_{i_k}}_{\L^2(M)} -\fintop_{S^*M\cap\Omega_\text{reg}}s \, \frac{d\mu}{\vol_\O}\Big|\longrightarrow 0 \qquad \text{ as } {k\to\infty},
\end{equation*}
where we wrote $\mu$ for $\mu_1$ and $\Op$ for $\Op_1$, which is the ordinary non-semiclassical quantization.
\end{result}

In the special case of a single isotypic component, Result \ref{res:3} simplifies to the following statement. Let $\{u^\chi_j\}_{j\in\N}$ be an orthonormal basis of $\L^2_\chi(M)$ consisting of eigenfunctions of $- \Delta$. Then, there is a subsequence $\{u^\chi_{j_k}\}_{k\in\N}$ of density $1$ in $\{u^\chi_j\}_{j\in\N}$ such that for all $a\in \Cinft(S^*M)$ one has
\begin{equation*}
\eklm{\Op(a) u^\chi_{j_k}, u^\chi_{j_k}}_{\L^2(M)} \;{\longrightarrow}\; \frac 1{\vol_{\frac {\mu}{\vol_\O}} (S^\ast M\cap\Omega_\text{reg})}
\intop_{S^*M\cap\Omega_\text{reg}}a \, \frac{d\mu}{\vol_\O} \qquad \text{ as } {k\to\infty},
\end{equation*}
 see Theorem \ref{thm:quantlim2}.
\medskip

The obtained quantum limits $(\vol_{\frac {\mu_c}{\vol_\O}} (\Sigma_c\cap\Omega_\text{reg}))^{-1} \frac{d\mu_c}{\vol_\O}$ describe the ergodic properties of the eigenfunctions in the presence of symmetries, and are the answer to our initial question. They are singular measures since they are supported on $\Sigma_c\cap\Omega_\text{reg}$, which is a  submanifold of $\Sigma_c$ of codimension $\kappa$.  In fact, they correspond to Liouville measures on the smooth bundles 
\bqn 
S^\ast_{\widetilde p,c}(\widetilde M_\text{reg}):=\mklm{(x,\xi) \in T^\ast (\widetilde M_\text{reg}): \widetilde p(x,\xi)=c}
\eqn
over  the space of principal orbits in $M$; if $\widetilde M$ is an orbifold, they are given by  integrals over the orbifold bundles $S^\ast_{\widetilde p,c}(\widetilde M):=\mklm{(x,\xi) \in T^\ast \widetilde M: \widetilde p(x,\xi)=c}$, see Remark \ref{rem:0304}. In the latter case,  the ergodicity of the reduced flow $\widetilde{\varphi}_{t}$  on $\widetilde{\Sigma}_{c}$ is equivalent to the ergodicity of the corresponding Hamiltonian flow on the orbifold bundle $S^\ast_{\widetilde p,c} (\widetilde M)$ with respect to the canonical Liouville measures. 
\medskip

Projecting from $S^*M\cap\Omega_\text{reg}$ onto $M$  we immediately deduce from Result \ref{res:3} for any $f \in C(M)$
\begin{equation*}
\frac{1}{\sqrt{d_{\chi_{i_k}}[\pi_{\chi_{i_k}}|_{H}:\mathds{1}]}}\;\Big| \intop_M f | u^{\W,\mathcal{P}}_{i_k}|^2 dM -\fintop_{M}f \, \frac{dM}{\vol_\O} \Big|\longrightarrow 0 \qquad \text{ as } {k\to\infty}, 
\end{equation*} 
which describes the asymptotic equidistribution of the eigenfunctions in the presence of symmetries, see Corollary \ref{cor:equi}. For a single isotypic component we get the weak convergence of measures
\begin{equation*}
|u^\chi_{j_k}|^2 \d M \;{\longrightarrow}\;\Big(\text{vol}_{\frac{dM}{\vol_\O}}M\Big)^{-1}\frac{dM}{\vol_\O} \qquad \text{ as } {k\to\infty},
\end{equation*}
compare Corollary \ref{cor:24.08.2015}. The fact that  the reduced and the non-reduced flow cannot be simultaneously ergodic is consistent with the QUE conjecture, since otherwise our results would, in principle,  imply the existence of  exceptional subsequences for ergodic geodesic flows. In this sense, our results can be understood as complementary to the previously known results. 
Applying some elementary representation theory, one can deduce from Corollary \ref{cor:equi} a statement on convergence of measures on the topological Hausdorff space $\widetilde{M}$ associated to irreducible $G$-representations. For this, choose an orthogonal decomposition of $\L^2(M)$ into a direct sum  $\bigoplus_{i\in \N}V_i$ of irreducible unitary $G$-modules such that each $V_i$ is contained in an  eigenspace of the Laplace-Beltrami operator corresponding to some eigenvalue $E_{j(i)}$. Denote by $\chi_i\in \widehat G$ the class of $V_i$.

\begin{result}[{\bf Representation-theoretic equidistribution theorem}, Theorem \ref{thm:equishnirelman3}]\label{res:4}
Assume that the reduced geodesic flow is ergodic.  Choose a semiclassical character family $\mklm{\W_h}_{h \in (0,1]}$ of growth rate $\vartheta<\frac{1}{2\kappa +3}$ and a partition $\mathcal{P}$ of $\{E_j\}_{j\in \N}$ of order $\beta\in (0,\frac{1-(2\kappa+3)\vartheta}{2\kappa +4})$. Define the set of irreducible $G$-modules
\[
\big\{V^{\W,\mathcal{P}}_l\big\}_{l\in\N}:=\big\{V_i: \chi_i\in \W_{E^{-1/2}_{\P(j(i))}}\big\}.
\]
As in Lemma \ref{lem:orbitint}, assign to each $V^{\W,\mathcal{P}}_l$ the $G$-invariant function $\Theta_l:= \Theta_{V^{\W,\mathcal{P}}_l}:M\to [0,\infty)$, and regard it as a function on $M/G=\widetilde{M}$.  Then, there is a subsequence $\big\{V^{\W,\mathcal{P}}_{l_m}\big\}_{m\in\N}$ with
\[
\lim_{N\to \infty}\frac{\sum_{l_m\leq N}d_{\chi_{l_m}}}{\sum_{i\leq N}d_{\chi_i}}=1
\]
for which 
\begin{equation*}
\frac{1}{\sqrt{d_{\chi_{l_m}}[\pi_{\chi_{l_m}}|_{H}:\mathds{1}]}}\;\Big| \int_{\widetilde M} f \, \Theta_{l_m}\,  \d\widetilde M - \fint_{\widetilde M} f \frac{d\widetilde{M}}{\vol} \Big|\longrightarrow 0 \qquad \text{ as } {m\to\infty},
\end{equation*}
where $d\widetilde{M}:=\pi_{\ast}dM$ is the pushforward measure defined by the orbit projection $\pi:M\to M/G=\widetilde{M}$ and $\vol: \widetilde{M} \rightarrow (0,\infty)$ assigns to an orbit its Riemannian volume.
\end{result}

For a single isotypic component, one  obtains  a simpler  statement by considering an orthogonal decomposition of  $\L^2_\chi(M)$ into  a sum $\bigoplus_{i\in \N}V^\chi_i$ of irreducible unitary $G$-modules of class $\chi$ such that each $V^\chi_i$ is contained in some eigenspace of the Laplace-Beltrami operator. Then,  we have the weak convergence of measures
\begin{equation*}
\Theta^{\chi}_{i_k} \d\widetilde{M} \;\overset{k\to\infty}{\longrightarrow}\; \Big(\textrm{vol}_{\frac{d\widetilde{M}}{\vol}}\widetilde{M}\Big)^{-1}\frac{d\widetilde{M}}{{\vol }}
\end{equation*}
for a subsequence $\{V^{\chi}_{i_k}\}_{k\in \N}$ of density $1$ in $\{V^{\chi}_i\}_{i\in \N}$, see Theorem \ref{thm:equishnirelman34}.
Note that Result \ref{res:4} is a statement about limits of representations, or multiplicities, and not eigenfunctions,  since it assigns to unitary irreducible  $G$-module in $\L^2(M)$  a measure on $\widetilde{M}$, and then considers the weak convergence of those measures.  In essence, it  can therefore  be regarded as a representation-theoretic statement
 in which the spectral theory for the Laplacian  only enters in  \emph{choosing a concrete decomposition} of each isotypic component.   In the case of the trivial group $G=\{e\}$, there is only one isotypic component in $\L^2(M)$, associated to the trivial representation, and choosing a family  of irreducible modules is equivalent to choosing a Hilbert basis of $\L^2(M)$ of eigenfunctions of the Laplace-Beltrami operator. Result \ref{res:4} then reduces to the classical equidistribution theorem for the Laplacian.

\medskip

In Section \ref{sec:example} we consider some  concrete examples to illustrate our results. They include 
\begin{itemize}
\item compact locally symmetric spaces ${\mathbb Y}:=\Gamma \backslash \G/K$,  where $\G$ is a  connected semisimple Lie group of rank $1$ with finite center,  $\Gamma$ a discrete co-compact subgroup, and $K$ a maximal compact subgroup;
\item all surfaces of revolution diffeomorphic to the $2$-sphere;
\item $S^3$-invariant metrics on the $4$-sphere.
\end{itemize}
In the first case,  $K$ acts with finite isotropy groups on  $\Xbb:=\Gamma \backslash \G$, so that $\mathbb Y$ is an orbifold. Furthermore,  the  orbit volume
% \footnote{\bf Check in the torsion case.} 
is constant. The reduced geodesic flow on $M=\Xbb:=\Gamma \backslash \G$ coincides with the geodesic flow on  ${\mathbb Y}$ and is ergodic, since ${\mathbb Y}$ has strictly negative sectional curvature.  Our results recover  the Shnirelman-Zelditch-Colin-de-Verdi\`{e}re theorem for  $\L^2({\mathbb Y})\simeq \L^2(\Xbb)^K$, and generalize it to non-trivial isotypic components  of $\L^2(\Xbb)$. 
In the examples of the 2- and 4-dimensional spheres, the considered actions have two fixed points,  and the reduced geodesic flow is ergodic for topological reasons, regardless of the choice of invariant Riemannian metric and in spite of the fact that the geodesic flow can be totally integrable.
Since the eigenfunctions of the Laplacian on the standard $2$-sphere -- the spherical harmonics -- are well understood, we can independently verify Result \ref{res:4}  for single isotypic components in this case.

\subsection{Previously known results} \label{sec:1.4}

In case that    $G$ acts on $M$ with only one orbit type,  $\widetilde{M}$ is a compact smooth manifold with Riemannian metric induced by the $G$-invariant Riemannian metric on $M$. By co-tangent  bundle reduction, $T^\ast\widetilde M$ is symplectomorphic to $\Jbb^{-1}(\{0\})/G$, so the ergodicity of  the reduced geodesic flow on $M$ and that of the geodesic flow on $\widetilde{M}$ are equivalent. Under these circumstances, one can apply the classical Shnirelman-Zelditch-Colin-de-Verdi\`{e}re equidistribution theorem to $\widetilde{M}$, yielding an equidistribution statement for the eigenfunctions of the Laplacian $\Delta_{\widetilde{M}}$ on $\widetilde{M}$ in terms of weak convergence of measures on $\widetilde{M}$. On the other hand, one could as well  apply Corollary \ref{cor:equi} and Theorem \ref{thm:equishnirelman3} to $M$, yielding also a statement about weak convergence of measures on $\widetilde{M}$, but this time with measures related to eigenfunctions of the Laplacian $\Delta_{M}$ on $M$ in families of isotypic components of $\L^2(M)$. It is then an obvious question how these two  results are related. The answer  is  rather difficult in general, since -- in spite of the presence of the isometric group action -- the geometry of $M$ may be much more complicated than that of $\widetilde{M}$. Consequently, the eigenfunctions of $\Delta_{M}$, even those in the trivial isotypic component, that is, those that are  $G$-invariant, may be much harder to understand than the eigenfunctions of $\Delta_{\widetilde{M}}$. Only  in case that all orbits are  totally geodesic or minimal submanifolds, or, more generally,  do all have the same volume, one can show that an eigenfunction of $\Delta_{\widetilde{M}}$ lifts to a unique $G$-invariant eigenfunction of $\Delta_{M}$ \cite{watson, bruening81,bergery}. In this particular situation, it is easy to see that the  application of the Shnirelman-Zelditch-Colin-de-Verdi\`{e}re equidistribution theorem implies our results, but only for the single trivial isotypic component. The case of a compact locally symmetric space treated in Section \ref{sec:8.1} is an example of this in the torsion-free case. In cases where the orbit volume is not constant,  we do not know of any significant results about the relation between the eigenfunctions of $\Delta_{\widetilde{M}}$ and $\Delta_{M}$. 

An explicitly  studied case is that of a general free $G$-action, when the projection $M\to M/G= \widetilde{M}$ is a Riemannian principal $G$-bundle. Extending work of Schrader and Taylor \cite{schrader}, Zelditch  \cite{Zelditch19921}  obtained  quantum limits for sequences of eigenfunctions of $\Delta_M$ in so-called \emph{fuzzy ladders}. These are subsets of $\L^2(M)$ associated to a so-called \emph{ray of representations} originating from some chosen $\chi\in\widehat{G}$. The obtained quantums limit  are  directly related to the symplectic orbit reduction $\Jbb^{-1}(\mathcal{O}_\chi)/G \simeq T^\ast \widetilde{M}$, where $\mathcal{O}_\chi\subset\mathfrak{g}^*$ is the co-adjoint orbit associated to $\chi$ by the Borel-Weil theorem. They are given by Liouville measures on hypersurfaces in $\Jbb^{-1}(\mathcal{O}_\chi)/G$, and their projections onto the base manifold agree with ours. 

Further, significant efforts were recently made towards the understanding of quantum (unique) ergodicity for locally symmetric spaces, which are particular manifolds of negative sectional curvature. As before, let $\G$ be a connected, semisimple Lie group with finite center, $\G=KAN$ an Iwasawa decomposition of $\G$,  and $\Gamma$ a torsion-free, discrete subgroup in $\G$. Following earlier work of Zelditch and Lindenstrauss,  Silberman and Venkatesh introduced  in \cite{silberman-venkatesh07} certain representation theoretic lifts from ${{\mathbb Y}}=\Gamma\backslash \G / K$  to $\Xbb=\Gamma \backslash \G$ that substitute the previously considered microlocal lifts and take into account the additional structure of locally symmetric spaces. These representation theoretic lifts should play an important role in solving the QUE conjecture, already settled by Lindenstrauss in  particular cases, also for  higher rank symmetric spaces. In case that $\Gamma$ is co-compact, their results were generalized by Bunke and Olbrich \cite{bunke-olbrich06} to homogeneous vector bundles $\Xbb \times_K V_\chi$  over ${{\mathbb Y}}$ associated to equivalence classes of irreducible representations $\chi\in \widehat K$ of the maximal compact subgroup $K$.  The constructed representation theoretic lifts are invariant with respect to the action of $A$, which corresponds to the invariance of the microlocal lifts under the geodesic flow. Since $\Gamma$ has no torsion, $K$ acts on $\Xbb$ only with one orbit type. 

Finally, there has  been much work in recent times  concerning the spectral theory of elliptic operators on orbifolds. Such spaces are locally homeomorphic to a quotient of Euclidean space by a finite group while, globally, any (reduced) orbifold is a quotient of a smooth manifold by a compact Lie group action with finite isotropy groups, that is, in particular, with no singular isotropy types \cite{adem-leida-ruan, moerdijk-mrcun}. As it turns out, the theory of elliptic operators on orbifolds is essentially equivalent to the theory of invariant elliptic operators on manifolds carrying the action of a compact Lie group with finite isotropy groups \cite{bucicovschi, dryden-et-al, stanhope-uribe}. In particular, Kordyukov \cite{kordyukov12}  obtained the Shnirelman-Zelditch-Colin-de-Verdi\`{e}re theorem for elliptic operators on compact orbifolds, using their original high-energy approach. Result \ref{res:3} recovers his result for the Laplacian, and generalizes it to  singular group actions and  growing families of isotypic components.

\medskip

Thus, in all the previously examined cases,  no singular orbits occur. Actually, our work can be viewed  as part of an attempt to develop a  spectral theory of elliptic operators on  general singular $G$-spaces.

\medskip

To close, it might be appropriate to mention that Marklof and O'Keefe \cite{marklof}  obtained quantum limits in situations where the geodesic flow is ergodic only in certain regions of  phase space. Conceptually, this is both similar and contrary to our approach, since in this case the geodesic flow is partially ergodic as well, but not due to symmetries.

\subsection{Comments and outlook}\label{sec:1.5}
We would like to close this introduction by making some comments, and indicating some possible research lines for the future. 

Weaker versions of Result \ref{res:2} and \ref{res:3} can be proved in the case 
%where $\Wh=\{\chi\}$ for all $h\in (0,1]$ and some fixed $\chi\in \widehat G$ 
of a single isotypic component by the same methods employed here with a less sharp energy localization in a fixed interval $[c,c+\varepsilon]$  instead of a shrinking interval $[c,c+h^\beta]$. The point is that for these weaker statements no remainder estimate in the semiclassical Weyl law is necessary, see  Remark \ref{rem:23.11}. Thus, at least the weaker version of Result \ref{res:3} could have also been obtained within the classical framework in the late 1970's   using heat kernel methods as in \cite{donnelly78} or \cite{bruening-heintze79}. In contrast, for the stronger versions of equivariant quantum ergodicity proved in Result \ref{res:2} and \ref{res:3}, remainder estimates in the equivariant Weyl law, and in particular the results obtained in  \cite{ramacher10}  for general group actions via resolution of singularities, are necessary. However, the weaker versions would still be strong enough to imply Result \ref{res:4} for a single isotypic component. Therefore, in principle, Theorem \ref{thm:equishnirelman34}
% the representation-theoretic equidistribution theorem for single isotypic components 
could have been proved already when Shnirelman formulated his theorem more than 40 years ago.

\begin{figure}[h!]
\hspace{-1.2em}\begin{minipage}{0.50\linewidth}
\begin{overpic}[scale=0.8]{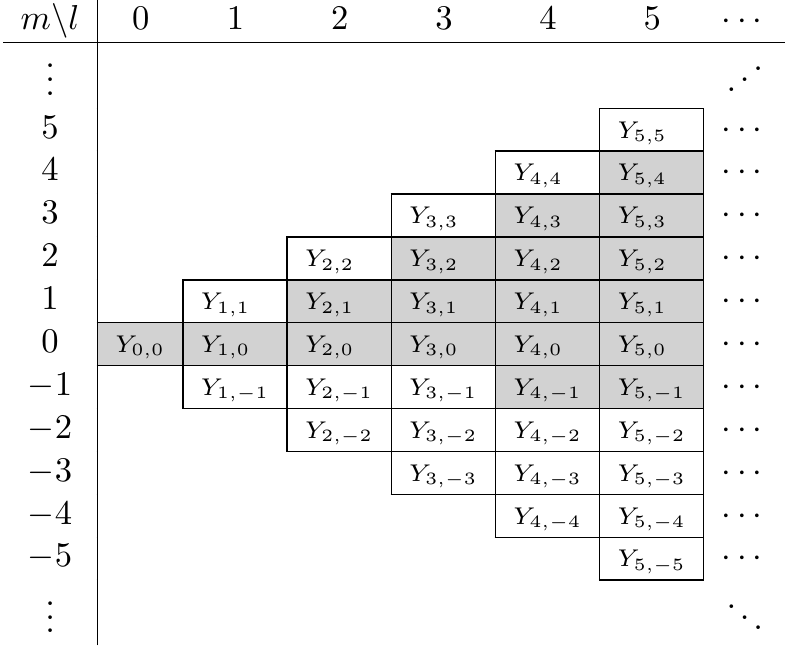}
\put(21.8,45,9){\includegraphics[scale=0.55]{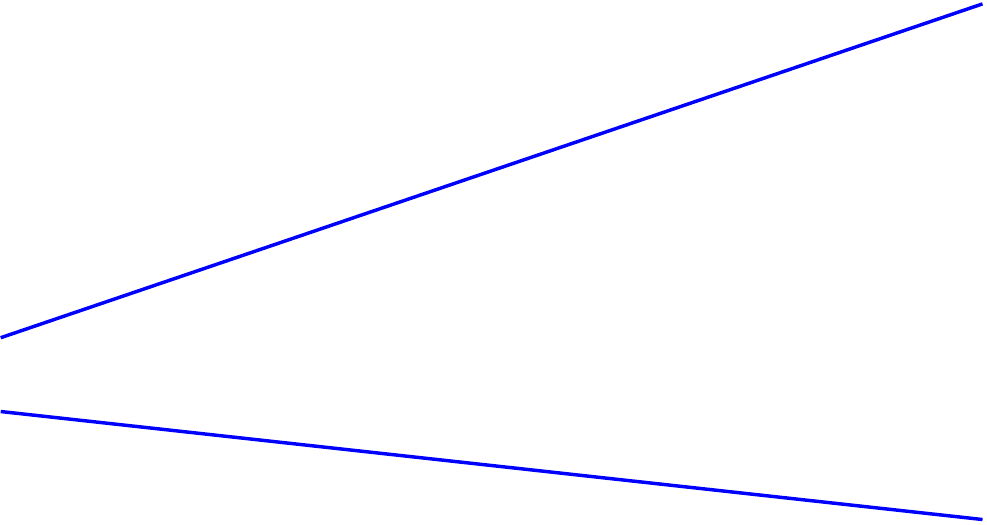}}
\end{overpic}
\centering
\caption{Spherical harmonics on $S^2$ in  cone-like families of representations.% like those   considered by Schrader and Taylor in \cite{schrader}.
 \label{fig:schrader}}
\end{minipage}\centering
\hspace{.5em} \begin{minipage}{0.50\linewidth}
{\includegraphics[scale=0.8]{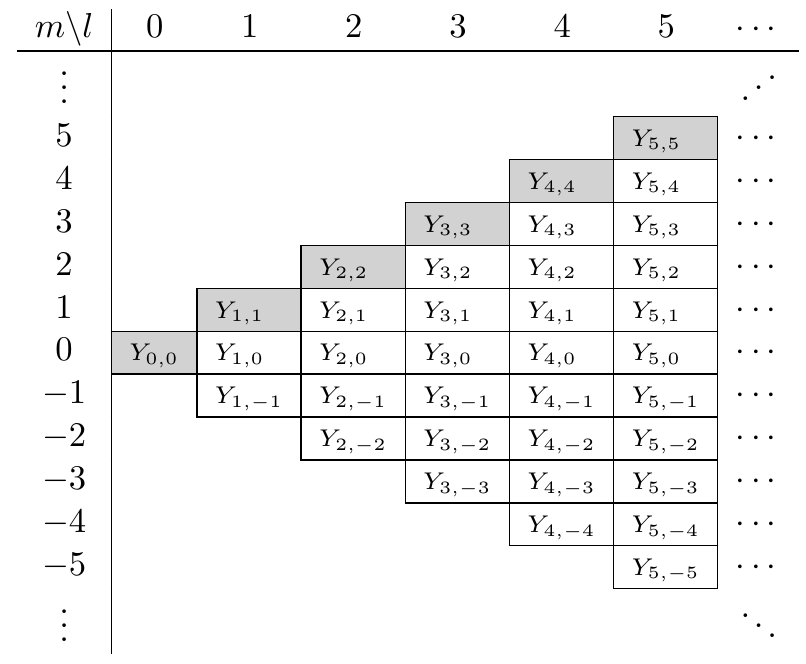}}
\centering\caption{Zonal spherical harmonics on $S^2$. \label{fig:zonal}}
\end{minipage}
\end{figure}

As mentioned above, the  idea of considering families of representations that vary with the asymptotic parameter has been known since the end of the 1980's, compare \cite{schrader, guillemin-uribe90, Zelditch19921}, and it is a natural problem to determine what kind of families can be considered in the context of quantum ergodicity, and study them from a more conceptional point of view. To illustrate this, consider the example of the  standard $2$-sphere $S^2\subset \R^3$, acted upon by the group $\SO(2)\simeq S^1$ of rotations around the $z$-axis in $\R^3$. This action has exactly two fixed points given by the north pole and the south pole of $S^2$, while all other orbits are circles. The eigenvalues of $-\Delta$ on $S^2$ are given by the numbers $l(l+1)$, $ l=0,1,2,3\dots$, and the corresponding eigenspaces $\E_l$ are of dimension $2l+1$. They are  spanned by the spherical harmonics $Y_{l,m}$ given by the Legendre polynomials, where $m \in \Z$, $|m|\leq l$. Each subspace $\C \cdot Y_{l,m}$ corresponds to an irreducible representation of $\SO(2)$,  and each irreducible representation with character $\chi_k(e^{i\phi})=e^{ik\phi}$  and $|k|\leq l$ occurs in the eigenspace $\E_l$ with multiplicity $1$. The semiclassical character families considered in our work have been illustrated in Figures 1.1 and 1.2 of Part I. As opposed to our results, Figure \ref{fig:schrader} illustrates a cone-like family of representations  that  would correspond to subsequences of eigenfunctions of density larger than zero, while 
Figure \ref{fig:zonal} depicts the  sequence of zonal spherical harmonics $Y_{l,l}$, which are known  to  localize at the equator of $S^2$ as $l\to \infty$, and therefore yield a different limit measure than the one implied by Result \ref{res:4},  see Section \ref{sec:6.2} and in particular Remark \ref{rem:last1}. Therefore, different kinds of families of representations give rise to qualitatively different quantum limits, and it would be illuminating to understand this interrelation in a deeper way.

\medskip

As further lines of research,  it would be interesting to see whether our results can be generalized to $G$-vector bundles, as well as manifolds with boundary and non-compact situations. Also, in view of Result \ref{res:4}, it might be possible to deepen our understanding of  equivariant quantum ergodicity via representation theory. Finally, one can ask what could be a suitable symmetry-reduced version of the QUE conjecture, and we intend to deal with these questions in the future. 
In the particular case of the $\SO(2)$-action on the standard 2-sphere studied in Section \ref{sec:example}, we actually show that in each fixed isotypic component the representation-theoretic equidistribution theorem for the Laplacian applies to  the full sequence of spherical harmonics, so that equivariant QUE holds in this case. However, even for this simple example it is unclear whether equivariant QUE holds for  growing families of isotypic components.

\section{Background}

In this section we describe  the setup in more detail, and collect the relevant results from Part I  \cite{kuester-ramacher15a} needed in the upcoming sections. For a systematic exposition of the background with corresponding references, we refer the reader to Section 2 and Appendix A of Part I.

\subsection{Symplectic reduction}
\label{sec:2.3} In what follows, we  review  in some detail  the  theory of symplectic reduction of Marsden and Weinstein, Sjamaar, Lerman and Bates. It was already briefly recalled in Part I. The theory emerged out of classical mechanics, and is based on the fundamental fact that the  presence of conserved quantities or integrals of motion leads to the elimination of variables. %For a detailed exposition of these facts we refer the reader to \cite{ortega-ratiu}. 
Let $(\Xbf,\omega)$ be a connected symplectic manifold, and assume that  $(\Xbf,\omega)$ carries  a global Hamiltonian action of a Lie group  $G$. In particular, we will be interested in the case where $\Xbf=T^\ast M$ is the co-tangent bundle of our manifold $M$.   Let 
\begin{eqnarray*}
\Jbb:\Xbf  \to  \g^\ast,\qquad \Jbb(\eta)(X)=  \Jbb_X(\eta),
\end{eqnarray*}
be the corresponding \emph{momentum map}, where $\Jbb_{X}:\Xbf  \to  \mathbb{R}$ is a $\Cinft$-function depending linearly on $X\in \g$ such that the fundamental vector field $\widetilde X$ on $\Xbf$ associated to $X$  is given by the Hamiltonian vector field of $\Jbb_X$. It is clear from the definition that  $\Ad^\ast(g^{-1}) \circ \Jbb= \Jbb \circ g$.  Furthermore, for each $X \in \g$ the function $\Jbb_X$ is a \emph{conserved quantity} or \emph{integral of motion} for any $G$-invariant function $p \in \Cinft(\Xbf)$ since in this case 
\bqn 
\mklm{\Jbb_X, p}=\omega(\sgrad \Jbb_X,\sgrad p)=-\omega(\widetilde X, \sgrad p)=dp(\widetilde X)=\widetilde X(p)=0,
\eqn
where  $\mklm{\cdot, \cdot}$ is the Poisson-bracket on $\Xbf$  given by $\omega$. Now, define 
\bqn 
\Omega:=\Jbb^{-1}(\mklm{0}), \qquad \widetilde \Omega:=\Omega/G.
\eqn
Unless the $G$-action on $\Xbf$ is free, the reduced space $\widetilde \Omega$ will in general not be a smooth manifold, but a  topological quotient space. Nevertheless, one can show that $\widetilde \Omega$ constitutes  a stratified   symplectic space in the following sense. A function $\widetilde f: \widetilde \Omega \rightarrow \R$ is defined to be \emph{smooth}, if there exists a $G$-invariant function $f\in \Cinft(\Xbf)^G$ such that $f|_{\Omega}=\pi^\ast \widetilde f$, where $\pi:\Omega \rightarrow \widetilde \Omega$ denotes the orbit map. One can then show that  $\Cinft(\widetilde \Omega)$ inherits a Poisson algebra structure from $\Cinft(\Xbf)$ which is compatible with a stratification of the reduced space into symplectic manifolds. Moreover, the Hamiltonian flow $\varphi_t$ corresponding to $f$ is $G$-invariant and leaves $\Omega$ invariant, and consequently descends to a flow $\widetilde \varphi_t$ on $\widetilde \Omega$ \cite{lerman-sjamaar}.

More precisely, let $\mu$ be  a value of $\Jbb$, and  $G_\mu$ the isotropy group of $\mu$ with respect to the co-adjoint action on $\g^\ast$. Consider further an isotropy group $K\subset G$ of the $G$-action on $\Xbf$, let $\eta\in \Jbb^{-1}(\{\mu\})$ be such that $G_\eta=K$, and $\Xbf^\eta_K$ be the  connected component of $\Xbf_K:=\mklm{\zeta \in \Xbf:G_\zeta =K}$ containing $\eta$. Then \cite[Theorem 8.1.1]{ortega-ratiu}
the set  $\Jbb^{-1}(\{\mu\}) \cap G_\mu \cdot \Xbf^\eta_K$ is a smooth submanifold of $\Xbf$, and the quotient
\[
\widetilde{\Omega}_\mu^{(K)}:= \big(\Jbb^{-1}(\{\mu\}) \cap G_\mu \cdot \Xbf^\eta_K\big)\big / G_\mu
\]
possesses a differentiable structure such that the projection $\pi_\mu^{(K)}:\Jbb^{-1}(\{\mu\}) \cap G_\mu \cdot \Xbf^\eta_K \rightarrow \widetilde{\Omega}_\mu^{(K)}$ is a surjective submersion. Furthermore, there exists a unique symplectic form $\widetilde{\omega}_\mu^{(K)}$ on $\widetilde{\Omega}_\mu^{(K)}$ such that 
$
(\iota_\mu^{(K)})^\ast \omega = (\pi_\mu^{(K)})^\ast (\widetilde{\omega}_\mu^{(K)})
$,
where $\iota_\mu^{(K)}:  \Jbb^{-1}(\{\mu\}) \cap G_\mu \cdot \Xbf^\eta_K \hookrightarrow \Xbf$ denotes the inclusion.  
Finally, if $p \in \Cinft(\Xbf)$ is  a $G$-invariant function, $H_p:=\sgrad p$ its Hamiltonian vector field, and $\varphi_t$ the corresponding  flow, then $\varphi_t$ leaves invariant the components of $ \Jbb^{-1}(\{\mu\}) \cap G_\mu \cdot \Xbf^\eta_K$ and commutes with the $G_\mu$-action, yielding a reduced flow $\widetilde{\varphi}_t ^\mu$ on $\widetilde{\Omega}_\mu^{(K)}$ given by
\bq
\label{eq:11.06.2015}
\pi_\mu^{(K)} \circ \varphi_t \circ \iota_\mu^{(K)}=\widetilde{\varphi}_t^\mu \circ \pi_\mu^{(K)}. 
\eq
This reduced flow $\widetilde{\varphi}_t ^\mu$ on $\widetilde{\Omega}_\mu^{(K)}$ turns out to be Hamiltonian, and its Hamiltonian $\widetilde{p}_\mu^{(K)}:\widetilde{\Omega}_\mu^{(K)}\rightarrow \R$ satisfies 
$\widetilde{p}_\mu^{(K)} \circ \pi_\mu^{(K)} = p \circ \iota_\mu^{(K)}.$

%\end{thm}
%\vspace*{-1em}
%\qed

\begin{rem}
With the notation above  we have $G\cdot \Xbf_K=\Xbf(K)$. Indeed, for $x\in \Xbf_K$, the isotropy group of $x$ is $K$. If $g'g\cdot x =g\cdot x $ for some $g,g'\in G$, then $g^{-1}g'g\cdot x =x$, hence $g^{-1}g'g\in K$, that is  $g'\in (K)$. That shows $G\cdot \Xbf_K\subset \Xbf(K)$.  On the other hand, if $x\in \Xbf(K)$, then $(G_x)=(K)$, hence for every $g'\in G_x$, there is a $k\in K$ and a $g\in G$ such that $g'=gkg^{-1}$. But then $kg^{-1}\cdot x=g^{-1}\cdot x$, so that $g^{-1}\cdot x\in \Xbf_K$, and in particular $x\in G\cdot \Xbf_K$.
\end{rem}

\begin{example}
\label{ex:2.2}
Let $G$ be a Lie group.  An important class of examples of Hamiltonian group actions is given by  induced actions  on co-tangent  bundles of $G$-manifolds. Thus, let $\Psi: G\times M \to M, (g,x) \to \Psi_g(x):=g\cdot x$  be  a smooth $G$-action on a smooth manifold $M$. The induced action on $T^\ast M$  is given by
\bqn 
(g\cdot \eta_x)(v)= ( (\Psi_{g^{-1}})_{g\cdot x}^\ast \cdot  \eta_x) (v)= \eta_{x} ( (\Psi_{g^{-1}})_{\ast,g\cdot x} \cdot v), \qquad \eta_x \in T^\ast_xM, \, v \in T_{g\cdot x}M,
\eqn
where $(\Psi_{g})_{\ast, x}:T_{x}M \to T_{g \cdot x}M$ denotes the derivative of the map $g:M \to M, x \mapsto g\cdot x$. 
Now, if  $\tau: \Xbf=T^{*}M {\rightarrow} M$ denotes the co-tangent  bundle  with standard symplectic form $\omega=-d\theta$, where $\theta$ is the \textit{tautological} or \textit{Liouville} one-form on $T^*M$, then 
\begin{eqnarray}
\label{eq:19.07.2015}
\Jbb: T^{*}M \ni  \eta  \mapsto \Jbb(\eta)(X):= \eta\big(\widetilde{X}_{\tau(\eta)}\big), \qquad X \in \g,
\end{eqnarray}
defines a co-adjoint equivariant momentum map, meaning that the $G$-action on $T^\ast M$ is Hamiltonian. Here $\widetilde{X}_{\tau(\eta)}$ denotes the fundamental vector field on $M$ corresponding to $X$ evaluated at the point $\tau(\eta)$. In the particular case when $M=G$ is itself a Lie group, and $L:G\times G \to G$ denotes the left action of $G$ onto itself, there exists a vector bundle isomorphism
\bq
\label{eq:lefttriv}
T^\ast G \stackrel{\simeq}\longrightarrow G \times \g^\ast, \qquad  \eta_g \mapsto (g, (L_{g})_e^\ast \cdot \eta_g),
\eq
called the \emph{left trivialization} of $T^\ast G$, and the induced left action takes the form
\bqn
g\cdot (h,\mu) = (gh,\mu), \qquad g,h \in G, \, \mu \in \g^\ast.
\eqn
Consequently, the decomposition of $T^\ast G$ into orbit types  of this action is given by the one of $G$ and 
\bqn 
(T^\ast G)(H)=T^\ast(G(H)),
\eqn
$H$ being an arbitrary closed subgroup of $G$. On the other hand,  the  momentum map reads $\Jbb(g,\mu)=\Ad^\ast_{g^{-1}} \mu$, since with $\mu=(L_g)^\ast_e  \cdot \eta_g$ one computes for $X \in \g$
\bqn 
\Jbb(g,\mu)(X) = \Jbb(\eta_g) (X)= (L_{g^{-1}})^\ast_g \, \mu (\widetilde X_g)=  \mu ( (L_{g^{-1}})_{\ast,g} \widetilde X_g)=\mu \Big ( \frac d{dt} ({g^{-1}}  \e{tX} g )_{|t=0}\Big )=\mu(\Ad(g^{-1}) X),
\eqn
compare \cite[Example 4.5.5]{ortega-ratiu}.
\end{example}

Let us now apply these general results to the situation of this paper. Thus, let $\Xbf=T^\ast M$, where   $M$ is a connected compact boundary-less Riemannian manifold of dimension $n$, carrying an isometric  effective action of a compact connected 
Lie group $G$. In all what follows, the principal isotropy type of the action will be denoted by $(H)$, $H$ being a closed subgroup of $G$,  and the dimension of the principal orbits in $M$ by $\kappa$. Furthermore, we shall always assume that $\kappa<n$. $T^\ast M$ constitutes a Hamiltonian $G$-space when endowed with the canonical symplectic structure and the $G$-action induced from the smooth action on $M$, and one has
\begin{equation}
\Omega=\Jbb^{-1}(\{0\})=\bigsqcup_{x\in M}\textrm{Ann}\,T_x(G\cdot x)\label{eq:omega},
\end{equation}
where $\textrm{Ann}\,V_x\subset T^\ast_x M$ denotes the annihilator of a  subspace $V_x \subset T_xM$. 
%Clearly, as soon as there are two orbits $G\cdot x$, $G\cdot x'$ in $M$ of different dimensions, their annihilators $\textrm{Ann}\,T_x(G\cdot x)$ and $\textrm{Ann}\,T_x(G\cdot x')$ have different dimensions, so that $\Omega$ is not a  vector bundle in that case. 
Further, let 
\begin{align*}
M_{\textrm{reg}}&:=M(H), \qquad 
\Omega_{\textrm{reg}}:=\Omega \cap (T^{*}M)(H),
\end{align*}
where $M(H)$ and $(T^{*}M)(H)$ denote the union of orbits of type $(H)$ in $M$ and $T^{*}M$, respectively. By the principal orbit theorem, $M_{\textrm{reg}}$ is open  in $M$, hence $M_{\textrm{reg}}$ is a smooth submanifold.  We then define
 $$\widetilde{M}_{\textrm{reg}}:=M_{\textrm{reg}}/G.$$
$\widetilde{M}_{\textrm{reg}}$   is a smooth boundary-less manifold, since $G$ acts on ${M}_{\textrm{reg}}$ with only one orbit type and $M_{\textrm{reg}}$ is open in $M$.  Moreover, because the Riemannian metric on $M$ is $G$-invariant, it induces a Riemannian metric on $\widetilde{M}_{\textrm{reg}}$. On the other hand, by symplectic reduction  $\Omega_{\textrm{reg}}$ is a smooth submanifold of $T^{*}M$, and  the quotient
$$
\widetilde{\Omega}_{\textrm{reg}}:= \Omega_{\textrm{reg}}/G
$$
possesses a unique differentiable structure such that the projection $\pi: \Omega_{\textrm{reg}} \rightarrow \widetilde{\Omega}_{\textrm{reg}}$ is a surjective submersion. Furthermore, there exists a unique symplectic form $\widetilde{\omega}$ on $\widetilde{\Omega}_{\textrm{reg}}$ such that $
\iota^\ast \omega = \pi^\ast \widetilde{\omega}$,
where $\iota:  \Omega_{\textrm{reg}} \hookrightarrow T^{*}M$ denotes the inclusion and $\omega$  the canonical symplectic form on $T^{*}M$. In addition, by co-tangent bundle reduction the two $2(n-\kappa)$-dimensional symplectic manifolds 
\bq
\label{eq:isomorphic}
(T^*M_\text{reg}\cap \Omega)/G \, \simeq \,  T^*\widetilde{M}_\text{reg}
\eq
 are canonically symplectomorphic. In case that $G$ acts on $M$ only with finite isotropy groups, $\widetilde M:=M/G$ is an orbifold, and  the relation above constitutes the quotient presentation of the co-tangent bundle of $\widetilde M$ as an orbifold. % \cite{kordyukov12}.

\subsection{Generalized equivariant semiclassical Weyl law}
Let  $M$ be a  compact  Riemannian manifold of dimension $n$,  
 and denote by $\Psi_h^{m}(M)$ the $\C$-linear space of all semiclassical pseudodifferential operators on $M$ of order $m$, and by $S^m(M)$ the corresponding space of symbols, where $m \in \Z,\; h\in(0,1]$. In what follows, we shall write $$\Psi^{-\infty}_h(M)= \bigcap_{m\in \Z} \Psi^{m}_h(M),\qquad \Psi^{m}(M):=\Psi_1^{m}(M),$$ the latter being the linear space of ordinary pseudodifferential operators on $M$ of order $m$.  Similarly,  we write
 $\mathrm{Op}:=\mathrm{Op}_1$ for the ordinary quantization of non-semiclassical symbol functions.

The main result from Part I is a generalized equivariant semiclassical Weyl law that will be crucial in our study of equivariant quantum ergodicity. To state it, assume that $M$ carries an isometric effective action of a compact connected Lie group $G$ with principal orbits of dimension $\kappa<n$.  Let    \eqref{eq:15.08.2015} be  a Schr\"odinger operator on $M$ with real-valued, smooth, $G$-invariant potential  and Hamiltonian function \eqref{eq:hamiltonianfunction}. Consider further  the Peter-Weyl decomposition \ref{eq:PW} of the {left regular representation} of $G$    on $\L^2(M)$. Since the operator $P(h)$ commutes with the left-regular $G$-representation on $\L^2(M)$,    $P(h)_\chi=P(h)\circ T_\chi=T_\chi\circ P(h)$.  We then have the following\footnote{ For the precise definition of all relevant measures, the reader is referred to Section 2.4 of Part I.}  

 \begin{thm}[{\bf Generalized equivariant semiclassical Weyl law,  \cite[Theorem 4.1]{kuester-ramacher15a}}]
\label{thm:weyl2} 
Let $\beta\in \big(0,\frac{1}{2\kappa+4}\big)$ and choose an operator $B\in \Psi_{h}^0(M)\subset \B(\L^2(M))$ with principal symbol represented by $b\in S^0(M)$ and a semiclassical character family $\{\Wh\}_{h\in(0,1]}$ with growth rate $\vartheta<\frac{1-(2\kappa+4)\beta}{2\kappa+3}$.  Write
\[
J(h):=\big\{j\in \N:E_j(h)\in[c,c+h^\beta],\; \chi_j(h) \in \Wh\big\},
\]
where $\chi_j(h)\in \widehat G$ is defined by $u_j(h)\in \L^2_{\chi_j(h)}(M)$. Then, one has in the semiclassical limit $h\to 0$ \begin{align*}
\begin{split}
\frac{(2\pi)^{n-\kappa} h^{n-\kappa-\beta}}{\#\Wh}\sum_{J(h)}\frac{\langle Bu_{j}(h),u_{j}(h)\rangle_{\L^2(M)}}{d_{\chi_j(h)}\,[ \pi_{\chi_j(h)}|_{H}:\mathds{1}]}&=\intop_{{\Sigma}_c\cap \,\Omega_{\text{reg}}}b \, \frac {\d{\mu}_c}{\vol_\O}+\; \mathrm{O}\Big(h^{\beta}+h^{\frac{1-(2\kappa+3)\vartheta}{2\kappa +4}-\beta}\left (\log h^{-1}\right)^{{\Lambda^{}}-1}\Big).
\end{split}
\end{align*}
\end{thm}
\begin{rem}\label{rem:0304}
The integral in the leading term  can be written as   $\intop_{\widetilde{\Sigma}_c}\widetilde {\eklm{b}}_G\, \d\widetilde{\Sigma}_c$. In case that $\widetilde M$ is an orbifold, it  is given by an integral over the orbifold bundle $S^\ast_{\widetilde p,c}(\widetilde M):=\mklm{(x,\xi) \in T^\ast \widetilde M: \widetilde p(x,\xi)=c}$, compare Remark 4.2 of Part I.
\end{rem}

 The proof of Theorem \ref{thm:weyl2}  relies on a semiclassical calculus for $h$-dependent functions developed in \cite{kuester15}, and the description of  the asymptotic behavior of  certain oscillatory integrals that are locally of the form 
\begin{equation*}
I(\mu)=\intop_{T^{*}U}\int_{G}e^{i\mu\Phi(x,\xi,g)}a_\mu(x,\xi,g)\, dg\, d\left(T^{*}U\right)(x,\xi), \qquad \mu \to +\infty, 
\end{equation*}
where $(\gamma,U)$ denotes a local chart, $dg$ normalized Haar measure on $G$,  $d(T^\ast U)$ the canonical volume form on $T^\ast U$,  $a_\mu\in \CT(T^{*}U\times G)$ is an amplitude that might depend on the parameter  $\mu{>0}$ such that $(x,\xi,g) \in \supp a_\mu$ implies $g\cdot x \in U$, and 
\bq\Phi(x,\xi,g):=\left\langle \gamma(x)-\gamma(g\cdot x ),\xi\right\rangle.
\label{eq:phase} 
\eq
The major difficulty here resides in the fact that, unless the $G$-action on $T^\ast M$ is free,  the critical set of the phase function $\Phi$ is not a smooth manifold. The stationary phase theorem can therefore not immediately be applied to the integrals  $I(\mu)$. Nevertheless, it was shown in \cite{ramacher10, ramacher15a}   that by constructing  a partial desingularization  of the critical  set, and applying the stationary phase theorem in the resolution space, an asymptotic description of $I(\mu)$ can be obtained.

\section{Symmetry-reduced classical ergodicity}\label{sec:classicalergod}

We begin now with our study of ergodicity, and first turn to the examination of classical ergodicity in the presence of symmetries within the framework of symplectic reduction. As we already mentioned, the latter is based on the fundamental fact that the  presence of conserved quantities or first integrals of motion leads to the elimination of variables, and reduces the given configuration space with its symmetries to a lower-dimensional one, in which the degeneracies and the conserved quantitites have been eliminated. In particular, the Hamiltonian flows associated to $G$-invariant Hamiltonians give rise to corresponding reduced Hamiltonian flows on the different symplectic strata of the reduction. Therefore, the concept of ergodicity can be studied naturally in the context of symplectic reduction, leading to  a symmetry-reduced notion of ergodicity. 

Recall that, in general, a measure-preserving transformation $T:\Xbf \to \Xbf$ on a finite measure space $(\Xbf,\mu)$ is called \emph{ergodic} if $T^{-1}(A)=A$ implies $\mu(A)\in\{0,\mu(\Xbf)\}$ for every measurable set $A\subset \Xbf$. Consider now a connected, symplectic manifold $(\Xbf,\omega)$ with a global Hamiltonian action of a Lie group $G$, and let $\Jbb:\Xbf  \to  \g^\ast,\, \Jbb(\eta)(X)=  \Jbb_X(\eta)$ be the corresponding momentum map. As already noted in Section \ref{sec:2.3}, for each $X \in \g$ the function $\Jbb_X$ is a {conserved quantity}  for any $G$-invariant function $p \in \Cinft(\Xbf,\R)$, so that $\mklm{\Jbb_X, p}=0$. This implies that for any value $\mu$ of $\Jbb$, the fiber $\Jbb^{-1}(\{\mu\})$ is invariant under the Hamiltonian flow of $p$, which means that  $\Jbb$ fulfills \emph{Noether's condition}. In particular, if $c \in \R$ is a regular value of $p$ and $\Sigma_c:=p^{-1}(\{c\})$, the pre-image under $\Jbb$ of any open proper subset
%\footnote{\bf Clarify existence! For the geodesic flow, this should be clear since $\Jbb(S^\ast M) = \Jbb(T^\ast M)$.}
 in $\Jbb(\Sigma_c)$  will be an open proper subset in $\Sigma_c$ that is invariant under the Hamiltonian flow of $p$, so the latter cannot be ergodic with respect to the induced Liouville measure on $\Sigma_c$, unless $G$ is trivial. 

Let now $p$ and $\mu$ be fixed, $K\subset G$ an isotropy group  of the $G$-action on $\Xbf$, and    $\eta\in \Jbb^{-1}(\{\mu\})$. With the notation as in Section \ref{sec:2.3}, let  $c \in \R$, and put $\widetilde{\Sigma}_{\mu,c}^{(K)}:=(\widetilde{p}_\mu^{(K)})^{-1} (\{c\})$. Let $\widetilde{g}$  be a Riemannian metric on $\widetilde{\Omega}_\mu^{(K)}$ and $\mathcal{J}:T\widetilde{\Omega}_\mu^{(K)} \rightarrow T\widetilde{\Omega}_\mu^{(K)}$ the almost complex structure determined by $\widetilde{\omega}_\mu^{(K)}$ and $\widetilde{g}$, so that $(\widetilde{\Omega}_\mu^{(K)},\mathcal{J},\widetilde{g})$ becomes an almost Hermitian manifold. We then make the following  

\begin{assumption}\label{assumption:1}
 $c$ is a regular value of  $\widetilde{p}_\mu^{(K)}$.
\end{assumption}

Note that this assumption is implied by the condition  
%\footnote{\bf Clarify in as far these assumptions are equivalent.}
 that for all $\xi \in \Jbb^{-1}(\{\mu\}) \cap G_\mu \cdot \Xbf^\eta_K \cap \Sigma_c$ one has 
\[
H_p(\xi) \notin \g_\mu\cdot \xi,
\]
where $\g_\mu$ denotes the Lie algebra of $G_\mu$. Indeed, assume that there exists some $[\xi]\in \widetilde{\Sigma}_{\mu,c}^{(K)}$ such that  $\grad \widetilde{p}_\mu^{(K)}([\xi])=0$. Since 
\begin{align*}
\widetilde{\omega}_\mu^{(K)}(\sgrad \widetilde{p}_\mu^{(K)}, \X)=d\widetilde{p}_\mu^{(K)}(\X)=\widetilde{g}(\grad \widetilde{p}_\mu^{(K)}, \X),
\end{align*}
we infer that $H_{\widetilde{p}_\mu^{(K)}}([\xi])=\sgrad \widetilde{p}_\mu^{(K)}([\xi])=0$, which means that $[\xi]\in \widetilde{\Sigma}_{\mu,c}^{(K)}$  is a stationary point for the reduced flow, so that 
$
\widetilde{\varphi}_t^\mu([\xi]) = [\xi]$ for all $t 
\in \R$.  
By \eqref{eq:11.06.2015}, this is equivalent to 
\[
\pi_\mu^{(K)} \circ \varphi_t \circ \iota_\mu^{(K)}(\xi')= \widetilde{\varphi}_t^\mu([\xi]) \qquad \forall \, t \in \R,\; \xi' \in G_\mu \cdot \xi,
\]
which in turn is equivalent to
$
\varphi_t \circ \iota_\mu^{(K)}(\xi') \in G_\mu \cdot \xi'
$. 
Thus, there exists a $G_\mu$-orbit in $ \Jbb^{-1}(\{\mu\}) \cap G_\mu \cdot \Xbf^\eta_K\cap \Sigma_c$ which is invariant under $\varphi_t$. In particular one has
$
H_p(\xi') \in \g_\mu \cdot \xi' $ for all $\xi' \in  G_\mu \cdot \xi.
$ 

Assumption \ref{assumption:1} ensures that $\widetilde{\Sigma}_{\mu,c}^{(K)}$ is a smooth submanifold of $\widetilde{\Omega}_\mu^{(K)}$. Equipping $\widetilde{\Omega}_\mu^{(K)}$ with the symplectic volume form defined by the unique symplectic form on $\widetilde{\Omega}_\mu^{(K)}$ described in Section \ref{sec:2.3}, there is a unique induced hypersurface measure $\nu_{\mu,c}^{(K)}$ on $\widetilde{\Sigma}_{\mu,c}^{(K)}$, see Lemma A8 of Part I.  Moreover,  $\nu_{\mu,c}^{(K)}$ is invariant under the reduced flow $\widetilde{\varphi}_t^\mu$, since the latter constitutes a symplectomorphism due to Cartan's homotopy formula. Suppose now that the hypersurface $\widetilde{\Sigma}_{\mu,c}^{(K)}$ has finite volume with respect to the measure $\nu_{\mu,c}^{(K)}$. It is then natural to make the following

\begin{definition}
The reduced flow $\widetilde{\varphi}_t^\mu$ is called \emph{ergodic on $\widetilde{\Sigma}_{\mu,c}^{(K)}$} if for any measurable subset $E\subset  \widetilde{\Sigma}_{\mu,c}^{(K)}$  with $\widetilde{\varphi}_t^\mu(E)=E$ one has
\[
\nu_{\mu,c}^{(K)}(E)=0 \qquad \text{or} \qquad \nu_{\mu,c}^{(K)}(E)= \nu_{\mu,c}^{(K)}(\widetilde{\Sigma}_{\mu,c}^{(K)}). 
\]
\end{definition}

We can now formulate 

\begin{thm}[\bf Symmetry-reduced mean ergodic theorem]\label{thm:meanergod}
Let Assumption \ref{assumption:1} above be fulfilled, and suppose that $\widetilde{\Sigma}_{\mu,c}^{(K)}$ has finite volume with respect to its hypersurface measure $\nu_{\mu,c}^{(K)}$, and that the reduced flow $\widetilde{\varphi}_t^\mu$ is ergodic on $\widetilde{\Sigma}_{\mu,c}^{(K)}$. Then, for each $f\in \L^2\big(\widetilde{\Sigma}_{\mu,c}^{(K)},d\nu_{\mu,c}^{(K)}\big)$ we have
\[
 \left\langle f\right\rangle _{T} \quad \overset{T\to\infty}{\longrightarrow} \quad \frac 1 {\nu_{\mu,c}^{(K)}(\widetilde{\Sigma}_{\mu,c}^{(K)})}\intop_{\widetilde{\Sigma}_{\mu,c}^{(K)}}f\, d\nu_{\mu,c}^{(K)}
\]
with respect to the norm topology of $\L^2\big(\widetilde{\Sigma}_{\mu,c}^{(K)},d\nu_{\mu,c}^{(K)}\big)$,
where
\[
\left\langle f\right\rangle _{T}([\mu]):=\frac{1}{T}\int_{0}^{T}f\left(\widetilde{\varphi}_t^\mu([\mu])\right)dt,\quad[\mu]\in\widetilde{\Sigma}_{\mu,c}^{(K)}.
\]
\end{thm}
\begin{proof}
The proof is completely analogous to the existing proofs of the classical mean ergodic theorem, compare e.g. \cite[Theorem 15.1]{zworski}.
\end{proof}

In all what follows, we shall apply the general results outlined above to the case where $\Xbf=T^*M$ with $M$ and $G$ as in the introduction, $\mu=0$, $K=H$ is given by a principal isotropy group, and $p$ is the Hamiltonian function (\ref{eq:hamiltonianfunction}). We shall then use the simpler notation 
\[
\widetilde{\Omega}_\text{reg} = \widetilde{\Omega}_0^{(H)},\quad
\widetilde{\varphi}_t = \widetilde{\varphi}^0_t, \quad
\widetilde{\Sigma}_c = \widetilde{\Sigma}_{0,c}^{(H)}, \quad
d\widetilde{\Sigma}_c  = d\nu_{0,c}^{(H)}, \quad
\widetilde{p} = \widetilde{p}_0^{(H)}.
\]
As a special case of Theorem \ref{thm:meanergod} we  get the following 
\begin{thm}
\label{thm:classicalergod} Suppose that the reduced flow $\widetilde{\varphi}_{t}$ is ergodic on $\big(\widetilde{\Sigma}_{c},d\widetilde{\Sigma}_c\big)$.
Then for each $f\in \L^2\big(\widetilde{\Sigma}_{c},d\widetilde{\Sigma}_c\big)$,
\[
\lim_{T\to\infty}\int_{\widetilde{\Sigma}_{c}}\Big(\left\langle f\right\rangle _{T}-\fint_{\widetilde{\Sigma}_{c}}f\, d\widetilde{\Sigma}_c\Big)^{2}d\widetilde{\Sigma}_c=0.
\]
\end{thm}

\begin{rem}
Note that if $\widetilde M$ is an orbifold, the ergodicity of the reduced flow $\widetilde{\varphi}_{t}$  on $\big(\widetilde{\Sigma}_{c},d\widetilde{\Sigma}_c\big)$ is equivalent to the ergodicity of the corresponding Hamiltonian flow on the orbifold bundle $S^\ast_{\widetilde p,c} (\widetilde M)=\mklm{(x,\xi) \in T^\ast (\widetilde M): \widetilde p(x,\xi)=c}$ with respect to Liouville measure. 
\end{rem}

\label{subsec:symred}
Next, we examine the relation between classical time evolution and symmetry reduction.  Let $a\in \Cinft (T^*M)$. For a $G$-equivariant diffeomorphism $\Phi:T^*M\to T^*M$, we have \[
\left<a\circ \Phi\right>_G(\eta)=\int_G a(\Phi(g\cdot \eta))\,dg=\int_G a(g\cdot \Phi(\eta))\,dg = \left<a\right>_G(\Phi(\eta)), 
\]
so that $\left<a\circ \Phi\right>_G=\left<a\right>_G\circ \Phi$ and consequently $(\left<a\circ\Phi\right>_G)^{\widetilde{\; }}=(\left<a\right>_G\circ \Phi)^{\widetilde{\; }}$ holds. Now, we apply this result to the case $\Phi=\varphi_t$, where $\varphi_t$ is the Hamiltonian flow associated to the symbol function $p$ of the Schr\"odinger operator. If $i:\Omega_{\textrm{reg}}\hookrightarrow T^{*}M$ denotes the inclusion and $\pi:\Omega_{\textrm{reg}}\to\widetilde{\Omega}_{\textrm{reg}}$
the projection onto the $G$-orbit space,  we have $\pi\circ\varphi_{t}\circ i=\widetilde{\varphi}_{t}\circ\pi$.
Since 
\begin{eqnarray*}
\left<a\right>_G\circ\varphi_{t}\circ i =  \left(\left<a\right>_G\circ\varphi_{t}\right)^{\widetilde{\;}}\circ\pi,\qquad \left<a\right>_G\circ i  =  \widetilde{\left<a\right>}_G\circ\pi,\qquad\text{we get}
\end{eqnarray*}
\[
\widetilde{\left<a\right>}_G\circ\widetilde{\varphi}_{t}\circ\pi = \widetilde{\left<a\right>}_G\circ\pi\circ\varphi_{t}\circ i = \left<a\right>_G\circ i\circ\varphi_{t}\circ i=\left<a\right>_G\circ\varphi_{t}\circ i = \left(\left<a\right>_G\circ\varphi_{t}\right)^{\widetilde{\;}}\circ\pi,
\]
where we used that $i\circ\varphi_{t}\circ i=\varphi_{t}\circ i$. Since $\pi$ is surjective, we have shown

\begin{lem}
\label{lem:evolvred}Let $a\in \Cinft (T^*M)$ and $\varphi_{t}$ be the flow on $T^{*}M$
associated to the Hamiltonian $p$. Let $\widetilde{\varphi}_{t}$
be the reduced flow on $\widetilde{\Omega}_{\textrm{reg}}$ 
associated to $\widetilde{p}$. Then time evolution and reduction commute:
\[
\left(\left<a\right>_G\circ\varphi_{t}\right)^{\widetilde{\;}}=\widetilde{\left<a\right>}_G\circ\widetilde{\varphi}_{t}.
\]
\end{lem}

\section{Equivariant quantum ergodicity}

 We are now ready to formulate our first quantum ergodic theorem in a  symmetry-reduced context. Let the notation be as in the previous sections.

\begin{thm}[{\bf Integrated equivariant quantum ergodicity}]
\label{thm:ergod1} Suppose that the reduced flow $\widetilde{\varphi}_{t}$ corresponding to the reduced Hamiltonian function $\widetilde{p}$ is ergodic on $\widetilde{\Sigma}_c=\widetilde{p}^{-1}(\{c\})$. Let $A\in\Psi_h^{0}(M)$ be a semiclassical pseudodifferential operator with principal symbol $\sigma(A)=[a]$, where $a\in S^0(M)$ is independent of $h$. For a number $\beta\in\big(0,\frac{1}{2\kappa+4}\big)$ and a semiclassical character family $\{\Wh\}_{h\in(0,1]}$ with growth rate $\vartheta<\frac{1-(2\kappa+4)\beta}{2\kappa+3}$ set
\[
J(h):=\{j\in \N:E_j(h)\in[c,c+h^\beta],\; \chi_j(h) \in \Wh\},
\]
where $\chi_j(h)$ is defined by $u_j(h)\in \L^2_{\chi_j(h)}(M)$. Then, one has
\begin{equation}
\lim_{h\to0}\;\frac{h^{n-\kappa-\beta}}{\#\Wh}\sum_{J(h)}\frac{1}{d_{\chi_j(h)}[\pi_{\chi_j(h)}|_{H}:\mathds{1}]}\;\Big|\left\langle Au_{j}(h),u_{j}(h)\right\rangle_{\L^2(M)} -\fint_{\Sigma_c\cap \Omega_\text{reg}}a\, \frac{d\mu_c}{\vol_\O}\Big|^{2}\;=\;0,\label{eq:ergod1}
\end{equation}
\end{thm}

\begin{rem}
Again, the integral in (\ref{eq:ergod1}) can also we written as $\fint_{\widetilde{\Sigma}_c}\widetilde{\eklm{a}}_G \d \widetilde{\Sigma}_c$, and if $\widetilde M$ is an orbifold, it can be written as an integral over $S^\ast_{\widetilde p,c} (\widetilde M)$, compare Remark \ref{rem:0304}.
\end{rem}

\begin{proof}
We shall adapt the existing proofs of quantum ergodicity to the equivariant situation, following mainly \cite[Theorem 15.4]{zworski}. 
Let us write $u_j(h)=u_j$ and $E_j(h)=E_j$, and $\rho\in \CT(\mathbb{R},[0,1])$ be
such that $\rho\equiv1$ in a neighbourhood of $c$. Without loss of generality we may assume for the rest of the proof that $h$ is small enough so that $\rho\equiv1$ on $[c,c+h^\beta]$. Set
\begin{equation}
B:=\rho(P(h))\circ \left(A-\alpha\,\mathds{1}_{\L^2(M)}\right)\label{eq:B}, \qquad \alpha:=\fint_{\Sigma_c\cap \Omega_\text{reg}}a\, \frac{d\mu_c}{\vol_\O}=\fint_{\widetilde{\Sigma}_c}\widetilde{\left<a\right>}_G \d \widetilde{\Sigma}_c,
\end{equation}
where $\widetilde{\left<a\right>}_G$ was defined in (\ref{eq:orbitalintegral}). Note that by the semiclassical calculus  we have $B\in\Psi_h^{-\infty}(M)$. Furthermore, 
\begin{align*}
\sigma(B)&=\left(\rho\circ\sigma(P(h))\right)\sigma\left(A-\alpha\,\mathds{1}_{\L^2(M)}\right)=\left[ \left(\rho\circ p\right) \left(a-\alpha\, 1_{T^*M}\right) \right] \in S^{-\infty}(M)\slash hS^{-\infty}(M),
\end{align*}
see  Section 2.1 of Part I.  Let us write $b:=\left(\rho\circ p\right) \left(a-\alpha\, 1_{T^*M}\right)$, so that $\sigma(B)=[b]$. Clearly,
\begin{equation}
\widetilde{\left<b\right>}_G=\left(\left(\rho\circ p\right)\left(\left<a\right>_G -\alpha\, 1_{T^*M}\right)\right)^{\widetilde{\;}}=\left(\rho\circ \widetilde{p}\right)\left(\left<a\right>_G-\alpha\, 1_{T^*M}\right)^{\widetilde{\;}}=\left(\rho\circ \widetilde{p}\right) \left(\widetilde{\left<a\right>}_G-\alpha\,1_{\widetilde{\Omega}_\text{reg}}\right).\label{eq:sigmatildeB}
\end{equation}
Next, we define
\begin{equation}
\Lcal(h):= \frac{(2\pi)^{n-\kappa} h^{n-\kappa-\beta}}{\#\Wh}\sum_{J(h)}\frac{1}{d_{\chi_j(h)}[\pi_{\chi_j(h)}|_{H}:\mathds{1}]}\;\Big|\left\langle Bu_{j},u_{j}\right\rangle_{\L^2(M)}\Big|^{2}.\label{eq:epsilon}
\end{equation}
By the spectral theorem,  $\rho(P(h))u_{j}=u_{j}$ for $E_{j}\in[c,c+h^\beta]$, since $\rho\equiv1$ on $[c,c+h^\beta]$. Taking into account the self-adjointness of  $\rho(P(h))$ one sees that for $E_{j}\in[c,c+h^\beta]$
\begin{equation}
\left\langle Bu_{j},u_{j}\right\rangle_{\L^2(M)} =\left\langle Au_{j},u_{j}\right\rangle_{\L^2(M)} -\alpha.\label{eq:diff2}
\end{equation}
Consequently, we will be  done with the proof if we can show that 
\begin{equation}
\lim_{h\to0}\Lcal(h)=0. \label{eq:showepsilon}
\end{equation}
In order to do so, one considers the time evolution operator
\begin{equation*}
F^{h}(t):\L^2(M)\to \L^2(M), \quad F^{h}(t):=e^{-itP(h)/h}, \qquad t \in \R,
\end{equation*}
which by Stone's theorem \cite[Section XI.13]{yosida} is a well-defined bounded operator. One then sets
\[
B(t) :=  F^{h}(t)^{-1}BF^{h}(t).
\]
In order to make use of classical ergodicity, one notes that the expectation value
\begin{align*}
 \nonumber\left\langle Bu_{j},u_{j}\right\rangle_{\L^2(M)}  & =  \left\langle Be^{-itE_{j}/h}u_{j},e^{-itE_{j}/h}u_{j}\right\rangle_{\L^2(M)} = \left\langle Be^{-itP(h)/h}u_{j},e^{-itP(h)/h}u_{j}\right\rangle_{\L^2(M)} \\
 & =  \left\langle B(t)u_{j},u_{j}\right\rangle_{\L^2(M)},  \qquad \qquad t\in[0,T],
\end{align*}
is actually time-independent. This implies for each $T>0$
\[
\left\langle Bu_{j},u_{j}\right\rangle_{\L^2(M)} =\left\langle \left\langle B\right\rangle_{T}u_{j},u_{j}\right\rangle_{\L^2(M)} ,\]
where we set $\left\langle B\right\rangle_{T}=\frac{1}{T}\int_{0}^{T}B(t)dt\in\Psi_h^{-\infty}(M)$.
Taking into account $\left\Vert u_{j}\right\Vert _{\L^2(M)}^{2}=1$ and the Cauchy-Schwarz inequality one arrives at
\begin{equation*}
\left|\left\langle Bu_{j},u_{j}\right\rangle_{\L^2(M)} \right|^{2}\leq\left\Vert \left\langle B\right\rangle _{T}u_{j}\right\Vert _{\L^2(M)}^{2}.
\end{equation*}
We therefore conclude from (\ref{eq:epsilon}) for each $T>0$ that
\begin{equation}
\Lcal({h})\leq\frac{(2\pi )^{n-\kappa} h^{n-\kappa-\beta}}{\#\Wh}\sum_{J(h)}\frac{1}{d_{\chi_j(h)}[\pi_{\chi_j(h)}|_{H}:\mathds{1}]}\;\left\langle \left\langle B^{*}\right\rangle _{T}\left\langle B\right\rangle _{T}u_{j},u_{j}\right\rangle_{\L^2(M)}. \label{eq:SMALLEPSILON}
\end{equation}
Next, let $\overline{B}(t)$ be an element in $\Psi_h^{-\infty}(M)$ with principal symbol $\sigma(B)\circ\varphi_{t}$. 
By the weak Egorov theorem \cite[Theorem 15.2]{zworski} one has
\begin{equation*}
\left\Vert B(t)-\overline{B}(t)\right\Vert _{\B\left(\L^2(M)\right)}=\mathrm{O}(h)\quad\textrm{uniformly for }\, t\in[0,T],
\end{equation*}
which implies
\begin{equation}
\left\langle B\right\rangle _{T}=\left\langle \overline{B}\right\rangle _{T}+\mathrm{O}_{\B\left(\L^2(M)\right)}^{T}(h).\label{eq:both}
\end{equation}
From the definition of $\overline{B}$ we get
\[
\sigma\left(\left\langle \overline{B}\right\rangle _{T}\right)=\Bigg[\frac{1}{T}\int_{0}^{T}b\circ\varphi_{t}\, dt\Bigg].
\]
Furthermore, the symbol map is a $*$-algebra homomorphism from $\Psi_h^{-\infty}(M)$ to $S^{-\infty}(M)/hS^{-\infty}(M)$,  with involution given by the adjoint operation and pointwise complex conjugation, respectively. That leads to
\[
\sigma\Big(\big\langle \overline{B}^*\big\rangle _{T}\big\langle \overline{B}\big\rangle _{T}\Big)=\Bigg[\, \Big|\frac{1}{T}\int_{0}^{T}b\circ\varphi_{t}\, dt\Big|^2\, \Bigg].
\]
Now,  note that by Lemma \ref{lem:evolvred}
\begin{equation}
 \Bigg<\frac{1}{T}\int_{0}^{T}b\circ\varphi_{t}\, dt\Bigg>_G^{\widetilde{\;}} = \frac{1}{T}\int_{0}^{T}\left(\left<b\right>_G\circ\varphi_{t}\right)^{\widetilde{\;}}\, dt
 =  \frac{1}{T}\int_{0}^{T}\widetilde{\left<b\right>}_G\circ\widetilde{\varphi}_{t}\, dt
 =  \langle \widetilde{\left<b\right>}_G\rangle_{T}, \label{eq:redBsymb}
\end{equation}
which is where the transition from the flow $\varphi_t$ to the reduced flow $\widetilde{\varphi}_t$ takes place. We can then apply the generalized equivariant Weyl law, Theorem
\ref{thm:weyl2}, which  together with (\ref{eq:redBsymb}) yields 
\begin{align}
\begin{split}
\frac{(2\pi )^{n-\kappa} h^{n-\kappa-\beta}}{\#\Wh}\sum_{J(h)} &\frac{1}{d_{\chi_j(h)}[\pi_{\chi_j(h)}|_{H}:\mathds{1}]}\;\big\langle \big\langle \overline{B}^{*}\big\rangle _{T}\big\langle \overline{B}\big\rangle _{T}u_{j},u_{j}\big\rangle_{\L^2(M)} \\
=\int_{\widetilde{\Sigma}_c}&|\langle \widetilde{\left<b\right>}_G\rangle_{T}|^2\d\widetilde{\Sigma}_c 
+ \mathrm{O}\Big(h^\beta+h^{\frac{1-(2\kappa+3)\vartheta}{2\kappa +4}-\beta} (\log h^{-1})^{\Lambda-1}\Big).\label{eq:limnew}
\end{split}
\end{align}
From (\ref{eq:sigmatildeB}) we see that over $\widetilde{\Sigma}_c=\widetilde{p}^{-1}(\{c\})$ we have $
\widetilde{\left<b\right>}_G|_{\widetilde{\Sigma}_c}=\widetilde{\left<a\right>}_G|_{\widetilde{\Sigma}_c}-\alpha\cdot1_{\widetilde{\Sigma}_c}=:\widetilde{b}_c$. With (\ref{eq:SMALLEPSILON}), (\ref{eq:both}) and (\ref{eq:limnew}) we deduce for each $T>0$
\begin{align*}
\Lcal(h)&\leq \int_{\widetilde{\Sigma}_c}|\langle \widetilde{b}_c\rangle _{T}|^{2}\d\widetilde{\Sigma}_c 
+ \mathrm{O}\Big(h^\beta+h^{\frac{1-(2\kappa+3)\vartheta}{2\kappa +4}-\beta} (\log h^{-1})^{\Lambda-1}\Big) \\
&+ \left[\frac{h^{n-\kappa-\beta}}{\#\Wh}\sum_{J(h)}\frac{1}{d_{\chi_j(h)}[\pi_{\chi_j(h)}|_{H}:\mathds{1}]}\right]\cdot\mathrm{O}(h).
\end{align*}
By Theorem \ref{thm:weyl2}, the factor in front of the $\mathrm{O}(h)$-remainder is convergent and therefore bounded as $h\to 0$. Moreover, the number $\int_{\widetilde{\Sigma}_c}|\langle \widetilde{b}_c\rangle _{T}|^{2}\d\widetilde{\Sigma}_c $ is independent of $h$, as we assume that $a$ is independent of $h$. Thus, 
\begin{equation}
\limsup_{h\to0}\,\Lcal(h)\leq \int_{\widetilde{\Sigma}_c}|\langle \widetilde{b}_c\rangle _{T}|^{2}\d\widetilde{\Sigma}_c \qquad \forall \;T>0.\label{eq:limsup-1}
\end{equation}
This is now the point where symmetry-reduced classical ergodicity is used. Since $\widetilde{b}_c$ fulfills $\fint_{\widetilde{\Sigma}_c}\widetilde{b}_c\, d\widetilde{\Sigma}_c=0$,  Theorem \ref{thm:classicalergod} yields $\lim_{T\to\infty}\int_{\widetilde{\Sigma}_c}|\langle \widetilde{b}_c\rangle _{T}|^{2}d\widetilde{\Sigma}_c=0$.  Because the left hand side of (\ref{eq:limsup-1}) is independent of $T$, it follows that it must be zero, yielding (\ref{eq:showepsilon}).
\end{proof}
\begin{rem}\label{rem:endofsec}
Note that one could have still exhibited the Weyl law remainder estimate in (\ref{eq:limsup-1}). But since the rate of convergence in Theorem \ref{thm:classicalergod} is unknown in general, it is not possible to give a remainder estimate in Theorem \ref{thm:ergod1} with the methods employed here. Nevertheless, in certain dynamical situations, the rate could probably be made explicit.
\end{rem}
In the special case of a constant semiclassical character family, corresponding to the study of a single fixed isotypic component, we obtain as a direct consequence
\begin{thm}[{\bf Integrated equivariant quantum ergodicity for single isotypic components}]
\label{thm:ergod11} Suppose that the reduced flow $\widetilde{\varphi}_{t}$ corresponding to the reduced Hamiltonian function $\widetilde{p}$ is ergodic on $\widetilde{\Sigma}_c:=\widetilde{p}^{-1}(\{c\})$. Let $A\in\Psi_h^{0}(M)$ be a semiclassical pseudodifferential operator with principal symbol $\sigma(A)=[a]$, where $a\in S^0(M)$ is independent of $h$. Choose $\beta\in\big(0,\frac{1}{2\kappa+4}\big)$ and $\chi\in \widehat G$. Then, one has
\begin{equation}
\lim_{h\to0}\;h^{n-\kappa-\beta}\sum_{J^\chi (h)}\Big|\left\langle Au_{j}(h),u_{j}(h)\right\rangle_{\L^2(M)} -\fint_{\Sigma_c\cap \Omega_\text{reg}}a\, \frac{d\mu_c}{\vol_\O}\Big|^{2}\;=\;0,\label{eq:ergod11}
\end{equation}
where
\bq
\label{eq:24.08.2015}
J^\chi (h):=\big\{j\in \N:E_j(h)\in[c,c+h^\beta],\; u_j(h) \in \L^2_\chi(M)\big\}.
\eq
\end{thm}

\begin{rem}
\label{rem:23.11}
A  weaker version of Theorem \ref{thm:ergod11} can be proved with a less sharp energy localization in an interval $[r,s]$ with $r<s$ by the same methods employed here. In fact, under the additional assumption that the mean value $\alpha$ introduced in \eqref{eq:B} is the same for all $c \in [r,s]$ and all considered $c$ are regular values of $p$, the reduced flow being ergodic on each of the contemplated hypersurfaces $\widetilde \Sigma_c$, one can show that 
\begin{equation}
\label{eq:20.07.2015}
\lim_{h\to0}\;h^{n-\kappa}\sum_{\begin{array}{c}\scriptstyle 
j\in\mathbb{N}:\,u_{j}(h)\in \L_{\chi}^{2}(M), \\
\scriptstyle E_{j}(h)\in [r,s]\end{array}}\Big|\left\langle Au_{j}(h),u_{j}(h)\right\rangle_{\L^2(M)} -\fintop_{p^{-1}([r,s])\cap\Omega_\text{reg}}a\,\frac{d\Omega_\text{reg}}{\vol_\O}\Big|^{2}=0. 
\end{equation}
The proof of this relies on a corresponding semiclassical Weyl law for the interval $[r,s]$ and a single isotypic component, see Remark 4.4 of Part I.
The point is that for the weaker statement \eqref{eq:20.07.2015} a remainder estimate of order $\text{o}(h^{n-\kappa})$ is sufficient in Weyl's law, since the rate of convergence  in \eqref{eq:20.07.2015}  is the one of the leading term. Thus, in principle, this weaker result could have also been obtained using heat kernel methods as in \cite{donnelly78} or \cite{bruening-heintze79} adapted to the semiclassical setting, at least for the Laplacian. Nevertheless, for the stronger version proved in Theorem \ref{thm:ergod11}, remainder estimates of order $\mathrm{O}(h^{n-\kappa-\beta})$ in Weyl's law and in particular the results of  \cite{ramacher10} are necessary.
\end{rem}

In what follows, we shall use our previous results to prove our main result, a symmetry-reduced quantum ergodicity theorem for Schr\"odinger operators. \\

\begin{thm}[{\bf Equivariant quantum ergodicity for Schr\"odinger operators}]
\label{thm:ergod2} With the notation and assumptions as in Theorem \ref{thm:ergod1}, there is a $h_0\in (0,1]$ such that for each $h\in(0,h_0]$ we have a subset $\Lambda(h)\subset J(h)$ satisfying 
\begin{equation}
\label{eq:ergod20}
\lim_{h\to0}\frac{\#\Lambda(h)}{\#J(h)}=1
\end{equation}
such that for each semiclassical pseudodifferential operator $A\in\Psi_h^{0}(M)$ with principal symbol $\sigma(A)=[a]$, where $a$ is $h$-independent, the following holds. For all $\epsilon>0$ there is a $h_\epsilon \in (0,h_0]$ such that
\begin{equation}
\frac{1}{\sqrt{d_{\chi_j(h)}[\pi_{\chi_j(h)}|_{H}:\mathds{1}]}}\;\Big|\left\langle Au_j(h),u_j(h)\right\rangle_{\L^2(M)}-\fintop_{\Sigma_c\cap \Omega_\text{reg}}a\, \frac{d\mu_c}{\vol_\O}\Big|\;<\;\epsilon \qquad\forall\,j\in \Lambda(h), \; \forall\, h\in(0,h_\epsilon],\label{eq:ergod2}
\end{equation}
where the integral in (\ref{eq:ergod2}) equals $\fint_{\widetilde{\Sigma}_c}\widetilde{\eklm{a}}_G \d \widetilde{\Sigma}_c$.
\end{thm}
\begin{proof} 
Again, this proof is an adaptation of existing proofs like \cite[Theorem 15.5]{zworski} to the equivariant setting, only that we do not need the technical condition that the value of the integral $\fint_{\widetilde{\Sigma}_c}\widetilde{a}\, d\widetilde{\Sigma}_c $ must stay the same when varying $c$ in some interval, which slightly simplifies the proof.

Write $u_j(h)=u_j$ and $E_j(h)=E_j$.  By Theorem \ref{thm:weyl2} we can choose a $h_0\in(0,1]$ such that $J(h)\neq \emptyset$ for all $h\in(0,h_0]$, and suppose that $h \in (0,h_0]$. With the notation as in (\ref{eq:orbitalintegral}), we set for any smooth function $s$ on $T^*M$
$$\alpha(s):=\fint_{\widetilde{\Sigma}_c}\widetilde{\eklm{s}}_G\, d\widetilde{\Sigma}_c.$$
Let $\tau\in \CT(\mathbb{R},[0,1])$ be 
such that $\tau\equiv1$ in a neighbourhood of $c$. Without loss of generality, we assume for the rest of the proof that $h_0$ is small enough so that $\tau\equiv1$ on $[c,c+h_0^\beta]$. Now, for any operator $A$ as in the statement of the theorem set 
\[
{B}:=A-\alpha(a)\, \tau\left(P(h)\right).
\] 
By the semiclassical calculus  we know that the principal symbol of ${B}$ is given by $\sigma({B})=[b]$ with $b:=a-\alpha(a) \, \tau\circ p$. Clearly, $\alpha(b)=0$, since   $\tau\circ\widetilde{p}\equiv 1$ on $\widetilde{\Sigma}_c$. Let us now assume that  the statement of the theorem holds for all operators $A$ with $\alpha(a)=0$. Then, there is a sequence of subsets $\Lambda(h)$ of density 1 such that for all $\epsilon>0$ there is a $h_\epsilon \in (0,h_0]$ such that 
\begin{equation}
\frac{1}{\sqrt{d_{\chi_j(h)}[\pi_{\chi_j(h)}|_{H}:\mathds{1}]}}\;\left|\left\langle {B}u_j,u_j\right\rangle_{\L^2(M)} \right|<\epsilon \qquad \forall \, h\in(0,h_\epsilon], \quad  \forall \, j\in \Lambda(h).
 \label{eq:assume}
\end{equation}
Due to the choice of the function $\tau$ we have $\tau\left(P(h)\right)(u_j)=u_j$ for all $u_j$ with $E_j\in [c,c+h^\beta]$. Consequently,  (\ref{eq:assume}) implies that for all $\epsilon>0$ there is $h_\epsilon \in (0,h_0]$ such that\begin{equation*}
\frac{1}{\sqrt{d_{\chi_j(h)}[\pi_{\chi_j(h)}|_{H}:\mathds{1}]}}\;\left|\left\langle Au_j,u_j\right\rangle_{\L^2(M)} - \alpha (a) \right|<\epsilon \qquad \forall \, h\in(0,h_\epsilon], \quad  \forall \, j\in \Lambda(h), 
\end{equation*}
and we obtain the statement of the theorem for general $A$. We are therefore left with the task of proving \eqref{eq:assume} for arbitrary operators $B$ with $\alpha(b)=0$, and shall proceed in a similar fashion to parts 1\ - 5\ of the proof of \cite[Theorem 15.5]{zworski}, pointing out only the main arguments. By Theorem \ref{thm:ergod1} we have for fixed $B$
\begin{equation*}
\frac{h^{n-\kappa-\beta}}{\# \Wh} \sum_{j \in J(h)}\frac{\big|\big\langle Bu_{j},u_{j}\big\rangle_{\L^2(M)} \big|^{2}}{d_{\chi_j(h)}[\pi_{\chi_j(h)}|_{H}:\mathds{1}]}=:r(h) \to 0
\end{equation*}
as $h \to 0$. One then defines for $h \in (0,h_0]$ the $B$-dependent subsets 
$$\Lambda(h):=J(h)-\Gamma(h), \qquad \Gamma(h):=\Big\{j \in J(h): \frac{|\langle Bu_{j},u_{j}\rangle_{\L^2(M)}|^2}{d_{\chi_j(h)}[\pi_{\chi_j(h)}|_{H}:\mathds{1}]}\; \geq \sqrt {r(h)}\Big\}.
$$
Clearly,
\begin{align*}
\# \Gamma(h)\leq \sum_{j \in J(h)}  \frac{|\langle Bu_{j},u_{j}\rangle_{\L^2(M)}|^2}{d_{\chi_j(h)}[\pi_{\chi_j(h)}|_{H}:\mathds{1}] \, \sqrt {r(h)}}=\frac{\# \Wh \, \sqrt{r(h)}}{h^{n-\kappa-\beta}},
\end{align*}
and taking $B=\1$ in Theorem \ref{thm:weyl2} one computes
\begin{align*}
\frac{\#\Gamma(h)}{\# J(h)}&\leq \frac{\# \Wh \sqrt{r(h)}} {h^{n-\kappa-\beta}\sum_{j \in J(h)}({d_{\chi_j(h)}[\pi_{\chi_j(h)}|_{H}:\mathds{1}]})^{-1}}\\ & =\frac{\sqrt{r(h)}}{(2\pi)^{\kappa-n} \intop_{\widetilde{\Sigma}_c}\widetilde {\eklm{b}}_G\, \d\widetilde{\Sigma}_c +  \mathrm{O}\big(h^{\beta}+h^{\frac{1-(2\kappa+3)\vartheta}{2\kappa +4}-\beta}\left (\log h^{-1}\right)^{\Lambda-1}\big)} \longrightarrow  0.
\end{align*}
On the other hand,
\bqn 
\frac{1}{\sqrt{d_{\chi_j(h)}[\pi_{\chi_j(h)}|_{H}:\mathds{1}]}}\;\left|\left\langle {B}u_j,u_j\right\rangle_{\L^2(M)} \right|<r(h)^{1/4} \qquad \forall \, j \in \Lambda(h), 
\eqn
yielding  \eqref{eq:assume}  for  these particular   $\Lambda(h)$ and $B$.  

Consider now a family $\mklm{A_k}_{k \in \N}$ of semiclassical pseudodifferential operators in $\Psi_h^0(M)$ with principal symbols represented by $h$-independent symbol functions. By our previous considerations, for each $k$  there is a sequence of subsets $\Lambda_k(h)\subset J(h)$ such that \eqref{eq:ergod20} and \eqref{eq:ergod2} hold for each particular $A_k$ and $\Lambda_k(h)$.  
One then shows that for sufficiently small $h$ there is a sequence of subsets $\Lambda_\infty(h)\subset J(h)$ of density 1 such that $\Lambda_k(h)\subset \Lambda_\infty(h)$ for each $k$. Hence, the theorem is true for countable families of operators. To obtain it for all operators in $\Psi_h^{-\infty}(M)$, it suffices to find a sequence of operators $\mklm{A_k}_{k \in \N}$ which is dense in the set of operators in $\Psi_h^{-\infty}(M)$ whose principal symbol is represented by an $h$-independent symbol function, in the sense that for any given $A\in \Psi_h^{-\infty}(M)$ of the mentioned form and any $\epsilon >0$ there exists a $k$ such that 
\[
 \| A-A_k\|_{\L^2(M) \rightarrow \L^2(M)} < \epsilon, \qquad 
 \fint_{\widetilde{\Sigma}_c}\widetilde{\eklm{a-a_k}}_G\, d\widetilde{\Sigma}_c  < \epsilon
\]
for sufficiently small $h$. To find such a sequence $\mklm{A_k}_{k \in \N}\subset \Psi_h^{-\infty}(M)$, note that for two symbol functions $a$ and $b$ and the corresponding semiclassical quantizations $A$ and $B$, one has
 \begin{gather*}
 \| A-B\|_{\L^2(M) \rightarrow \L^2(M)}  \leq \| a-b \|_{\L^\infty(T^*M)} +C \sqrt h, \qquad 
 \fint_{\widetilde{\Sigma}_c}\widetilde{\eklm{a-b}}_G \, d\widetilde{\Sigma}_c  \leq C\| a-b\|_{\L^\infty(T^\ast M)}.
 \end{gather*}
Consequently, we only need to find a sequence of $h$-independent symbol functions that is dense in $S^{-\infty}(M)$ equipped with the $L^\infty$-norm. That such a sequence exists follows directly from the facts that $\CT (T^*M)$ is $L^\infty$-norm dense in the Banach space $\mathrm{C}_0(T^*M)\supset S^{-\infty}(M)$ of continuous functions vanishing at infinity, and that $\CT (T^*M)$ is separable. This proves the theorem for operators $A$ in $ \Psi_h^{-\infty}(M)$ with principal symbol represented by an $h$-independent symbol function. Finally, if $A\in \Psi_h^{0}(M)$ is a general  operator with principal symbol represented by an $h$-independent symbol function, one multiplies $A$ with the smoothing operator $\rho(P(h))$, where $\rho \in \CT(\R)$ equals $1$ near $c$. This completes the proof of the theorem.
\end{proof}
Again, in the special case that $\Wh=\{\chi\}$ for all $h\in (0,1]$ and some fixed $\chi\in \widehat G$, we obtain
\begin{thm}[{\bf Equivariant quantum ergodicity for Schr\"odinger operators and single isotypic components}]
\label{thm:ergod22}
With the notation and assumptions as in Theorem \ref{thm:ergod11},  let $\chi\in\widehat{G}$, $\beta\in\big(0,\frac{1}{2\kappa+4}\big)$ be fixed,  and let $J^\chi(h)$ be as in \eqref{eq:24.08.2015}. Then,  there is a $h_0\in (0,1]$ such that for each $h\in(0,h_0]$ we have a subset $\Lambda^\chi(h)\subset J^\chi(h)$ satisfying 
$\lim_{h\to0}\frac{\#\Lambda^\chi(h)}{\#J^\chi(h)}=1$
such that for each semiclassical pseudodifferential operator $A\in\Psi_h^{0}(M)$ with principal symbol $\sigma(A)=[a]$ the following holds. For all $\epsilon>0$ there is a $h_\epsilon \in (0,h_0]$ such that
\begin{equation*}
\Big|\left\langle Au_j(h),u_j(h)\right\rangle_{\L^2(M)}-\fintop_{\Sigma_c\cap \Omega_\text{reg}}a\, \frac{d\mu_c}{\vol_\O}\Big|<\epsilon \qquad\forall\,j\in \Lambda^\chi(h), \; \forall\, h\in(0,h_\epsilon].
\end{equation*}
\end{thm}

\section{Equivariant quantum limits for the Laplace-Beltrami operator}\label{sec:laplacian}

\subsection{Eigenfunctions of the Laplace-Beltrami operator} We shall now apply the semiclassical results from the previous section to study the distribution of eigenfunctions of the Laplace-Beltrami operator on a compact connected Riemannian $G$-manifold $M$ without boundary in the limit of large eigenvalues, $G$ being a compact connected Lie group acting isometrically and effectively on $M$, with principal orbits of dimension $\kappa<n=\dim M$.  Let $\Delta $ be the unique self-adjoint extension of the Laplace-Beltrami operator $ \breve \Delta$ on $M$, and choose an orthonormal basis $\mklm{u_j}_{j \in \N}$ of $\L^2(M)$ of  eigenfunctions of $- \Delta $ with corresponding eigenvalues $\mklm{E_j}_{j \in \N}$, repeated according to their multiplicity. Consider further the Schr\"odinger operator $P(h)$ given by \eqref{eq:15.08.2015} with  $V\equiv 0$ and principal symbol defined by the symbol function $p=\left\Vert \cdot\right\Vert _{T^{*}M}^{2}$. Clearly,   $P(h)=-h^2\Delta$,  and each $u_j$ is an eigenfunction of $P(h)$  with eigenvalue $E_j(h)=h^2 E_j$. Furthermore, under the identification $T^\ast M\simeq TM$ given by the Riemannian metric, the Hamiltonian flow $\varphi_t$ induced by $p$ corresponds to the geodesic flow of $M$.  Each $c>0$ is a regular value of $p$, and since $V\equiv0$ the dynamics of the reduced geodesic flow $\widetilde \varphi_t$ are equivalent on any two hypersurfaces $\widetilde\Sigma_c$ and  $\widetilde\Sigma_{c'}$. In the following, we shall therefore choose $c=1$ without loss of generality, and call the reduced geodesic flow \emph{ergodic} if it is ergodic on $\widetilde{\Sigma}_1=\widetilde{p}^{-1}(\{1\})$. 
The following construction will allow a simpler formulation of the subsequent theorems.
\begin{definition}\label{def:partition}Let $\{a_j\}_{j\in\N}$ be a non-decreasing unbounded sequence of positive real numbers. For $\beta>0$, the \emph{partition of $\{a_j\}_{j\in\N}$ of order $\beta$} is the non-decreasing sequence $\mathcal{P}=\{\mathcal{P}(j)\}_{j\in\N}\subset \N$ defined as follows.  Consider the subsequence $\{j_k\}_{k\in \N}\subset \N$ of indices given by the  inductive rule $$j_1=1, \qquad j_{k+1}:=\min\big\{j\in \N: a_{j_k}(1+a_{j_k}^{-\beta/2})<a_j\big\}.$$
Then, $\mathcal{P}(j):=j_k$,  where  $j_k$ is uniquely defined by $a_{j_k}\leq a_j<a_{j_{k+1}}$.
\end{definition}
\begin{example}\label{ex:partitioninglaplacian}
If $a_j=E_j=j(j+1)$, the $j$-th eigenvalue of the Laplacian on the standard $2$-sphere $S^2$, then the partition of $\{E_j\}_{j\in \N}$ of order $\frac{1}{6}$  is given by  
\[
\{j_k\}_{k\in \N}=\{1,2,3,5,7,10,14,\ldots\},\qquad \{\P(j)\}_{j\in \N}=\{1,2,3,3,5,5,7,7,7,10,10,10,10,14,\ldots\}.
\]
\end{example}
We are now prepared to state and prove an equivariant version of the classical Shnirelman-Zelditch-Colin-de-Verdi\`{e}re quantum ergodicity theorem \cite{shnirelman,zelditch1987,colindv}. In the special case that $\widetilde M=M/G$ is an orbifold, a similar statement has been proved by Kordyukov \cite{kordyukov12} for the trivial isotypic component. 
\begin{thm}[\bf Equivariant quantum limits for the Laplacian]\label{thm:quantlim}With the notation as above, 
assume that the reduced geodesic flow is ergodic. 
%Let $\{u_j\}_{j\in\N}$ be an orthonormal basis in $\L^2(M)$ of eigenfunctions of $- \Delta$ with associated eigenvalues $\{E_j\}_{j\in \N}$. 
Choose a semiclassical character family $\mklm{\W_h}_{h \in (0,1]}$ of growth rate $\vartheta<\frac{1}{2\kappa +3}$ and a partition $\mathcal{P}$ of $\{E_j\}_{j\in \N}$ of order $\beta\in \big(0,\frac{1-(2\kappa+3)\vartheta}{2\kappa+4}\big)$. Define the set of eigenfunctions
\[
\big\{u^{\W,\mathcal{P}}_i\big\}_{i\in\N}:=\big\{u_j: \chi_j\in \W_{E^{-1/2}_{\P(j)}}\big\},
\]
where $\chi_j$ is defined by $u_j\in \L^2_{\chi_{j}}(M)$. Then, there is a subsequence $\big\{u^{\W,\mathcal{P}}_{i_k}\big\}_{k\in\N}$ of \footnote{ The expression  \emph{of density 1} means that $
\lim_{m\to\infty}{\#\{k:\,i_k\leq m\}}/{m}=1$.
} density 1 in $\big\{u^{\W,\mathcal{P}}_i\big\}_{i\in\N}$ such that for all $s\in \Cinft(S^*M)$ one has
\begin{equation}
\frac{1}{\sqrt{d_{\chi_{i_k}}[\pi_{\chi_{i_k}}|_{H}:\mathds{1}]}}\;\Big|\eklm{\Op(s) u^{\W,\mathcal{P}}_{i_k}, u^{\W,\mathcal{P}}_{i_k}}_{\L^2(M)} -\fintop_{S^*M\cap\Omega_\text{reg}}s \, \frac{d\mu}{\vol_\O}\Big|\longrightarrow 0 \qquad \text{ as } {k\to\infty},\label{eq:ass}
\end{equation}
where we wrote $\mu$ for $\mu_1$.%, and $\chi_{j_k}$ is defined by $u^{\W,\mathcal{P}}_{j_k}\in \L^2_{\chi_{j_k}}(M)$.
\end{thm}
\begin{rem}
\label{rem:7.2}
The integral in (\ref{eq:ass}) can also be written as $\fint_{S^*\widetilde{M}_\text{reg}}s'\, d(S^*\widetilde{M}_\text{reg})$, 
where $s' \in \Cinft(S^*\widetilde{M}_\text{reg})$ is the function corresponding to $\widetilde{\eklm{s}}_G$ under the diffeomorphism  $\widetilde\Sigma_1\simeq S^*\widetilde{M}_\text{reg}$ up to a null set, and $d(S^*\widetilde{M}_\text{reg})$ is the Liouville measure on the unit co-sphere bundle, see Lemma 2.2, Corollary A.3 and Remark A.11 from Part I. In the orbifold case, this integral is given by an integral over the orbifold co-sphere bundle $S^\ast \widetilde M$. 
\end{rem}

\begin{proof}First, we extend $s$ to a function $\overline{s}\in S^0(M)\subset \Cinft (T^*M)$ with $\overline{s}|_{S^*M}=s$ as follows. Set $\widehat{\overline{s}}(x,\xi):=s(x,\xi/\norm{\xi}_x)$ for $x\in M$, $\xi\in T_x^*M-\{0\}$. Choose a small $\delta>0$ and a smooth cut-off function
 $\varphi:T^*M\to [0,1]$ with 
\begin{align*}
\varphi(x,\xi)&= 1\qquad \forall\, x\in M,\qquad \forall\, \xi\in T_x^*M\text{ with }\norm{\xi}_x\geq 1-\delta,\\
\varphi(x,\xi)&= 0\qquad \forall\, x\in M,\qquad \forall\, \xi\in T_x^*M\text{ with }\norm{\xi}_x\leq\delta.
\end{align*}
Now set $\overline{s}(x,\xi):=\varphi(x,\xi)\cdot\widehat{\overline{s}}(x,\xi)$ for $\xi\in T_x^*M-\{0\}$ and $\overline{s}(x,0):=0$. Then $\Op(\overline{s})$ is a pseudodifferential operator in $\Psi^0(M)$. Because $\overline{s}$ is polyhomogenous of degree $0$ and therefore independent of $|\xi|$ for large $\xi$, the ordinary non-semiclassical quantization $\Op(\overline{s})$ differs only
%\footnote{\bf Clarify this in detail.}
 by an operator in $h^\infty \Psi_h^{-\infty}(M)$ from the semiclassical pseudodifferential operator $\Op_h(\overline{s})\in\Psi_h^0(M)$ with principal symbol $\sigma(\Op_h(\overline{s}))$ = $[\overline{s}]$. Thus, we can apply Theorem \ref{thm:ergod2} to $P(h)=-h^2\Delta$, and we are allowed to replace $\Op_h(\overline{s})$ by $\Op(\overline{s})$ in the results. Fix some $\beta\in \big(0,\frac{1-(2\kappa+3)\vartheta}{2\kappa+4}\big)$. With $c=1$ and $E_j(h)=h^2 E_j$ one has
\begin{align*}
J(h)&=\{j\in \N:E_j(h)\in[c,c+h^\beta],\; \chi_j(h) \in \Wh\}=\left\{j\in \N:E_j\in\left[\frac{1}{h^2},\frac{1}{h^2}+\frac{1}{h^{2-\beta}}\right],\; \chi_j \in \W_{h}\right\}.
\end{align*} 
 Now, by Theorem \ref{thm:ergod2}, there is a number $h_0\in(0,1]$ together with subsets $\Lambda(h)\subset J(h)$, $h\in(0,h_0]$,  satisfying
\begin{equation}
\lim_{h\to0}\frac{\#\Lambda(h)}{\#J(h)}=1,\label{eq:lim1}
\end{equation}
and for each $s\in \Cinft(S^*M)$ and arbitrary $\epsilon>0$ there is a $h_\epsilon \in (0,h_0]$ such that
\begin{equation}
\frac{1}{\sqrt{d_{\chi_j}[\pi_{\chi_j}|_{H}:\mathds{1}]}}\;\Big|\eklm{\Op(\overline{s}) u_j, u_j}_{\L^2(M)}-\fint_{\Sigma_1\cap \Omega_\text{reg}}a\, \frac{d\mu_1}{\vol_\O} \Big|<\epsilon \quad\forall\,j\in \Lambda(h),  \, \forall\, h\in(0,h_\epsilon].\label{eq:estim3}
\end{equation}
Next, consider
%\footnote{\bf Check positivity of spectrum!} 
a partition $\mathcal{P}$ of $\{E_j\}_{j\in \N}$ of order $\beta$ with  $j_k$, $\P(j)$ as in Definition \ref{def:partition}.   Since there are only finitely many eigenvalues $E_j$ with $h_0 <\frac{1}{\sqrt{E_j} }$ there is a $k_0\in \N$  such that $h_k:=\frac 1{\sqrt{E_{j_k}}}\leq h_0$ for all $k \geq k_0$. Let us  apply the results above to the sequence $\mklm{h_k}_{k \geq k_0}$. By construction, $k\neq k'$ implies $J(h_k)\cap J(h_{k'})=\emptyset$ since
\begin{align*}
J(h_k)&=\left\{j\in \N:E_j\in\left[E_{j_k},E_{j_k}\big(1+E_{j_k}^{-\beta/2}\big)\right],\;\chi_j \in \W_{E_{j_k}^{-1/2}}\right\}\\ &=\left\{j\in \N:E_j\in\big[E_{j_k},E_{j_{k+1}}\big),\; \chi_j \in \W_{E^{-1/2}_{\P(j)}}\right\}.
\end{align*}
Now, if  $(a_q)_{q\in \N}$ and $(b_q)_{q\in \N}$ are sequences of real numbers such that $0<a_q\leq b_q$ for all $q$, and  $\liminf_{q\to\infty}b_q>0$,  $\lim_{q\to\infty}\frac{a_q}{b_q}=1$,  the Stolz-Cesaro lemma implies that 
\begin{equation*}
\lim_{N\to\infty}\frac{\sum_{q=1}^N a_q}{\sum_{q=1}^N b_q}=1.
\end{equation*}
Applied to our situation and  taking into account that $J(h_k)\cap J(h_{k'})=\emptyset$ when $k\neq k'$ we deduce from (\ref{eq:lim1}) that \bq
\lim_{N\to \infty}\frac{\#\bigcup_{k=k_0}^N\Lambda(h_k) }
{\#\bigcup_{k=k_0}^NJ(h_k)  }=\lim_{N\to \infty}\frac{\sum_{k=k_0}^N \# \Lambda(h_k) }
{\sum_{k=k_0}^N \#J(h_k)  }=1. \label{eq:density327583}
\eq
If we therefore set
\begin{equation*}
J:=\bigcup_{k \geq k_0}J(h_k)=\Big\{j\in \N: \frac 1{\sqrt{E_{j}}}\leq h_0, \, \chi_j\in \W_{E^{-1/2}_{\P(j)}}\Big\},\qquad \Lambda:=\bigcup_{k \geq k_0}\Lambda(h_k),
\end{equation*}
 we obtain from (\ref{eq:density327583})
\begin{equation*}
\lim_{N\to \infty}\frac{ \#\{\lambda\in\Lambda:\lambda\leq N\} }{\#\{j\in J:j\leq N\} }=1.
\end{equation*}
Consequently, $\{i_k\}_{k\in\N}:=\{j\in \N: \frac 1{\sqrt{E_{j}}}> h_0, \, \chi_j\in \W_{E^{-1/2}_{\P(j)}}\}\cup \Lambda$ is a density $1$ subsequence of $\{j\in \N: \chi_j\in \W_{E^{-1/2}_{\P(j)}}\}$. 
Now, by construction $\widetilde{\eklm{\overline{s}}}_G|_{\widetilde{\Sigma}_1}=\widetilde{\eklm{s}}_G$, and by Lemma 2.2, Corollary A.3 and Remark A.11 from Part I we have
\[
\fint_{\widetilde{\Sigma}_1}\widetilde{\eklm{s}}_G\, d\widetilde{\Sigma}_1=\fint_{S^*\widetilde{M}_\text{reg}}s'\, d(S^*\widetilde{M}_\text{reg}).
\]
From  (\ref{eq:estim3}) we therefore conclude that the sequence $\big\{u^{\W,\mathcal{P}}_{i_k}\big\}_{k\in\N}$ fulfills (\ref{eq:ass}), completing the proof of Theorem \ref{thm:quantlim}.
\end{proof}

Projecting from $S^*M\cap\Omega_\text{reg}$ onto $M$ we  obtain

\begin{cor}[\bf Equidistribution of eigenfunctions of the Laplacian]
\label{cor:equi}
In the situation of Theorem \ref{thm:quantlim}, we have for any $f \in C(M)$
\begin{equation*}
\frac{1}{\sqrt{d_{\chi_{i_k}}[\pi_{\chi_{i_k}}|_{H}:\mathds{1}]}}\;\Big| \intop_M f | u^{\W,\mathcal{P}}_{i_k}|^2 dM -\fintop_{M}f \, \frac{dM}{\vol_\O} \Big|\longrightarrow 0 \qquad \text{ as } {k\to\infty}.
\end{equation*} 

\end{cor}
\begin{proof}  Let $\pi:T^*M\to M$ be the co-tangent  bundle projection and  consider for $f\in \Cinft(M)$ the pseudodifferential operator $\text{Op}(f\circ\pi)$, which corresponds to pointwise multiplication with $f$ up to lower order terms. Since the Sasaki metric on $T^*M$ projects onto the Riemannian metric on $M$ and is fiber-wise just the Euclidean metric, and the Sasaki metric induces $d\mu$, we have
\begin{equation*}
\fintop_{S^*M\cap\Omega_\text{reg}}f\circ \pi \, \frac{d\mu}{\vol_\O}=\fintop_{M}f \, \frac{dM}{\vol_\O},
\end{equation*}
see \cite{kuester}. Consequently,  the assertion follows directly from Theorem \ref{thm:quantlim} by approximating continuous functions on $M$ by smooth functions.
\end{proof}

\subsection{Limits of representations} Corollary \ref{cor:equi} immediately leads to a statement about  measures on the topological Hausdorff space $\widetilde M= M/G$ and to a representation-theoretic formulation of our results. 

\begin{cor}\label{lem:equishnirelman1}
In the situation of Theorem \ref{thm:quantlim}, we have for any $f \in C(\widetilde M)$
\begin{equation*}
\frac{1}{\sqrt{d_{\chi_{i_k}}[\pi_{\chi_{i_k}}|_{H}:\mathds{1}]}}\;\Big| \int_{\widetilde M} f \widetilde{\eklm{ | u^{\W,\mathcal{P}}_{i_k}|^2}}_G d\widetilde M - \fint_{\widetilde M} f \frac{d\widetilde{M}}{\vol} \Big|\longrightarrow 0 \qquad \text{ as } {k\to\infty}.
\end{equation*} 
\end{cor}
\begin{proof}Let $f\in C(\widetilde M)$,  $\pi:M\to \widetilde M$ be the canonical projection, and denote by $\overline{f}:=f\circ \pi \in C(M)$ the lift of  $f$  to a $G$-invariant function.  With Equation (2.17) in Part I and Corollary A3 from the Appendix of Part I one deduces for any $u \in \Cinft(M)$
\begin{align*}
\intop_M \overline{f}(x)|u(x)|^2 \d M(x) &=\intop_{M_\text{reg}} \overline{f}(x)|u(x)|^2 \d M(x)= \intop_{\widetilde{M}_\text{reg}}\intop_{G\cdot x} \overline{f}(x')|u(x')|^2 \d \mu_{G\cdot x}(x') \d \widetilde{M}_\text{reg}(G\cdot x) \\
&=\int_{\widetilde{M}_\text{reg}}f(G\cdot x)\int_{G\cdot x} |u(x')|^2 \d \mu_{G\cdot x}(x') \d \widetilde{M}_\text{reg}(G\cdot x) \\
&=\int_{\widetilde{M}_\text{reg}}f(G\cdot x)\vol(G\cdot x)\int_{G} |u(g\cdot x)|^2 \d g \d \widetilde{M}_\text{reg}(G\cdot x) %\quad \text{by (\ref{eq:orbithaar})} 
\\
&=\int_{\widetilde M}f(G\cdot x) \widetilde{\eklm{|u|^2}}_G(G\cdot x) \d\widetilde{M}(G\cdot x),%\quad \text{by Lemma \ref{lem:qnormmeas}}.
\end{align*}
as well as  $\fint_{M} \overline{f} \frac{dM}{\vol_\O}=\fint_{\widetilde{M}_\text{reg}}f \d \widetilde{M}_\text{reg}=\fint_{\widetilde M} f \frac{d\widetilde{M}}{\vol}$. The claim now follows from Corollary \ref{cor:equi}.
\end{proof}
Next, let us state a simple fact  from elementary representation theory.

\begin{lem}\label{lem:orbitint}Let  $V\subset \L^2(M)$ be an irreducible $G$-module of class $\chi\in \widehat G$. Let further  $\{v_1,\ldots,v_{d_\chi}\}$ denote an $\L^2$-orthonormal basis of $V$, and $a\in V\cap \Cinft(M)$ have $\L^2$-norm equal to $1$. Then, for any $x \in M$, 
\begin{equation}
\eklm{|a|^2}_G(x)=d_\chi^{-1}\sum_{k=1}^{d_\chi}|v_k(x)|^2. \label{eq:orbitint}
\end{equation}
In particular, the function 
\begin{align*}
\Theta_V: M &\to\R, \quad x \mapsto d_\chi^{-1}\sum_{k=1}^{d_\chi}|v_k(x)|^2,
\end{align*}
is a $G$-invariant element of $\Cinft(M)$ that is independent of the choice of orthonormal basis, and the left hand side of (\ref{eq:orbitint}) is independent of the choice of $a$.
\end{lem}
\begin{proof}
Since the left hand side of (\ref{eq:orbitint}) is clearly $G$-invariant, smooth, and independent of the choice of orthonormal basis, it suffices to prove (\ref{eq:orbitint}). Now, one has 
$a=\sum_{j=1}^{d_\chi}a_jv_j$ with $a_j\in\C$, $\sum_{j=1}^{d_\chi}|a_j|^2=1$, and
\[
(L_ga)(x)=a(g^{-1}\cdot x)=\sum_{j=1}^{d_\chi}a_jv_j(g^{-1}\cdot x)=\sum_{j,k=1}^{d_\chi} a_jc_{jk}(g)v_k(x), \qquad g \in G, \, x \in M, 
\]
where $\{c_{jk}\}_{1\leq j,k\leq d_\chi}$ denote the matrix coefficients of the $G$-representation on $V$. This yields 
\begin{align}
\nonumber\int_G |a(g^{-1}\cdot x)|^2\,dg &= \int_G a(g^{-1}\cdot x)\overline{a}(g^{-1}\cdot x)\,dg
=\int_G \Bigg(\sum_{j,k=1}^{d_\chi}a_jc_{jk}(g)v_k(x)\Bigg)\Bigg(\sum_{l,m=1}^{d_\chi}\overline{a}_l\overline{c}_{lm}(g)\overline{v}_m(x)\Bigg)\,dg, 
\end{align}
and we obtain \eqref{eq:orbitint} by taking into account the Schur orthogonality relations \cite[Corollary 1.10]{knapp} 
\begin{equation*}
\int_G c_{jk}(g)\overline{c}_{lm}(g)\,dg=d_\chi^{-1}\delta_{jl}\delta_{km},
\end{equation*}
and the fact that the substitution $g\mapsto g^{-1}$ leaves the Haar measure invariant.
\end{proof}
We can now restate Corollary \ref{cor:equi} in representation-theoretic terms.

\begin{thm}[\bf Representation-theoretic equidistribution theorem]\label{thm:equishnirelman3}
Assume that the reduced geodesic flow is ergodic. By the spectral theorem, choose an orthogonal decomposition $\L^2(M)=\bigoplus_{i\in \N}V_i$ into irreducible unitary $G$-modules such that each $V_i$ is contained in an   eigenspace of the Laplace-Beltrami operator corresponding to some eigenvalue $E_{j(i)}$. Denote by $\chi_i\in \widehat G$ the class of $V_i$.  Choose a semiclassical character family $\mklm{\W_h}_{h \in (0,1]}$  of growth rate $\vartheta<\frac{1}{2\kappa +3}$ and a partition $\mathcal{P}$ of $\{E_j\}_{j\in \N}$ of order $\beta\in \big(0,\frac{1-(2\kappa+3)\vartheta}{2\kappa+4}\big)$. Define the set of irreducible $G$-modules
\[
\big\{V^{\W,\mathcal{P}}_l\big\}_{l\in\N}:=\big\{V_i: \chi_i\in \W_{E^{-1/2}_{\P(j(i))}}\big\}.
\]
As in Lemma \ref{lem:orbitint}, assign to each $V^{\W,\mathcal{P}}_l$ the $G$-invariant function $\Theta_l:= \Theta_{V^{\W,\mathcal{P}}_l}:M\to [0,\infty)$, and regard it as a function on $M/G=\widetilde{M}$.   Then, there is a subsequence $\big\{V^{\W,\mathcal{P}}_{l_m}\big\}_{m\in\N}$ with
\[
\lim_{N\to \infty}\frac{\sum_{l_m\leq N}d_{\chi_{l_m}}}{\sum_{i\leq N}d_{\chi_i}}=1
\]
such that for any $f \in C(\widetilde M)$
\begin{equation*}
\frac{1}{\sqrt{d_{\chi_{l_m}}[\pi_{\chi_{l_m}}|_{H}:\mathds{1}]}}\;\Big| \int_{\widetilde M} f \, \Theta_{l_m}\,  \d\widetilde M - \fint_{\widetilde M} f \frac{d\widetilde{M}}{\vol} \Big|\longrightarrow 0 \qquad \text{ as } {m\to\infty},
\end{equation*}
where $d\widetilde{M}:=\pi_{\ast}dM$ is the pushforward measure defined by the orbit projection $\pi:M\to M/G=\widetilde{M}$ and $\vol: \widetilde{M} \rightarrow (0,\infty)$ assigns to an orbit its Riemannian volume.
\end{thm}
\begin{proof}
Consider the set of eigenfunctions $$\big\{u^{\W,\mathcal{P}}_i\big\}_{i\in\N}=\big\{u_j: \chi_j\in \W_{E^{-1/2}_{\P(j)}}\big\}$$
from Theorem \ref{thm:quantlim}. For each  $l\in \N$ one has $V^{\W,\mathcal{P}}_l=\text{span}\, \big\{u^{\W,\mathcal{P}}_i:i\in J_l\big\}$ for a unique index set $J_l$ with $\#J_l=d_{\chi_l}$. 
Without loss of generality, we can assume $\min(J_1)=1$ and $\min(J_{l+1})=\max(J_l)+1$ for each $l\in \N$. 
By Corollary \ref{lem:equishnirelman1}, there is a subsequence $\big\{u^{\W,\mathcal{P}}_{i_k}\big\}_{k\in\N}$ of density $1$ in $\big\{u^{\W,\mathcal{P}}_{i}\big\}_{i\in\N}$ such that  we have for any $f \in C(\widetilde M)$
\begin{equation*}
\frac{1}{\sqrt{d_{\chi_{i_k}}[\pi_{\chi_{i_k}}|_{H}:\mathds{1}]}}\;\Big| \int_{\widetilde M} f \widetilde{\eklm{ | u^{\W,\mathcal{P}}_{i_k}|^2}}_G d\widetilde M - \fint_{\widetilde M} f \frac{d\widetilde{M}}{\vol} \Big|\longrightarrow 0 \qquad \text{ as } {k\to\infty},
\end{equation*}
and by  Lemma \ref{lem:orbitint}, 
\begin{equation}
\label{eq:1709b}
\widetilde{\eklm{|u^{\W,\P}_{i_k}|^2}}_G=\Theta_l \qquad \text{if $i_k \in J_l$}.
\end{equation} 
Let now   $\{l_m\}_{m\in\N}$ be the sequence of those indices $l$ occurring in \eqref{eq:1709b} when $k$ varies over all of $\N$. Then, due to the way how we indexed our sets $J_l$, we have for each $N\in \N$
\begin{equation*}
\sum_{l_m\leq N}d_{\chi_{l_m}} \geq \sum_{l_m\leq N}\#\{k: i_k\in J_{l_m}\}=\#\Big\{k:i_k\leq \sum_{i\leq N}d_{\chi_{i}}\Big\}
\end{equation*}
Passing to the limit $N\to\infty$ we obtain
\[
1\geq\limsup_{N\to\infty}\frac{\sum_{l_m\leq N}d_{\chi_{l_m}}}{\sum_{i\leq N}d_{\chi_{i}}}\geq\liminf_{N\to\infty}\frac{\sum_{l_m\leq N}d_{\chi_{l_m}}}{\sum_{i\leq N}d_{\chi_{i}}}\geq \lim_{N\to\infty} \frac{\#\Big\{k:i_k\leq \sum_{i\leq N}d_{\chi_{i}}\Big\}}{\sum_{i\leq N}d_{\chi_{i}}}=1,
\]
where the final equality holds because $\big\{u^{\W,\mathcal{P}}_{i_k}\big\}_{k\in\N}$ has density $1$ in $\big\{u^{\W,\mathcal{P}}_{i}\big\}_{i\in\N}$. This concludes the proof of the theorem.
\end{proof}

Note that Theorem \ref{thm:equishnirelman3} is a statement about limits of representations, or multiplicities, in the sense that it assigns to each irreducible $G$-module in the character family a measure on $\widetilde{M}$, and then considers the limit measure. \\

To conclude this section, let us notice that in the special case that $\Wh=\{\chi\}$ for all $h\in (0,1]$ and some fixed $\chi\in \widehat G$, the partitioning of the eigenfunction sequence $\{E_j\}$ is not necessary, and the statements proved in this section become much simpler. Thus, as a direct consequence of Theorem \ref{thm:quantlim} we obtain

\begin{thm}[\bf Equivariant quantum limits for the Laplacian and single isotypic components]\label{thm:quantlim2}Assume that the reduced geodesic flow is ergodic, and choose $\chi\in\widehat{G}$. Let $\{u^\chi_j\}_{j\in\N}$ be an orthonormal basis of $\L^2_\chi(M)$ consisting of eigenfunctions of $- \Delta$. Then, there is a subsequence $\{u^\chi_{j_k}\}_{k\in\N}$ of density $1$ in $\{u^\chi_j\}_{j\in\N}$ such that for all $a\in \Cinft(S^*M)$ one has
\begin{equation}
\eklm{\Op(a) u^\chi_{j_k}, u^\chi_{j_k}}_{\L^2(M)} \;{\longrightarrow}\; \frac 1{\vol_{\frac {\mu}{\vol_\O}} (S^\ast M\cap\Omega_\text{reg})}
\intop_{S^*M\cap\Omega_\text{reg}}a \, \frac{d\mu}{\vol_\O} \qquad \text{ as } {k\to\infty}.\label{eq:res3333233}
\end{equation}
\end{thm}

Next,   recall that a sequence of measures $\mu_j$ on a metric space $\bf{X}$ is said to \emph{converge weakly} to a measure $\mu$, if for all bounded and continuous functions $f$ on ${\bf X}$ one has
\begin{equation*}
\int_{\bf{X}} f \d\mu_j \longrightarrow \int_{\bf{X}} f \d\mu\qquad \text{as }\,j\to\infty.
\end{equation*}
Projecting from $S^*M\cap\Omega_\text{reg}$ onto $M$ we  immediately deduce from Corollary \ref{cor:equi}
\begin{cor}[\bf Equidistribution of eigenfunctions of the Laplacian for single isotypic components]\label{cor:24.08.2015}
In the situation of Theorem \ref{thm:quantlim2} we have the weak convergence of measures 
\begin{equation*}
|u^\chi_{j_k}|^2 \d M\;{\longrightarrow}\; \big(\text{vol}_{\frac{dM}{\vol_\O}}M\big)^{-1}\frac{dM}{\vol_\O} \qquad \text{ as } {k\to\infty}.
\end{equation*}
\end{cor}

On the other hand, Theorem \ref{thm:equishnirelman3} directly implies 

\begin{thm}[\bf Representation-theoretic equidistribution theorem for single isotypic components]\label{thm:equishnirelman34}
Assume that the reduced geodesic flow is ergodic,  and let $\chi\in\widehat{G}$. By the  spectral theorem,  choose an orthogonal decomposition $\L^2_\chi(M)=\bigoplus_{i\in \N}V^\chi_i$ into irreducible unitary $G$-modules of class $\chi$ such that each $V^\chi_i$ is contained in some eigenspace of the Laplace-Beltrami operator. As in Lemma \ref{lem:orbitint}, assign to each $V^\chi_i$ the $G$-invariant function $\Theta_i:= \Theta_{V^\chi_i}:M\to [0,\infty)$, and regard it as a function on $M/G=\widetilde{M}$. Then,  there is a subsequence $\{V^{\chi}_{i_k}\}_{k\in \N}$ of density $1$ in $\{V^{\chi}_i\}_{i\in \N}$ such that
we have the weak convergence
\begin{equation*}
\Theta^{\chi}_{i_k} \d\widetilde{M} \;\overset{k\to\infty}{\longrightarrow}\; \Big(\textrm{vol}_{\frac{d\widetilde{M}}{\vol}}\widetilde{M}\Big)^{-1}\frac{d\widetilde{M}}{{\vol }}.
\end{equation*}
\end{thm}

\section{Applications}\label{sec:example}

In what follows, we apply our results to some concrete situations where a compact connected Riemannian manifold carries an effective isometric action of a compact connected Lie group such that the principal orbits are of lower dimension than the manifold, and the reduced geodesic flow is ergodic.

\subsection{Compact locally symmetric spaces} 
\label{sec:8.1} 
Let  $G$ be a connected semisimple Lie group  with finite center  and Lie algebra $\g$, and   $\Gamma$ a discrete  co-compact subgroup. Consider a Cartan decomposition
\bqn 
\g =\k\oplus \p
\eqn
of $\g$, and denote the  maximal compact subgroup of $G$ with Lie algebra $\k$ by $K$. Choose a left-invariant metric on $G$ given by an $\Ad(K)$-invariant bilinear form on  $\g$.
The quotient  $M=\Xbb:=\Gamma\backslash G$ is a compact manifold without boundary,  and by requiring that the projection $G \rightarrow \Xbb$ is a Riemannian submersion, we obtain a Riemannian structure on $\Xbb$.  
$K$ acts on $G$ and on $\Xbb$ from the right  in an isometric  and effective way, and  the isotropy group of a point $\Gamma g\in \Xbb$ is conjugate 
to the finite group $gKg^{-1}\cap \Gamma$. Hence, all $K$-orbits in $\Xbb$ are either principal or exceptional. Since the maximal compact subgroups of $G$ are precisely the conjugates of $K$, exceptional $K$-orbits arise from elements in $\Gamma$ of finite order.  
Now, let $\Jbb_G:T^\ast \Xbb \rightarrow \g^\ast$ be the momentum map of the right $G$-action on $\Xbb$  and $\mathrm{res}: \g^\ast \rightarrow \k^\ast$ the natural restriction map. Then  $\Jbb_K=\mathrm{res} \circ \Jbb_G$ is  the momentum map of the right $K$-action on $\Xbb$. As usual, let  $\Omega := \Jbb_K^{-1}(\{0\})$.
\begin{figure}[h!]
\begin{tikzpicture}[node distance=1.4cm, auto]

\node (A00) {$G$}; 

\node (B0) [above left of=A00] {${\Xbb}=\Gamma\backslash G$}; 
\node (C0) [right of=B0] {}; 
\node (D0) [right of=C0] {$G/K$}; 
\node (F0) [right of=D0] {}; 

\node (D) [above left of=B0] {}; 
\node (E) [right of=D] {}; 
\node (F) [right of=E] {${{\mathbb Y}}=\Gamma\backslash G/K$}; 

\node (D2) [above left of=F] {};
\node (F2) [left of=D2] {$\Omega$}; 
 
\node (F1) [above right=1cm of F2] {};
\node (G1) [right of=F1] {};
\node (H1) [right of=G1] {$\widetilde \Omega=\Omega/K \,  \simeq \, T^\ast {{\mathbb Y}}\simeq {\mathbb Y} \times \p^\ast$};

\node (E2) [left=1.2cm of F2] {$ \Xbb \times \g^\ast \simeq T^\ast {\Xbb}$}; 
\node (G2) [right of=F2] {$\quad$};

\draw[->, dashed] (A00) to node {} (B0);
\draw[->, dashed] (A00) to node {} (D0);

\draw[->] (E2) to node {} (B0);
\draw[->, dashed] (F2) to node {} (H1);
\draw[->] (F2) to node {$\iota$} (E2);
\draw[<-] (F) to node {} (H1);

\draw[<-, dashed] (F) to node {} (B0);
\draw[<-, dashed] (F) to node {} (D0);

\end{tikzpicture}
\caption{Co-tangent bundle reduction for a locally symmetric space}\label{fig:locsym}
\end{figure}
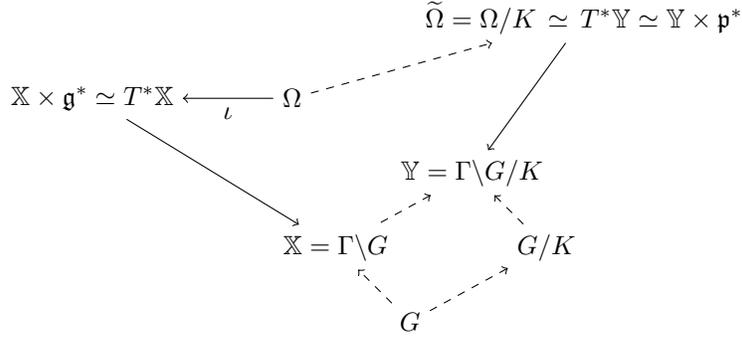

Let us  consider first the case when $\Gamma$ has no torsion, meaning that no non-trivial element $\gamma \in \Gamma$ is conjugate in $G$ to an element of $K$. In this case, there are no exceptional orbits, the action of $\Gamma $ on $G/K$ is free, and ${{\mathbb Y}}:=\Gamma \backslash G/K$ becomes a compact boundary-less   manifold of dimension $n-d$, where $n=\dim \Xbb$ and $d=\dim K$. Furthermore, by  co-tangent  bundle reduction,
\bq
\label{eq:2703c}
T^\ast {{\mathbb Y}}\simeq\Omega /K=: \widetilde \Omega,
\eq
as symplectic manifolds,  compare  \eqref{eq:isomorphic} and Figure \ref{fig:locsym}. In what follows, we give a more intrinsic description of  this symplectomorphism. The left trivialization $T^\ast G \simeq G \times \g^\ast$ described in \eqref{eq:lefttriv}  induces the trivialization 
\bqn
T^\ast{\Xbb}\stackrel{\simeq}\longrightarrow {\Xbb} \times \g^\ast, \qquad  \xi_{\Gamma g} \longmapsto (\Gamma g, (L_{g})_e^\ast \cdot \eta_g), \quad  \mathrm{pr}^\ast (\xi_{\Gamma g}) = \eta_g, \, \eta_g \in T^\ast_g G,
\eqn
$\mathrm{pr}:G \rightarrow \Gamma \backslash G$ being a submersion. The right $G$-action on $T^\ast {\Xbb}$ then takes the form 
\bqn
T^\ast _{\Gamma gh}({\Xbb}) \ni \xi_{\Gamma g} \cdot h=(R_{h^{-1}} )^\ast_{\Gamma g h}\,  \xi_{\Gamma g}\longmapsto(\Gamma g h, (L_{gh})^\ast_e \circ  (R_{h^{-1}} )^\ast_{ g h}\,  \eta_{g}),
\eqn
so that with $\mu= (L_{g})_e^\ast \cdot \eta_g$ we have 
\bq
 \label{eq:2703d}
(\Gamma g, \mu) \cdot h = (\Gamma g h, \Ad^\ast(h) \mu), \qquad h \in G.
\eq
 Now,  for   $X \in \g$ one computes 
\bqn 
\Jbb_G(\Gamma g,\mu)(X) = \Jbb_G(\xi_{\Gamma g}) (X)= (L_{g^{-1}})^\ast_g \, \mu (\widetilde X_g^{\mathrm{R}})=  \mu ( (L_{g^{-1}})_{\ast,g} \widetilde X_g^{\mathrm{R}})=\mu \Big ( \frac d{dt} ({g^{-1}}  g   \e{tX} )_{|t=0}\Big )=\mu( X),
\eqn
where $\widetilde X_g^{\mathrm{R}}$ denotes the vector field generated by the right action of $X$, compare Example \ref{ex:2.2}, so that the   momentum map reads
\bqn 
\Jbb_G(\Gamma g, \mu) =\mu, \qquad (\Gamma g, \mu) \in T^\ast {\Xbb}. 
\eqn   
If   $\eta=(\Gamma g,\mu) \in \Jbb^{-1}_K(\mklm{0}) \subset T^\ast{\Xbb}$,   the last equality  implies $\mu =\Jbb_G(\Gamma g,\mu) \in \p^\ast$. Furthermore, in view of the Cartan decomposition  $G=P K$, where $P$ is the parabolic subgroup with Lie algebra $\p$, one has  the diffeomorphism $G/K\simeq P$. Consequently, we can choose  as representant of the class  $[\eta] \in \widetilde \Omega$ an element $\eta=(\Gamma g,\mu)$ with $g\in P$ and $\mu \in \p^\ast$, yielding the identification 
\bq
\label{eq:2703a}
\widetilde \Omega\simeq (\Gamma \backslash P) \times \p^\ast  \simeq {{\mathbb Y}} \times \p^*.
\eq
On the other hand, the left trivialization $T^\ast P \simeq  P \times \p^\ast$ and the previous arguments  imply  the trivialization 
\bq
\label{eq:2703b}
T^\ast {{\mathbb Y}} \simeq  {{\mathbb Y}} \times \p^\ast.
\eq
Comparing \eqref{eq:2703a} and  \eqref{eq:2703b} then yields the desired intrinsic realization of the symplectomorphism  \eqref{eq:2703c}.

Let us now assume that $G$ has real rank $1$. In this case,  the orbit space ${{\mathbb Y}}$ has strictly negative sectional curvature inherited from $G/K$. Consequently,  its geodesic flow $\psi_t$ is ergodic. Since the measures on the spaces $T^\ast {{\mathbb Y}} \simeq \widetilde \Omega$ are given by the corresponding symplectic forms, this implies that  the reduced geodesic flow $\widetilde \varphi_t$ on $\widetilde \Omega$, which corresponds to $\psi_t$ under the symplectomorphism \eqref{eq:2703c},  is ergodic, and the results  from Section \ref{sec:laplacian} apply.

Next, let us consider a discrete co-compact subgroup  $\Gamma_1$  with torsion. In this case $K$ acts on $\Xbb_1:=\Gamma_1 \backslash G$ with non-conjugated finite isotropy groups, so that ${{\mathbb Y}_1}:=\Gamma_1 \backslash G/K$ is no longer a manifold, but an orbifold. Now, by a theorem of Selberg \cite{selberg60}, any finitely generated linear group contains a torsion free subgroup  of finite index. More generally, Borel  \cite{borel63} showed that  every finitely generated group  of isometries of a simply connected Riemannian symmetric manifold has a normal torsion-free subgroup of finite index. Let therefore  $\Gamma\subset \Gamma_1$ be a normal torsion-free co-compact subgroup of finite index \cite{bruening-heintze79}. In this case,   ${{\mathbb Y}}=\Gamma \backslash G/K$ is a smooth manifold and a finite covering of ${\mathbb Y}_1$, and
 \bqn
{\mathbb X}_1 \simeq F\backslash {\mathbb X}, \qquad {\mathbb Y}_1 \simeq F\backslash {\mathbb Y},
\eqn
where $F$ denotes the finite group $F:=\Gamma_1/\Gamma$. Next, let $\Jbb_G^1:T^\ast \Xbb_1 \rightarrow \g^\ast$ be the momentum map of the right $G$-action on $\Xbb_1$,    $\Jbb_K^1:=\mathrm{res} \circ \Jbb_G^1$, and  $\Omega_1:=(\Jbb_K^1)^{-1}(\mklm{0})$. As in the torsion-free case we have the left trivialization $T^\ast \Xbb_1 \simeq \Xbb_1 \times \g^\ast$ as smooth manifolds, and in analogy to \eqref{eq:2703c} one shows that as symplectic orbifolds 
\bq
\label{eq:0204}
T^\ast {\mathbb Y}_1 \simeq \widetilde \Omega_1,
\eq
 which represents the quotient presentation of the co-tangent bundle of ${\mathbb Y}_1$.  Furthermore,  with \eqref{eq:2703b} we obtain
\bqn
T^\ast {\mathbb Y}_1 \simeq \widetilde \Omega_1 \simeq F\backslash \widetilde \Omega \simeq F\backslash (T^\ast {\mathbb Y})\simeq  {\mathbb Y}_1 \times \p^\ast.
\eqn
Consequently, we have a diagram analogous to Figure \ref{fig:locsym} with $\Gamma$ being replaced by $\Gamma_1$. Besides,  if  ${\Xbb}'_1$ denotes the stratum of orbits of principal type of $\Xbb_1$, notice that singular co-tangent bundle reduction  \eqref{eq:isomorphic} implies
\bqn 
T^\ast {\mathbb Y}_1 \supset T^\ast({\Xbb}_1'/K) \simeq \big ((\Jbb^{1}_K)^{-1}(\mklm{0}) \cap T^\ast {\Xbb}'_1\big )/K\subset (\widetilde \Omega_1)_{\mathrm {reg}},
\eqn
  the measures on these spaces being given by the corresponding symplectic forms, and the complements of the inclusions having measure zero. Consider now the commutative diagram in Figure \ref{fig:torsion}, 
\begin{figure}[h!]
\begin{tikzpicture}[node distance=1.4cm, auto]
 
\node (C0) {$T^\ast \Xbb \supset\Omega$}; 
\node (D0) [right=1.7cm  of C0] {$\widetilde \Omega\simeq T^\ast {\mathbb{Y}}$}; 

\node (C) [below of=C0] {$T^\ast \Xbb_1 \supset \Omega_1$};
\node (D) [below of=D0] {$\widetilde \Omega_1\simeq T^\ast {\mathbb{Y}}_1 $};

\draw[->] (C0) to node {$\pi_K$} (D0);
\draw[->] (C) to node {$\pi_K$} (D);
\draw[->] (C0) to node {$\pi_F$} (C);
\draw[->] (D0) to node {$\pi_F$} (D);

\end{tikzpicture}
\caption{\label{fig:torsion}}
\end{figure}
\noindent
where $\pi_K$ and $\pi_F$ denote the projections of the $K$- and $F$-actions, respectively. To relate the dynamics on the symplectic quotients $\widetilde \Omega$ and $\widetilde \Omega_1$, let $\widetilde p_1 \in \Cinft(\widetilde \Omega_1)$ be a smooth function. By definition, there exists a function $p_1 \in \Cinft(T^\ast \Xbb_1)^K$ such that ${p_{1}}|_{\Omega_1}= \pi^\ast_K\widetilde p_1$. The Hamiltonian flow $\varphi_t^1$ of $p_1$ then induces a Hamiltonian flow $\widetilde \varphi_t^1$ on $\widetilde \Omega_1$, compare Section \ref{sec:2.3}. On the other hand, $\widetilde p_1$ yields a  function  $\widetilde p \in \Cinft(\widetilde \Omega)^F$ with Hamiltonian flow $\widetilde \varphi_t$ induced by the corresponding flow $\varphi_t$ on $T^\ast \Xbb$. Since $\varphi_t$ induces the flow $\varphi_t^1$, it is clear that $\widetilde \varphi_t$ induces a flow on $\widetilde \Omega_1$ given precisely by $\widetilde \varphi_t^1$. Indeed, for $\widetilde f_1\in \Cinft(\widetilde \Omega_1)$ and $\widetilde \eta_1=\pi_K(\eta_1)=\pi_F\circ \pi_K(\eta)=\pi_F(\widetilde \eta)\in \widetilde \Omega_1$ one computes for $\widetilde f_1(\widetilde \varphi_t^1(\widetilde \eta_1))$
\bqn 
\pi_K^\ast \widetilde f_1 (\varphi_t^1(\eta_1))=(\pi^\ast_F \circ \pi^\ast_K \widetilde f_1)_{} (\varphi_t(\eta))=(\pi^\ast_K \circ \pi^\ast_F \widetilde f_1)_{} (\varphi_t(\eta))=\pi^\ast_F \widetilde f_1( \widetilde \varphi_t(\widetilde \eta)).
\eqn
Furthermore,  in view of \eqref{eq:0204}, $\widetilde \varphi_t^1$ yields  a  flow $\psi_t^1$ on ${\mathbb Y}_1$.

Let now $\psi_t$ be the geodesic flow on ${{\mathbb Y}}$, and assume that the rank of $G$ is $1$, so that $\psi_t$ is ergodic. Then the induced flow  $\psi_t^1$ on ${\mathbb Y}_1$ is  ergodic, too, with respect to the orbifold symplectic measure on $T^\ast {\mathbb Y}_1$. More precisely, 
by our previous considerations the ergodicity of the flow $\widetilde \varphi_t$ on $\widetilde \Omega$  implies  that 
\bqn 
(\widetilde \varphi_t^1)|_{(\widetilde \Omega_1)_{\mathrm{reg}}},
\eqn
which is precisely the reduced geodesic flow  on the symplectic stratum $(\widetilde \Omega_1)_{\mathrm{reg}}$ given by \eqref{eq:11.06.2015},
must be ergodic  with respect to the symplectic measure $d((\widetilde \Omega_1)_{\mathrm{reg}})$.  Summing up, our results from Section \ref{sec:laplacian} apply. For simplicity, let us state here only the results  for single isotypic components. Then, Theorem \ref{thm:quantlim2} and Corollary \ref{cor:24.08.2015} yield 

\begin{prop}\label{prop:symmspace}
Let  $G$ be a connected semisimple Lie group of rank $1$ with finite center, $K$ a maximal compact subgroup,  and $\Gamma$ a discrete  co-compact subgroup, possibly with  torsion. Let $\Delta$ be the Laplace--Beltrami operator on $\Xbb=\Gamma \backslash G$, $\chi\in\widehat{K}$, and let $\{u^\chi_j\}_{j\in\N}$ be an orthonormal basis of $\L^2_\chi(\Xbb)$ of eigenfunctions of $- \Delta$. Then there is a subsequence $\{u^\chi_{j_k}\}_{k\in\N}$ of density 1 in $\{u^\chi_j\}_{j\in\N}$ such that for all $s\in \Cinft(S^*\Xbb)$ one has
\begin{equation}
\label{eq:2703f}
\eklm{\Op(s) u^\chi_{j_k}, u^\chi_{j_k}}_{\L^2(\Xbb)} \; \stackrel{k\to\infty}{\longrightarrow}\;\fintop_{S^*\Xbb \, \cap\, \Omega_{\mathrm{reg}}}s \, \frac{d\mu}{\vol_\O},
\end{equation}
as well as 
\begin{equation}
 \label{eq:2703e}
|u^\chi_{j_k}|^2 \d \Xbb \;\overset{k\to\infty}{\longrightarrow} \; \big(\text{vol}_{\frac{d\Xbb}{\vol_\O}}\Xbb\big)^{-1}\frac{d\Xbb}{\vol_\O}, \qquad \widetilde{\eklm{|u^\chi_{j_k}|^2}}_G \d{{\mathbb Y}} \;\overset{k\to\infty}{\longrightarrow}\; \left(\textrm{vol}_{\frac{d{ {\mathbb Y}}}{\vol}}{{\mathbb Y}}\right)^{-1}\frac{d{{\mathbb Y}}}{\vol},
\end{equation} 
where ${\mathbb{Y}}=\Gamma \backslash G/ K$ is in general an orbifold, and $d{\mathbb Y}$ is the pushforward of the measure $d\Xbb$ along the orbit projection $\Xbb\to {\mathbb Y}$, see \cite[Section 2.4]{kuester-ramacher15a}.
\end{prop}

Notice that the limit integral in \eqref{eq:2703f} represents an integral over the orbifold co-sphere bundle $S^\ast {\mathbb Y}$. Since the orbit volume function is constant in this case, eigenfunctions of the Laplacian $\Delta_{{\mathbb Y}}$ on ${{\mathbb Y}}$ correspond to $K$-invariant eigenfunctions of $\Delta$ on $\Xbb$, compare Section \ref{sec:1.4}. Furthermore, up to the constant given by the orbit volume, the pushforward measure $d\mathbb{Y}$ agrees in the orbifold case with the orbifold volume form. Consequently, in the special case that $\chi$ corresponds to the trivial representation, Proposition \ref{prop:symmspace} yields the following result already implied  by the work of Kordyukov \cite{kordyukov12}.

\begin{cor}[\bf Shnirelman-Zelditch-Colin-de-Verdi\`{e}re equidistribution theorem for ${{\mathbb Y}}$] With the assumptions of Proposition \ref{prop:symmspace}, let $\{v_j\}_{j\in\N}$ be an orthonormal basis of $\L^2({{\mathbb Y}})$ of eigenfunctions of $- \Delta_{{\mathbb Y}}$. Then there is a subsequence $\{v_{j_k}\}_{k\in\N}$ of density 1 in $\{v_j\}_{j\in\N}$ such that we have the weak convergence of measures
\bqn 
{{|v_{j_k}|^2}} \d{{\mathbb Y}} \;\overset{k\to\infty}{\longrightarrow}\; \left(\textrm{vol}_{\, {d{ {\mathbb Y}}}}{{\mathbb Y}}\right)^{-1}{d{{\mathbb Y}}}.
\eqn
\end{cor}

Notice that  in view of the left trivialization $T^\ast \Xbb \simeq \Xbb \times \g^\ast$ and  \eqref{eq:2703d},    $K$ acts on $\Omega\subset  {\mathbb{X}} \times \p^\ast$ by right multiplication according to 
\bqn 
 \Omega  \ni (\Gamma g , \mu) \cdot k = (\Gamma g k, \Ad^\ast(k) \mu) \in \Omega, \qquad k \in K,
\eqn
$\p$ being $\Ad(K)$-invariant.  In particular, regarding the decomposition of $T^\ast \Xbb$ into isotropy types with respect to the right $K$-action, whenever $\Gamma $ contains non-trivial elliptic elements, 
  the closure of $S^*\Xbb \, \cap\,   \Omega_{\mathrm{reg}}$ in $\Omega$ will  contain exceptional isotropy types, which means that in the proofs of  Theorems \ref{thm:ergod1} and \ref{thm:quantlim} one cannot assume that one can stay away from the singular points of $ \Omega$, compare also  Examples 4.8 of Part I.

\subsection{Invariant metrics on spheres in dimensions $2$ and $4$}\label{sec:6.2}
In contrast to  genuinely chaotic cases, it can happen that the reduced geodesic flow is ergodic  simply for topological reasons. Namely, when the singular symplectic reduction of the co-sphere-bundle is just $1$-dimensional, a single closed orbit of the reduced flow can have full measure. Although  non-generic, this situation is topologically invariant, so that if it occurs for some particular $G$-space, it occurs for any choice of $G$-invariant Riemannian metric on that space, leading to a whole class of examples which might well be complicated {geometrically}. 

In what follows, we will  show that the spheres in dimensions $2$ and $4$, with appropriate group actions and invariant  Riemannian metrics, are examples of the form just described. The reason why we consider only the dimensions $2$ and $4$ is that, in general, the $n$-sphere is topologically the suspension of the $(n-1)$-sphere, but only for $n\in\{2,4\}$, the $(n-1)$-sphere has the structure of a compact connected Lie group. Thus, let $G$ be a compact connected Lie group. The \emph{suspension} of $G$ is the quotient space $$SG:=\big ( [-1,1]\times G\big )/\big ((-1,g_1)\sim (-1,g_2),\; (1,h_1)\sim(1,h_2)\big).$$
 $SG$ is a compact connected Hausdorff space that carries an effective $G$-action induced by the $G$-action on $G$ by left-multiplication and the trivial action on $[-1,1]$. We will call this induced action the \emph{suspension of the $G$-action}. It has exactly two fixed points $N:=[\{1\}\times G]$ and $S:=[\{-1\}\times G]$ which we may call \emph{north pole} and \emph{south pole}. Now, in general, $SG$ does not possess a differentiable structure. However, if $G$ is an $n$-sphere, then $SG$ is homeomorphic to the $(n+1)$-sphere, and consequently carries a canonical smooth structure making it diffeomorphic to the standard $(n+1)$-sphere. As is well-known, the only connected Lie groups that are spheres are $\SO(2)\cong S^1$ and $\text{SU}(2)\cong S^3$. 

Note that $S^2$, with the $S^1$-action given by the suspension of left-multiplication on $S^1$ and equipped with an $S^1$-invariant Riemannian metric, is just a surface of revolution diffeomorphic to the $2$-sphere. Similarly, for $G=S^3$, we equip the suspension $S^4\cong SS^3$ with the $S^3$-action given by the suspension of left-multiplication on $S^3$ and an $S^3$-invariant Riemannian metric, obtaining a class of $4$-dimensional examples. We now have the following
\begin{prop}\label{prop:suspension}
For $n\in\{2,4\}$, equip the $n$-sphere $S^n\cong SS^{n-1}$ with the $S^{n-1}$-action given by the suspension of left-multiplication on $S^{n-1}$. Then the reduced geodesic flow with respect to any $S^{n-1}$-invariant Riemannian metric on $S^n$ is ergodic.
\end{prop}
\begin{proof}First, we prove the result for $S^2$. It will then become clear that the situation is entirely analogous for $S^4$. Thus, let $G=S^1\cong \SO(2)$. Then, for any choice of $\SO(2)$-invariant metric on $M:=SS^1$, we can identify $M$ with  a surface of revolution in $\R^3$ diffeomorphic to the $2$-sphere and endowed with the induced metric from $\R^3$. We assume that the  \emph{poles}  are given by the points $N=(0,0,1)$ and $S=(0,0,-1)$. The corresponding {meridians} are orthogonal to the $\SO(2)$-orbits, and since the metric is $\SO(2)$-invariant, each meridian is a closed geodesic. Now, for $(x,\xi)\in T^*M$, set $p(x,\xi):=\norm{\xi}^2_{x}$. Let $c>0$ and put $\Sigma_c:=p^{-1}(\{c\})$ and $\widetilde{\Sigma}_c:=\widetilde{p}^{-1}(\{c\})$, where $\widetilde{p}\in \Cinft(\widetilde{\Omega}_{\textrm{reg}})$ is the function induced by $p|_{\Omega_{\textrm{reg}}}$. Clearly, $c$ is a regular value of $p$.  To examine whether the reduced geodesic flow is ergodic on $\widetilde \Sigma_c$, note that with the identification $T^\ast M\simeq TM$ given by the Riemannian metric one has 
\bq
\Omega=\Jbb^{-1}(\{0\}) \simeq \bigsqcup _{x \in M} T_x(G \cdot x)^\perp,\label{eq:omega5656}
\eq
so that
\begin{align*}
\Omega_{\textrm{reg}} & \simeq \Big (\bigcup _{x \in M_{\textrm{reg}}} \, \{x\}\times T_x(G \cdot x)^\perp\, \Big ) \cup\Big ( \{N\}\times  \left( T_NM\backslash \{0\}\right) \Big ) \cup \Big ( \{S\} \times  \left( T_SM\backslash \{0\}\right)\Big ) ,\\
\widetilde \Omega_{\textrm{reg}} &\simeq  \big ( (-1,1) \times \R\big )  \cup \big ( \mklm{1}\times (0,\infty)  \big ) \cup \big ( \mklm{-1} \times (0, \infty) \big )\simeq \R^2\backslash \{(0,1),(0,-1)\} ,
\end{align*}
where $M_{\textrm{reg}}=M\backslash\mklm{N,S}$, $M_{\textrm{reg}}/G\simeq (-1,1)$. The diffeomorphism $\widetilde \Omega_{\textrm{reg}} \simeq \R^2\backslash \{(0,1),(0,-1)\}$ is illustrated  in Figures \ref{fig:fig1} and \ref{fig:fig2} for  $S^2$ with the round metric, which is the generic case since $M$ is $\SO(2)$-equivariantly diffeomorphic to it.  
\begin{figure}[h!]
\centering
\hspace{-6em}\begin{minipage}{0.45\linewidth}
\centering
\hspace*{3em}\includegraphics[width=\linewidth]{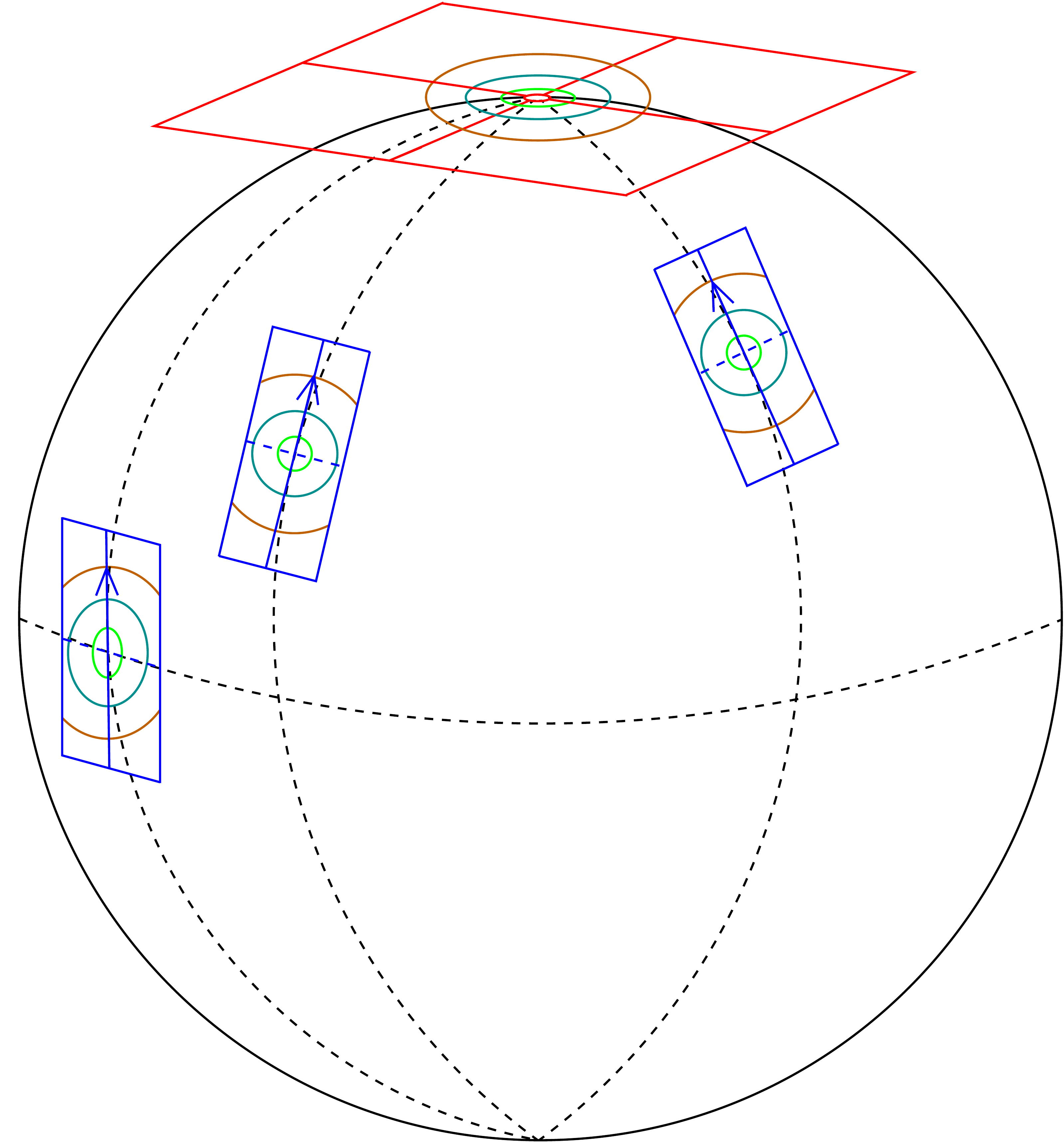}
\parbox{1.33\linewidth}{\caption{The space $T_N^*S^2\backslash\{0\}$ (red) and three co-tangent  spaces (blue) with arrows that represent elements of $\Omega_{\textrm{reg}}$. The three circles in each plane (brown, teal, green) correspond to the intersection of the plane with $\Sigma_c$ for three different values of $c$.}\label{fig:fig1}}
\end{minipage}\hspace*{8em}
\hspace*{-4em}\begin{minipage}{0.45\linewidth}
\centering
\hspace*{4em}\includegraphics[width=0.9\linewidth]{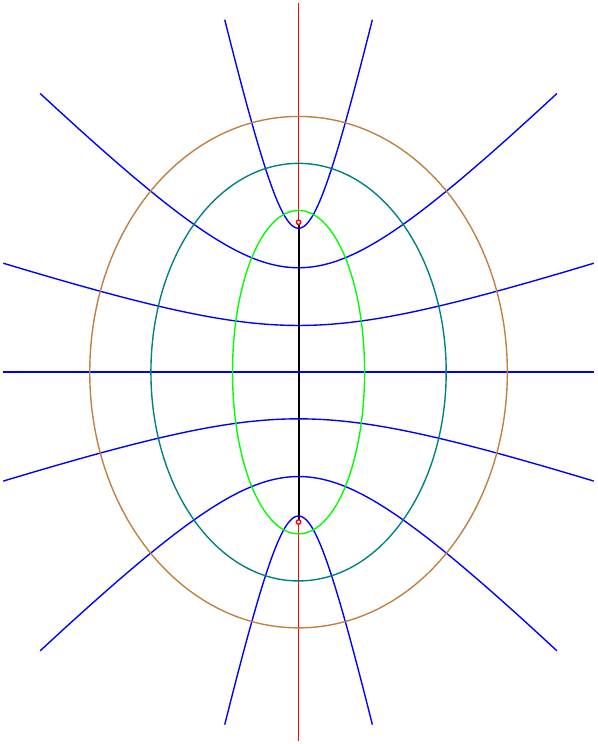}
\parbox{1.33\linewidth}{\vspace*{-0.75em}\caption{Under the projection $\Omega_{\textrm{reg}}\to \widetilde{\Omega}_{\textrm{reg}}$, $T_N^*S^2\backslash\{0\}$ and $T_S^*S^2\backslash\{0\}$ collapse to open half-lines (red) and for every $x\in S^2\backslash\{N,S\}$, $T^*_xS^2\cap \Omega_{\textrm{reg}}$ collapses to a line (blue). The ellipses (brown, teal, green) depict $\widetilde{\Sigma}_c$ for three different values of $c$.\label{fig:fig2}}}
\end{minipage}
\end{figure}
Under the diffeomorphism $\widetilde \Omega_{\textrm{reg}} \simeq \R^2\backslash \{(0,1),(0,-1)\}$, the hypersurface $\widetilde \Sigma_c$ corresponds to an ellipse with radii determined by $c$, as illustrated in Figure \ref{fig:fig2}. Let now $G\cdot(x,\xi) \in \widetilde \Sigma_c$. Since $\xi \in T_x(G \cdot x)^\perp$, the geodesic flow $\varphi_t$ transports $(x,\xi)$ around curves in $T^*M$ that project onto meridians through $N$ and $S$, so that the reduced geodesic flow $\tphi_t(G\cdot (x,\xi))\equiv G\cdot\varphi_t(x,\xi)$ through $G\cdot (x,\xi)$ corresponds to a periodic flow around the ellipse $\widetilde \Sigma_c$. Consequently, the only subsets of $\widetilde \Sigma_c$  which are invariant under $\widetilde \varphi_t$ are the whole ellipse and the empty set, implying that the reduced flow $\tphi_t$ on $\widetilde \Sigma_c$ is ergodic for arbitrary $c>0$. Besides, note that the points on the segment between $(0,1)$ and $ (0,-1)$ are stationary under $\widetilde \varphi_t$. 

Next, let us check what happens for a general compact connected Lie group $G$. Due to the definition of $SG$ and its $G$-action, it is clear that $SG/G$ is homeomorphic to $[-1,1]$ and, due to (\ref{eq:omega5656}), that $\widetilde{\Omega}_{\textrm{reg}}$ is diffeomorphic to $\R^2\backslash \{(0,1),(0,-1)\}$ whenever $SG$ is a smooth manifold, so that we always obtain not only an analogous but essentially {the same} picture as depicted in Figure \ref{fig:fig2}. Hence, for $G=S^3$, the  reduced geodesic flow is given by a periodic flow around an ellipse, and therefore ergodic.\end{proof}

We shall now apply some of our results from Section \ref{sec:laplacian} to a surface of revolution diffeomorphic to the $2$-sphere. Thus, let $M\subset \R^3$ be given by rotating a suitable smooth curve $\gamma:[0,L]\to \R_{x\geq 0}^2$ in the $xz$-half plane around the $z$-axis in $\R^3$. In particular, $\gamma'(t)$ has to be perpendicular to the $z$-axis at $\gamma(0)$ and $\gamma(L)$.  We assume that $\gamma(0)=(0,-1)$ and $\gamma(L)=(0,1)$ and that $\gamma$ is parametrized by arc length, so that  $\gamma: [0,L] \ni \theta\mapsto (R(\theta),z(\theta))$, where  $R:[0,L]\to [0,\infty)$, $R(0)=R(L)=0$, $R(\theta)>0$ for $\theta\in(0,L) $ corresponds to the distance to the $z$-axis, and  $z:[0,L]\to\R$ is smooth. This leads to a parametrization of $M$ according to 
$$M=\big\{(R(\theta)\cos \phi,R(\theta)\sin\phi,z(\theta)),\;\theta\in[0,L],\;\phi\in [0,2\pi)\big\}.$$ Now, let $M$ be endowed with the induced metric on $\R^3$. The Laplace-Beltrami operator $\Delta$ on $M$ commutes with $\partial_\phi$, so that separation of variables leads to a Hilbert basis of $L^2(M)$ of joint eigenfunctions of both operators of the form
\begin{equation}
e_{l,m}(\phi,\theta)=f_{l,m}(\theta)e^{im\phi}, \qquad (l,m)\in \mathcal I\subset \Z\times\Z. \label{eq:elm}
\end{equation}
The irreducible representations of $\SO(2)\simeq S^1=\{e^{i\varphi},\;\varphi\in[0,2\pi)\}\subset \C$ are all $1$-dimensional, and given by the characters $\chi_k(e^{i\phi})=e^{-ik\phi}$, $k \in \Z$. Thus, each subspace $\C \cdot e_{l,m}$ corresponds to an irreducible representation of $\SO(2)$, and $\{e_{l,m}\}_{l:\,(l,m)\in \mathcal{I}}$ is a Hilbert basis of $L^2_{\chi_{m}}(M)$. Furthermore,  $|e_{l,m}|^2$ is manifestly $\SO(2)$ invariant. 
Theorem \ref{thm:equishnirelman34} then yields for each $m \in \mathbb{Z}\simeq \widehat{\SO(2)}$  a subsequence $\{e_{l_k,m}\}_{k\in\N}$ of density $1$ in $\{e_{l,m}\}_{l:\,(l,m)\in \mathcal{I}}$ such that for all $a\in C(\widetilde M)$
\begin{align}
\int_{\widetilde M} a|e_{l_k,m}|^2 \d\widetilde{M}
\overset{k\to\infty}{\longrightarrow}\; \Big(\int_{\widetilde M} \frac{\d{\widetilde M}}{\text{vol}}\Big)^{-1}\int_{\widetilde M} a\frac{\d{\widetilde M}}{\text{vol}},\label{eq:statement60}
\end{align}
where as before  $\widetilde M=M/\SO(2)$.
Let us write (\ref{eq:statement60}) more explicitly. An $\SO(2)$-orbit of a point $x\in M$ with coordinates $(\phi,\theta)$ is of the form $\{(\phi',\theta):\;0<\phi'<2\pi\}$, up to a set of measure zero with respect to the induced orbit measure  $d\mu_{\SO(2)\cdot x}\equiv R(\theta) \d\phi$, and we obtain $\text{vol} (\SO(2)\cdot x)=\int_{0}^{2\pi}R(\theta) \,d\phi=2\pi R(\theta). $
Furthermore, $\widetilde M$ is homeomorphic to the closed interval $[0,L]\subset \R$, and the pushforward measure on $\widetilde M$ is given by $d\widetilde{M}(\theta)\equiv2\pi R(\theta) \d\theta$, where we identified $\SO(2)\cdot x$ and $\theta$. 
Summing up, (\ref{eq:statement60}) yields 
\begin{equation}
2\pi\int_{0}^{L}a(\theta) |f_{l_k,m}|^2(\theta) R(\theta)\d\theta \;
\overset{k\to\infty}{\longrightarrow}\; \frac{1}{L}\int_{0}^{L}a(\theta)\d \theta,\qquad a\in C([0,L]), \label{eq:statement62}
\end{equation}
which is a result about  weak convergence of measures on $\widetilde M\cong [0,L]$. Formulated on $M$, Corollary \ref{cor:24.08.2015} yields that for each $m$ there is a subsequence $\{f_{l_k,m}\}_{k\in \N}$ of density $1$ in $\{f_{l,m}\}_{l:\,(l,m)\in \mathcal{I}}$ such that one has the weak convergence of measures
\begin{equation}
|f_{l_k,m}|^2\d M\quad\stackrel{k\to \infty}{\longrightarrow}\quad \frac{1}{2\pi L}\frac{dM}{R}.\label{eq:surfresult}
\end{equation}
Here, $\frac{dM}{R}$ is to be understood as the extension by zero of the smooth measure $dM(\phi,\theta)/R(\phi,\theta)$ from $\{(\phi,\theta),\;\theta\in(0,L)\}$ to $\{(\phi,\theta),\;\theta\in[0,L]\}$, and we used that $\text{vol}_{\frac{dM}{R}}M=2\pi L$. In particular, the obtained quantum limit on $M$ is, up to a constant, related to the Riemannian volume density on $M$ by the reciprocal of the distance function $R$, which  tends to infinity towards the poles. This is  illustrated in Figure \ref{fig:3}, where the function $1/R$ is plotted on a surface of revolution.
\begin{figure}[h!]
\includegraphics[width=0.17\linewidth]{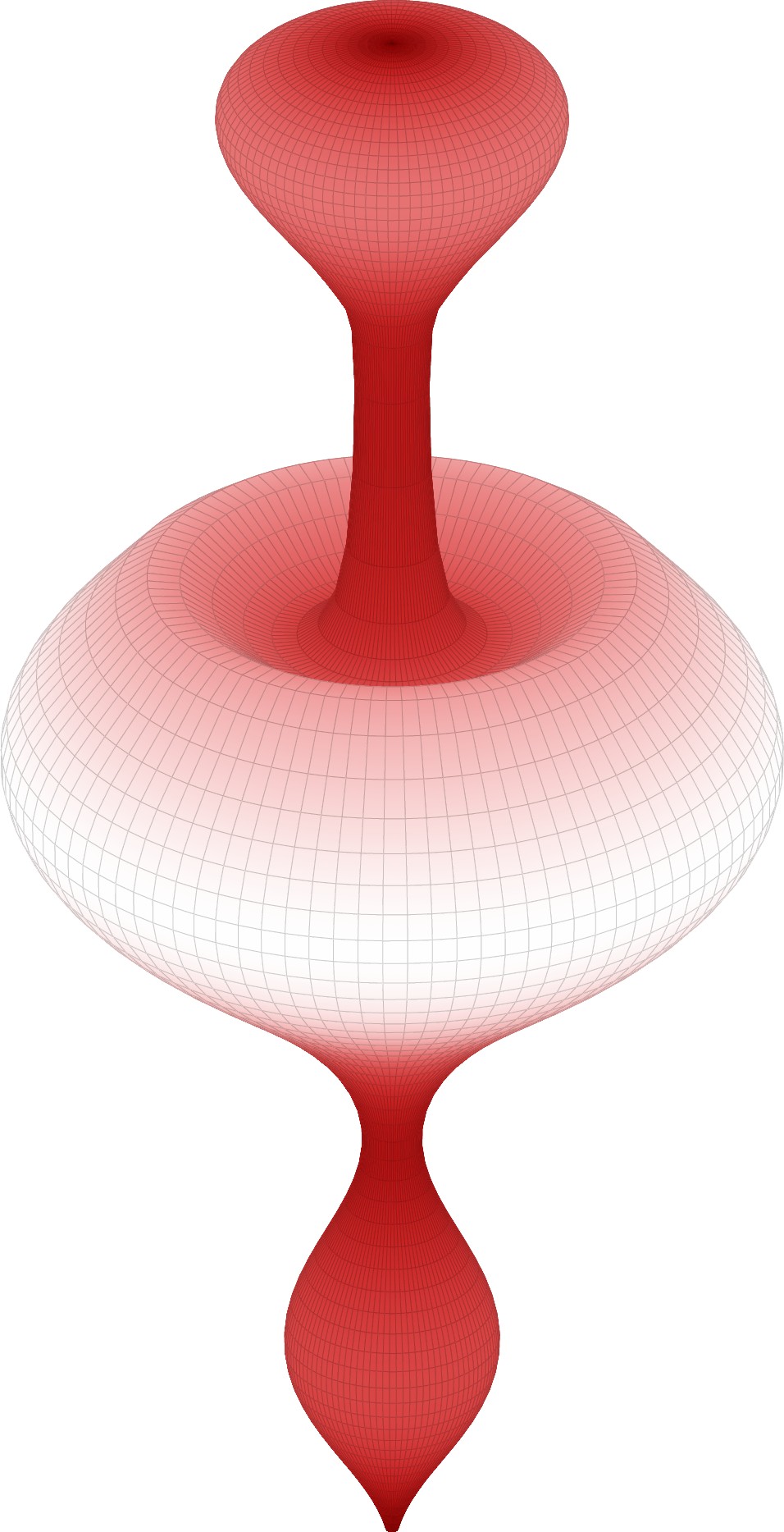}
\vspace{-.2cm}

\caption{A quantum limit on a surface of revolution.}\label{fig:3}
\end{figure}

So far, for simplicity of presentation, we have restricted ourselves to the special case of considering a single fixed isotypic component, which means keeping the index $m$ fixed. Even in this case, we do not know whether the results  (\ref{eq:statement62}) and (\ref{eq:surfresult}) are known for general surfaces of revolution. Having actually the more general Theorem \ref{thm:equishnirelman3} at hand, the results (\ref{eq:statement62}) and (\ref{eq:surfresult}) directly generalize to the situation of a semiclassical character family of growth rate $\vartheta<\frac{1}{5}$ since the dimensions of the irreducible representations are all $1$ in this case, and all principal isotropy groups are trivial, so that $[\pi_{\chi}|_{H}:\mathds{1}]=d_\chi=1$. 

%The growth rate $\vartheta$ means in this example simply that we work within the family of eigenfunctions $e_{l,m}$ with $|m|\leq l^{\vartheta'}$, where $\vartheta'\in(0,1)$ is determined by $\vartheta$. For instance, in the case of the standard $2$-sphere, the $j$-th eigenvalue $j(j+1)$ is asymptotially proportional to $j^2$ and we have $\vartheta'=\vartheta/2$, compare Figure 1.2 in Part I.

Physically, one can interprete these results as follows. For each family of symmetry types that  does not grow too fast in the high-energy limit, there is a sequence of quantum states such that  the corresponding sequence of probability densities on $M$ converges weakly and with density $1$ in the high-energy limit  to the probability density of finding  within a certain surface element of $M$ a classical particle  with known energy and zero angular momentum with respect to the $z$-axis, but unknown momentum. 

In the simplest case of the standard $2$-sphere $M=S^2$ with the round metric, the eigenfunctions are explicitly known, and we show in the following that at least our simplest result (\ref{eq:statement62}) for fixed isotypic components is implied by the classical theory of spherical harmonics. In fact, we will  see  that one does not  need to pass to a subsequence of density $1$. Recall  from Section \ref{sec:1.5} that the eigenvalues of $-\Delta$ on $S^2$ are given by the numbers $l(l+1)$, $ l=0,1,2,3\dots$, and the corresponding eigenspaces $E_l$ are of dimension $2l+1$. They are  spanned by the spherical harmonics
\begin{equation}
Y_{l,m}(\phi,\theta)=\sqrt{\frac{2l+1}{4\pi}\frac{(l-m)!}{(l+m)!}} P_{l,m}(\cos \theta)e^{im\phi}, \qquad 0\leq \phi<2\pi, \, 0 \leq \theta < \pi, \label{eq:Ylm}
\end{equation}
where $m \in \Z$, $|m|\leq l$, and  $P_{l,m}$ are the associated Legendre polynomials
\begin{equation}
P_{l,m}(x)=\frac{(-1)^m}{2^ll!}\left(1-x^2\right)^{\frac{m}{2}}\frac{d^{l+m}}{dx^{l+m}}\left(x^2-1\right)^{l},  \label{eq:legendre}
\end{equation}
compare  (\ref{eq:elm}). Each subspace $\C \cdot Y_{l,m}$ corresponds to an irreducible representation of $\SO(2)$, and each irreducible representation $\chi_k$ with $|k|\leq l$ occurs in the eigenspace $E_l$ with multiplicity $1$. The situation is illustrated in Figure 1.1 of Part I. For each $m$, the result (\ref{eq:statement62}) now turns into the following result about Legendre polynomials:
\begin{equation}
\frac{2l_k+1}{2}\frac{(l_k-m)!}{(l_k+m)!}\int_{0}^{\pi}a(\theta)\sin(\theta)|P_{l_k,m}(\cos \theta)|^2 \d\theta \;
\overset{k\to\infty}{\longrightarrow}\; \frac{1}{\pi}\int_{0}^{\pi}a(\theta)\d \theta\qquad \forall\;a\in C([0,\pi]).\label{eq:statement655}
\end{equation}
We now show the following 
\begin{prop}
For fixed $m$, (\ref{eq:statement655}) holds for the full sequence of Legendre polynomials, that is,  if $l_k$ is replaced by $l$ and ``$k \to \infty$'' is replaced by ``$l\to \infty$''.
\end{prop}
\begin{proof} Let us begin by recalling  the following classical result about the asymptotic behavior of Legendre polynomials \cite[page 303]{hobson}.  For fixed $m\in \Z$ and each small $\varepsilon>0$ one has
\begin{equation}
\frac{1}{l^m}P_{l,m}(\cos \theta)=\left(\frac{2}{l\pi \sin\theta}\right)^{1/2}\cos\left(\left(l+\frac{1}{2}\right)\theta-\frac{\pi}{4}+\frac{m\pi}{2}\right)+\mathrm{O}\left(l^{-3/2}\right)\label{eq:legendreasymp}
\end{equation}
 as $l\to \infty$
uniformly in  $\theta \in (\varepsilon, \pi-\varepsilon)$. From (\ref{eq:Ylm}) and (\ref{eq:legendreasymp}) we therefore obtain
\begin{align*}
\widetilde{|Y_{l,m}|}(\theta)^2 &= \left| \sqrt{ \frac{2l+1}{4\pi} \frac{(l-m)!}{(l+m)!} } P_{l,m}(\cos \theta) \right|^2 = \frac{2l+1}{4\pi}\frac{(l-m)!}{(l+m)!} l^{2m} \left|\frac{1}{l^m}P_{l,m}(\cos \theta)\right|^2 \\ 
&=\frac{2l+1}{4\pi}\frac{(l-m)!}{(l+m)!}l^{2m}\left|\left(\frac{2}{l\pi \sin \theta}\right)^{1/2}\cos\left(\left(l+\frac{1}{2}\right)\theta-\frac{\pi}{4}+\frac{m\pi}{2}\right)+\mathrm{O}\left(l^{-3/2}\right)\right|^2 \\
&=\frac{2l+1}{4\pi}\frac{(l-m)!}{(l+m)!}l^{2m}\left(\frac{2}{l\pi \sin \theta}\cos^2 \left(\left(l+\frac{1}{2}\right)\theta-\frac{\pi}{4}+\frac{m\pi}{2}\right)+\mathrm{O}\left(l^{-2}\right)\right).
\end{align*}
The asymptotic relation
\begin{equation}
(l-m)!/(l+m)!\sim l^{-2m}\quad \text{as }l\to\infty\label{eq:lasymp}
\end{equation}
implies that $\frac{(l-m)!}{(l+m)!}l^{2m}$ is bounded in $l$, so we can use the simple relation $\frac{2l+1}{l}=2+\mathrm{O}(l^{-1})$ to obtain
\begin{equation}
\widetilde{|Y_{l,m}|}(\theta)^2 =\frac{(l-m)!}{(l+m)!}l^{2m}\frac{1}{\pi^2 \sin \theta}\cos^2\left(\left(l+\frac{1}{2}\right)\theta-\frac{\pi}{4}+\frac{m\pi}{2}\right)+\mathrm{O}\left(l^{-1}\right),\label{eq:asymp66}
\end{equation}
uniformly for $\theta\in(\varepsilon,\pi-\varepsilon)$ and each small $\varepsilon>0$. Now let $f\in C([0,\pi], \R)$ and choose $\varepsilon>0$. Due to the uniform estimate (\ref{eq:asymp66}) and boundedness of the integration domain we get
\begin{align}
\begin{split}
2\pi\int_{\varepsilon}^{\pi-\varepsilon}&f(\theta)\widetilde{|Y_{l,m}|}(\theta)^2\sin\theta \d\theta \\
&=2\pi\int_{\varepsilon}^{\pi-\varepsilon}f(\theta)\frac{(l-m)!}{(l+m)!}l^{2m}\frac{1}{\pi^2 \sin(\theta)}\cos^2\left(\left(l+\frac{1}{2}\right)\theta-\frac{\pi}{4}+\frac{m\pi}{2}\right)\sin(\theta)\d\theta+\mathrm{O}\left(l^{-1}\right) \\
&=\frac{2}{\pi}\frac{(l-m)!}{(l+m)!}l^{2m}\int_{\varepsilon}^{\pi-\varepsilon}f(\theta)\cos^2\left(\left(l+\frac{1}{2}\right)\theta-\frac{\pi}{4}+\frac{m\pi}{2}\right)\d\theta+\mathrm{O}\left(l^{-1}\right).\label{eq:asymp67}
\end{split}
\end{align}
The oscillatory integral in (\ref{eq:asymp67}) has the limit
\begin{equation}
\lim_{l\to \infty}\int_{\varepsilon}^{\pi-\varepsilon}f(\theta)\cos^2\left(\left(l+\frac{1}{2}\right)\theta-\frac{\pi}{4}+\frac{m\pi}{2}\right)\d\theta
= \lim_{l\to \infty}\int_{\varepsilon}^{\pi-\varepsilon}f(\theta)\cos^2(l\theta)\d\theta=\frac{1}{2}\int_{\varepsilon}^{\pi-\varepsilon}f(\theta)\d\theta, \label{eq:oszint}
\end{equation}
where the final equality is true because $\lim_{l\to \infty}\int_{\varepsilon}^{\pi-\varepsilon}f(\theta)\cos^2(l\theta)\d\theta=\lim_{l\to \infty}\int_{\varepsilon}^{\pi-\varepsilon}f(\theta)\sin^2(l\theta)\d\theta$ and $\sin^2+\cos^2=1$. Using (\ref{eq:oszint}) and (\ref{eq:lasymp}) we conclude from (\ref{eq:asymp67}) for each small $\varepsilon>0$ that 
\begin{equation}
\lim_{l\to \infty}2\pi\int_{\varepsilon}^{\pi-\varepsilon}f(\theta)\widetilde{|Y_{l,m}|}(\theta)^2\sin(\theta) \d\theta=\frac{1}{\pi}\int_{\varepsilon}^{\pi-\varepsilon}f(\theta)\d\theta. \label{eq:asymp34}
\end{equation}
Noting that $\limsup_{x\to\infty} \cos^2(x)\leq 1$ and $\liminf_{x\to\infty} \cos^2(x)\leq 1$ exist, the $\varepsilon=0$ version of (\ref{eq:asymp34}) now follows from (\ref{eq:asymp66}) and (\ref{eq:asymp34}) using Fatou's Lemma. For the details of this, see \cite{kuester}.
\end{proof}

\begin{rem}\label{rem:last1}We do not know whether for the standard 2-sphere  Theorem \ref{thm:equishnirelman3} is directly implied by the classical theory  of Legendre polynomials. Moreover, it is crucial that $m$ grows slower than $l$ as $l\to \infty$. Indeed, if one considers the diagonal sequence $Y_{l,l}$ of zonal spherical harmonics, it is not difficult to see that, contrasting with our results, they concentrate along the equator in  $S^2$ as $l \to \infty$ in the sense that for a given $\epsilon>0$ there is a constant $c(\epsilon)>0$ such that 
\bqn 
\intop_{S^2-B_\epsilon} |Y_{l,l}|^2 \d S^2=\mathrm{O}(e^{-c(\epsilon) l}),
\eqn 
where $B_\epsilon$ denotes the tubular neighborhood of the equator of width $\epsilon$,  compare \cite{colindv} and Figure \ref{fig:zonal}, yielding qualitatively quite different limit measures. 
\end{rem}

%\qquad ********************** Endgueltige Revision bis hier *************************\\ 

\providecommand{\bysame}{\leavevmode\hbox to3em{\hrulefill}\thinspace}
\providecommand{\MR}{\relax\ifhmode\unskip\space\fi MR }
\providecommand{\MRhref}[2]{%
  \href{http://www.ams.org/mathscinet-getitem?mr=#1}{#2}
}
\providecommand{\href}[2]{#2}

%\bibliography{bibliography}
%\bibliographystyle{amsplain}

\end{document}